\documentclass{article}

\usepackage{tikz}
\usetikzlibrary{shapes, decorations.pathreplacing, patterns}
\usetikzlibrary{arrows.meta}

\usetikzlibrary{math}
\usepackage{hyperref}
\usepackage{algorithm}
\usepackage{algpseudocode}

\usepackage[normalem]{ulem}

\newcommand{\blambda}{\boldsymbol{\lambda}}
\newcommand{\domain}{\mathcal{D}}
\newtheorem{definition}{Definition}
\newtheorem{assumption}{Assumption}

\usepackage{amssymb,amsmath,psfrag,graphicx}
\usepackage{equationarray}
\usepackage{verbatim}
\usepackage{color}
\usepackage{morefloats}
\usepackage{colortbl}
\usepackage{fullpage}
\usepackage{enumitem}
\usepackage{tabularx}
\usepackage{marginnote}

\newtheorem{theorem}{Theorem} 
\newtheorem{lemma}[theorem]{Lemma} 
\newtheorem{corollary}[theorem]{Corollary} 
\newtheorem{remark}{Remark} 

\newenvironment{proof}{\paragraph{Proof.}}{\hfill$\square$}

\newcommand{\pv}{\boldsymbol{p}}

\newcommand{\uv}{\boldsymbol{u}}
\newcommand{\wv}{\boldsymbol{w}}

\newcommand{\fbv}{\boldsymbol{f}}

\newcommand{\fv}{\boldsymbol{f}}
\newcommand{\zerov}{\boldsymbol{0}}
\newcommand{\lbv}{\boldsymbol{\lambda}}
\newcommand{\lv}{\boldsymbol{\lambda}}
\newcommand{\zetav}{\boldsymbol{\zeta}}
\newcommand{\zetabv}{\boldsymbol{\zeta}}

\newcommand{\etav}{\boldsymbol{\eta}}

\newcommand{\chiv}{\boldsymbol{\chi}}
\newcommand{\chibv}{\boldsymbol{\chi}}

\newcommand{\phibv}{\underline{\boldsymbol\phi}}

\newcommand{\psibv}{\boldsymbol{\psi}}
\definecolor{violachiaro}{rgb}{.8,0.69,1}

\newcolumntype{L}{>{$}l<{$}}
\newcolumntype{C}{>{$}c<{$}}
\newcolumntype{R}{>{$}r<{$}}

\definecolor{yel}{rgb}{1,1,0.554}

\usepackage{mathrsfs}

\title{Algebraic convergence analysis for the Interface Control Domain Decomposition (ICDD) method}
\author{Marco Discacciati, Paola Gervasio, Alfio Quarteroni}
\date{\today}

\begin{document}

\maketitle
\begin{abstract}
We develop the convergence analysis of the Interface Control Domain Decomposition (ICDD) method, an overlapping domain decomposition method based on an optimal control framework with Dirichlet interface control functions and interface observation. We consider elliptic problems with possible discontinuous coefficients approximated by $hp-$FEM in each subdomain. When the discretizations are conforming on the overlap between 2D domains, we provide theoretical estimates of the number of GMRES iterations needed to solve the non-symmetric interface Schur complement system associated with ICDD. Our results are obtained by combining novel spectral estimates for the Schur complement matrix of ICDD and classical GMRES convergence theory. We prove that the convergence rate behaves as $\mathcal{O}(\delta^{-1} p^{3/2}\log p)$, where $\delta$ denotes the overlap width and $p$ the local polynomial degree, while remaining independent of the mesh size $h$. Numerical experiments verify the theoretical predictions and show the effectiveness of ICDD in the presence of large coefficient jumps in the computational domain. Since in the conforming case, the considered ICDD formulation coincides with the Substructured Restricted Additive Schwarz (SRAS) method, the analysis also provides convergence estimates for SRAS in two dimensions and, through its known equivalence, for the Restricted Additive Schwarz (RAS) method.
\end{abstract}

\medskip


\medskip


\section{Introduction}
The efficient numerical approximation of problems governed by partial differential equations is key in many areas of science and engineering. This has urged the development of a variety of decomposition methods that split a given global discretized problem (typically, a large-dimensional algebraic system) into a family of smaller subproblems. Classical techniques used in this context are non-overlapping and overlapping Schwarz methods, as presented in \cite{sbg,toselli-widlund,dolean-jolivet-nataf} and references therein.

Non-overlapping techniques rely on ad-hoc coupling conditions and are usually formulated in terms of interface variables alone. On the other hand, overlapping Schwarz methods normally involve volume degrees of freedom and exploit simpler coupling conditions (e.g., the continuity of the primal unknown). Although the possibility of working with interface variables in the context of overlapping methods was already outlined in \cite{sbg}, it has only recently been considered in the Substructured Restricted Additive Schwarz (SRAS) method \cite{SRAS2022}. 

A more flexible and general overlapping approach than SRAS  to couple both homogeneous and heterogeneous problems had been proposed in \cite{Glowinski:1983:DDM, Lions:1998:APS, lions:m2an:2000, glq, agq}, as an optimal control problem on interface traces, named virtual controls. This idea was then further extended, yielding the \emph{Interface Control Domain Decomposition} (ICDD) method \cite{dgq_ell1,dgq_ijnmf,dgq_ss,dggq,dg_jcp}, an overlapping domain decomposition method based on an optimal control framework with cost functional and control functions acting on the internal boundaries of the local subdomains. Different from approaches based on Lagrange multipliers, ICDD does not need to introduce any additional degrees of freedom. Instead, it provides a flexible coupling framework that uses standard coupling conditions of Dirichlet or Neumann type and exploits a small overlapping region between subdomains. In particular, in this work, we consider a version of ICDD where control variables play the role of the (unknown) traces of the solution at the interfaces (the internal boundaries of the subdomains), and the observation is on the interfaces themselves \cite{dgq_ell1}.

The well-posedness analysis of the method was carried out in \cite{dgq_ell1} using optimal control theory, and the method has been successfully applied to various linear partial differential equations of both homogeneous and heterogeneous types, also including reduced-order models \cite{dgq_ell1,dgq_ijnmf,dgq_ss,dggq,dg_jcp,discacciati-giacomini-0,discacciati-giacomini-1,discacciati-giacomini-2}.

From a computational point of view, the optimal control problem is tackled by solving the associated optimality system, which includes a couple of PDEs and a couple of equations on the interfaces. The optimality system of ICDD was originally composed of two couples of PDEs: two primal problems and two dual problems, plus two equations on the interfaces \cite{dgq_ell1,dgq_ijnmf,dggq}. When the overlap width is really tiny, solving the optimality system of ICDD with the dual problems performs better than solving the optimality system without dual problems, but it is, in general, more expensive, so we omit them in the formulation considered in this work. At the algebraic level, we solve the optimality system by constructing its Schur complement with respect to the interface degrees of freedom, that is, the Dirichlet controls associated with the optimal control problem. As information between the local subproblems is only exchanged through the trace of the state solution on the internal interfaces, the local solvers can be treated as ``black boxes''. This makes the implementation of the method fully non-intrusive, reducing communication costs among computing units on parallel architectures.

\medskip

In this paper, we analyze the convergence rate of ICDD iterations in the case of Dirichlet controls and interface observation for a second-order elliptic operator with possibly discontinuous coefficients in two subdomains. We aim to provide convergence estimates with respect to the overlap width $\delta$ and the discretization parameters (mesh size $h$ and local polynomial degree $p$ of an $hp$-FEM formulation). The Schur complement system of ICDD is non-symmetric; hence, we solve it by GMRES iterations. Because of the lack of symmetry, the condition number is not a good indicator of the convergence rate of the method. Instead, we can apply the theory of \cite{eisenstat-elman-schultz,saad_schultz} to estimate the minimum number of GMRES iterations guaranteeing that the relative residual is less than a given tolerance. In particular, we prove that, when the discretizations inside the two subdomains are conforming on the overlap, the convergence rate of the ICDD method is $\mathcal O(\delta^{-1}p^{3/2}\log p )$ and is independent of the mesh-size $h$. 
However, numerical results also show that the convergence rate is in fact independent of $p$.

Moreover, when the PDE features highly jumping coefficients across the interface, the convergence rate of ICDD tends to lose the dependence on the overlap width $\delta$: the larger the jump, the lower the number of iterations and the weaker the dependence on $\delta$.

It must be observed that, when the discretization is conforming on the overlap, the ICDD method considered in this work shares some similarities with SRAS \cite{SRAS2022}, for which a convergence analysis has not been established yet to the best of our knowledge. In \cite{SRAS2022}, it was proved that SRAS is equivalent to the Restricted Additive Schwarz (RAS) method proposed in \cite{Cai-Sarkis1999}, although the associated linear systems have significantly different dimensions. In fact, while RAS looks for volume degrees of freedom, SRAS is limited to the interface ones. The convergence analysis of RAS has been carried out in \cite{frommer_szyld} and \cite{sarkis-dryja}. In \cite{frommer_szyld}, the authors proved convergence for M-matrices using weighted max norms, without making explicit the dependence on the overlap width $\delta$. In \cite{sarkis-dryja}, provided that the system matrix corresponds to the discretization of the Laplace equation in 1D domains by second-order centered finite differences, the authors proved that the convergence rate of RAS with coarse level behaves like $(1+H/h)$, where $H$ is the subdomain size and $h$ is the mesh size. However, the analysis with discontinuous coefficients, more general discretizations, and 2D and 3D domains is still an open problem. 

In those cases where ICDD and SRAS coincide, the theory developed here advances the state of the art by providing a convergence rate for the latter method as well. Given the equivalence between SRAS and RAS discussed above, our results also apply to the RAS method.

The paper is organized as follows. In Section \ref{sec:settings}, we review the ICDD method, its $hp$-FEM discretization, and algebraic formulation. Moreover, we compare ICDD with other overlapping methods, like Additive Schwarz, RAS, and SRAS. In Section \ref{sec:convergence-analysis}, we analyse the spectral properties of the interface Schur complement matrices associated with ICDD, and in Sect.~\ref{sec:convergenceGMRES}, we use these results to estimate the convergence rate of ICDD. Numerical results to validate the theory are presented in Sects.~\ref{sec:numericalValidation} and \ref{sec:numerical_results}. Some technical results and pseudo-algorithms to practically implement the ICDD method are reported in the Appendix \ref{sec:appendix}.

\section{The ICDD coupling method}\label{sec:settings}

In this section, we review the formulation of the ICDD method considering the case of two subdomains for simplicity of exposition. We introduce the optimality system associated with the method, and its algebraic form highlighting the role of interface control variables.

\subsection{Formulation and optimality system}

Let $\Omega \subset \mathbb{R}^d$ ($d=1,2,3$) be an open bounded domain with Lipschitz boundary, where we consider the boundary value problem: find $u$ in $\Omega$ such that
\begin{equation}\label{eq:globalProblem}
\begin{array}{rl}
Lu = f & \quad \text{in } \Omega, \\
u=0 & \quad \text{on } \partial \Omega.
\end{array}
\end{equation}
Here, $f \in L^2(\Omega)$ is a given function, while $L$ is a second-order linear elliptic coercive operator so that \eqref{eq:globalProblem} is well posed. For example, we can consider
\begin{equation}\label{eq:operatorL}
Lu = - \nabla \cdot (\nu \nabla u) + \gamma u \,,
\end{equation}
with $\nu,\gamma \in L^\infty (\Omega)$ and such that $\exists \underline{\nu} > 0$ such that $\nu(\mathbf{x}) \geq \underline{\nu}$ a.e. in $\Omega$, and $\gamma(\mathbf{x}) \geq 0$ a.e. in $\Omega$.

We split $\Omega$ into two non-overlapping subdomains $\widetilde{\Omega}_1$ and $\widetilde{\Omega}_2$ where the coefficients $\nu$ and $\gamma$ may assume different values with a possible discontinuity at the common boundary $\partial \widetilde{\Omega}_1 \cap \partial \widetilde{\Omega}_2$:
\begin{equation*}
\nu =
\left\{
\begin{array}{ll}
\widetilde{\nu}_1 & \quad \text{in } \widetilde{\Omega}_1 \\
\widetilde{\nu}_2 & \quad \text{in } \widetilde{\Omega}_2 \\
\end{array}
\right.
\quad \text{and} \quad
\gamma =
\left\{
\begin{array}{ll}
\widetilde{\gamma}_1 & \quad \text{in } \widetilde{\Omega}_1 \\
\widetilde{\gamma}_2 & \quad \text{in } \widetilde{\Omega}_2. \\
\end{array}
\right.
\end{equation*}
For $k,\ell=1,2$ and $k \not= \ell$, we now extend the local subdomain $\widetilde{\Omega}_k$ into $\widetilde{\Omega}_\ell$ to form two overlapping subdomains $\Omega_1$ and $\Omega_2$ such that $\overline{\Omega} = \overline{\Omega_1 \cup \Omega_2}$. For $k=1,2$, $\Gamma_k = \partial \Omega_k \setminus \partial \Omega$ are two internal boundaries that we call \textit{interfaces}, and we assume that $\exists \underline{\delta}>0$ such that $\delta:=dist(\Gamma_1,\Gamma_2)\geq \underline{\delta}>0$. (This implies that $(\partial \Omega_k \setminus \Gamma_k )\subset \partial \Omega$, and that $\Gamma_1 \cap \Gamma_2 = \emptyset$.) An example of splitting in a simple cartesian setting in shown in Fig. \ref{fig:domain}.

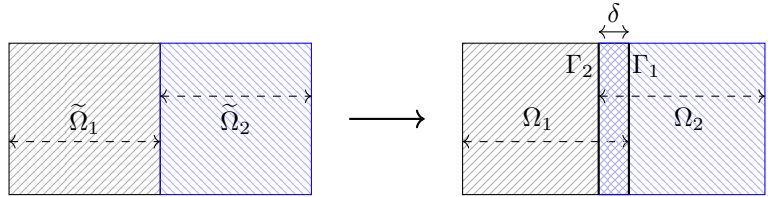
\begin{figure}[h]
\begin{center}
\begin{tikzpicture}
	\draw[black,pattern=north east lines, pattern color=black!30] (0, 0) rectangle (2, 2);
    \node[] at (1.0, 1.0) {$\widetilde{\Omega}_1$};
    \draw[<->, dashed]  (0, 0.7) -- (2, 0.7);
    \draw[blue,pattern=north west lines, pattern color=blue!30] (2, 0) rectangle (4, 2);
    \node[] at (3.0, 1.0) {$\widetilde{\Omega}_2$};
    \draw[<->, dashed]  (2, 1.3) -- (4, 1.3);
    
    \draw[->, thick]  (4.5, 1.0) -- (5.5, 1.0);
    
	\draw[black, pattern=north east lines, pattern color=black!30] (6, 0) rectangle (8.2, 2);
    \node[] at (7.0, 1.0) {$\Omega_1$};
    \draw[<->, dashed]  (6, 0.7) -- (8.2, 0.7);
    \draw[blue,pattern=north west lines, pattern color=blue!30] (7.8, 0) rectangle (10, 2);
    \node[] at (9.0, 1.0) {$\Omega_2$};
    \draw[<->, dashed]  (7.8, 1.3) -- (10, 1.3);
    \draw[<->]  (7.8, 2.15) -- (8.2, 2.15);
    \node[] at (8.0, 2.4) {$\delta$};
    \draw[thick]  (8.2, 0) -- (8.2, 2);
    \node[] at (8.43, 1.7) {$\Gamma_1$};
    \draw[thick]  (7.8, 0) -- (7.8, 2);
    \node[] at (7.55, 1.7) {$\Gamma_2$};
\end{tikzpicture}
\end{center}
\caption{Non-overlapping (left) and overlapping (right) splitting of the domain $\Omega$.}
\label{fig:domain}
\end{figure}

In each local subdomain $\Omega_k$ ($k=1,2$), we consider the local linear elliptic operator $L_k$
\begin{equation*}
L_ku = -\nabla \cdot (\nu_k \nabla u) + \gamma_k u \, ,
\end{equation*}
where
\begin{equation}\label{eq:localCoefficients}
\nu_k=\nu|_{\Omega_k} \quad \text{ and }\quad \gamma_k=\gamma|_{\Omega_k}.
\end{equation}

We can now introduce the multidomain problem:
find $u_1$ in $\Omega_1$ and $u_2$ in $\Omega_2$ such that
\begin{equation}\label{eq:problem_2dom_gamma}
\left\{\begin{array}{rcll}
L_ku_k &=& f_k & \mbox{in } \Omega_k, \; k=1,2,\\
u_k    &=& 0   & \mbox{on } \partial\Omega_k\setminus \Gamma_k, \; k=1,2,\\
u_1    &=& u_2 & \mbox{on } \Gamma_1 \cup  \Gamma_2,
\end{array}\right.
\end{equation}
where $f_k = f_{\vert\Omega_k} \in L^2(\Omega_k)$. Problem \eqref{eq:problem_2dom_gamma} has a unique solution with $u_k = u_{\vert\Omega_k}$, $u$ being the unique solution of \eqref{eq:globalProblem} (see, e.g. \cite{dgq_ell1}).

\begin{remark}
The local coefficients $\nu_k$ and $\gamma_k$ defined in \eqref{eq:localCoefficients} may be discontinuous in $\Omega_k$. However, the way they are defined is crucial to ensure the equivalence between \eqref{eq:globalProblem} and \eqref{eq:problem_2dom_gamma}. In particular, this equivalence would not hold if, to avoid internal discontinuities, $\nu_k$ and $\gamma_k$ were defined as continuous extensions in $\Omega_k \cap \widetilde{\Omega}_i$ of the coefficients $\widetilde{\nu}_j$ and $\widetilde{\gamma}_j$ ($j \not= i$).
\end{remark}

The multidomain problem (\ref{eq:problem_2dom_gamma}) can be reformulated as an \emph{optimal control problem} (see \cite{lions-controllo}) as follows.
\begin{enumerate}[noitemsep]
\item Introduce two unknown interface functions $\lambda_1$ on $\Gamma_1$ and $\lambda_2$ on $\Gamma_2$.
\item For $k=1,2$, define the \emph{state} problems
\begin{eqnarray}\label{eq:state_eq}
\left\{\begin{array}{rcll}
L_ku_k &=& f_k &\mbox{ in } \Omega_k, \\
u_k &=& \lambda_k &\mbox{ on } \Gamma_k, \\
u_k &=& 0 &\mbox{ on } \partial\Omega_k\setminus\Gamma_k
\end{array}\right.
\end{eqnarray}
where $u_k=u_k(\lambda_k)$ are the \textit{state} solutions.
\item Define the \textit{cost functional}
\begin{equation}\label{eq:JGamma}
J_\Gamma(\lambda_1,\lambda_2)=\frac{1}{2} \| u_1(\lambda_1) - u_2(\lambda_2) \|^2_{L^2(\Gamma_1\cup\Gamma_2)}.
\end{equation}
\item The interface functions $\lambda_1$ and $\lambda_2$ are the unique solution of the minimization problem
\begin{equation}\label{eq:infJGamma}
 \inf_{\mu_1,\mu_2}J_\Gamma(\mu_1,\mu_2).
\end{equation}
\end{enumerate}

Following \cite{Lions:1998:APS,glq}, the interface functions $\lambda_1$ and $\lambda_2$ are called \textit{virtual controls}, and, using the terminology of Optimal Control Theory, they are boundary-type controls or, more precisely, \emph{interface controls} that play the role of Dirichlet boundary data for the state problems (\ref{eq:state_eq}).

The approach involving steps 1--4 above was proposed and studied in \cite{dgq_ell1}, and it was given the name \emph{Interface Control Domain Decomposition (ICDD) method}. The cost functional \eqref{eq:JGamma} implements an \emph{interface observation} on the interfaces $\Gamma_1$ and $\Gamma_2$, and the minimization problem \eqref{eq:infJGamma} is the weak counterpart of the multidomain problem \eqref{eq:problem_2dom_gamma}. In \cite[Theorem 4.2]{dgq_ell1}, it was proved that the minimization problem \eqref{eq:infJGamma} admits a unique solution. 

\medskip

We now formulate the optimality system associated with problems \eqref{eq:state_eq}--\eqref{eq:infJGamma}. To this aim, for $k=1,2$, we first define the Hilbert spaces
\begin{eqnarray*}
V_k = \{ v \in H^1(\Omega_k) \, : \, v_k = 0 \mbox{ on } \partial\Omega_k
\setminus\Gamma_k \}, \quad
V_k^0 = \{ v \in H^1(\Omega_k) \, : \, v_k = 0 \mbox{ on } \partial\Omega_k \},
\end{eqnarray*} 
the trace spaces $\Lambda_k = H^{1/2}_{00}(\Gamma_k)$ of the functions defined on $\Gamma_k$ whose trivial prolongation to $\partial\Omega_k$ belongs to $H^{1/2}(\partial\Omega_k)$, and let $\boldsymbol{\Lambda} = \Lambda_1 \times \Lambda_2$.
The spaces $V_k$ and $\Lambda_k$ are endowed with the canonical norm of the space $H^1$ and $H^{1/2}$, respectively \cite{qv_ddm}.

In \cite[Sect. 4]{dgq_ell1}, we showed that the optimal control problem \eqref{eq:state_eq}--\eqref{eq:infJGamma} is not well-posed in $\boldsymbol{\Lambda}$, but in the space $\widehat{\boldsymbol{\Lambda}} = \widehat{\Lambda}_1 \times \widehat{\Lambda}_2$, which is the completion of the trace space $\boldsymbol{\Lambda}$ with respect to the norm induced by the cost functional $J_\Gamma$. This is due to the fact that the space $\boldsymbol{\Lambda}$ could not be complete with respect to the norm induced by the cost functional $J_\Gamma$, because the norm $\| u_1(\lambda_1) - u_2(\lambda_2) \|_{L^2(\Gamma_1 \cup \Gamma_2)}$ is not equivalent to the canonical norm $H_{00}^{1/2}(\Gamma_1) \times H_{00}^{1/2}(\Gamma_2)$ on $\boldsymbol{\Lambda}$. Thus, considering now the trace spaces $\widehat{\Lambda}_k$, we introduce two linear and continuous lifting operators from the interfaces to the subdomains:
\begin{equation}\label{eq:extension}
E_k:\widehat\Lambda_k\to V_k:\quad E_k\lambda_k=\lambda_k \mbox{ on }\Gamma_k
\qquad k=1,2,
\end{equation}
and, for all $u_k,v_k\in V_k$, we define the bilinear forms and the functionals
\begin{equation*}
\displaystyle a_k (u_k,v_k) =\int_{\Omega_k}\nu_k\nabla u_k \cdot \nabla v_k +\int_{\Omega_k}\gamma_k u_k v_k, \qquad
\displaystyle F_k(v_k) = \int_{\Omega_k} fv_k.
\end{equation*}

\smallskip

In \cite[Theorem 4.3]{dgq_ell1}, it was proved that (\ref{eq:state_eq})--(\ref{eq:infJGamma}) is equivalent to the optimality system: for $k=1,2$ and $\ell=3-k$, find $u_k^0\in V_k^0$, $p_k^0\in V_k^0$, and $\lambda_k\in\widehat{\Lambda}_k$ such that
\begin{eqnarray}\label{eq:os2_with_dual}
\left\{\begin{array}{ll}
a_k(u_k^0,v_k)+a_k(E_k\lambda_k,v_k)=F_k(v_k) 
 & \forall v_k \in V_k^0,\\[1mm]
u_k=u_k^0+E_k\lambda_k, & \\[1mm]
a_k(p_k^0,v_k)+a_k(E_k(u_k-u_{\ell}),v_k)=0 
 & \forall v_k \in V_k^0,\\[1mm]
p_k=p_k^0+E_k(u_k-u_{\ell}), & \\[1mm]
\lambda_k=u_\ell-p_\ell & \mbox{ on }\Gamma_k.
\end{array}\right.
\end{eqnarray}
The problems (\ref{eq:os2_with_dual})$_{3,4}$ are called the \emph{dual state problems}.

\smallskip

However, by using the same technique of the proof of Theorem 4.3 of \cite{dgq_ell1}, it is easy to prove that (\ref{eq:state_eq})--(\ref{eq:infJGamma}) is also equivalent to the following optimality system without the dual state problems: for $k=1,2$ find $u_k^0\in V_k^0$ and $\lambda_k\in \widehat{\Lambda}_k$ such that
\begin{eqnarray}\label{eq:os2}
\left\{\begin{array}{ll}
a_k(u_k^0,v_k)+a_k(E_k\lambda_k,v_k)=F_k(v_k) 
 & \forall v_k \in V_k^0,\\[1mm]
u_k=u_k^0+E_k\lambda_k, & \\[1mm]
\displaystyle 
\lambda_k=u_\ell & \mbox{ on }\Gamma_k.
\end{array}\right.
\end{eqnarray}

\medskip

While systems \eqref{eq:os2_with_dual} and \eqref{eq:os2} are equivalent at the continuous level, they give rise to two different numerical methods that we will denote \textit{dual ICCD method} and \textit{ICDD method}, respectively, in the rest of the paper. As we will show in Sect.~\ref{sec:numerical_results}, the dual ICDD method improves the stability of the ICDD method when the overlap width $\delta$ is very thin, and solving the dual state problems (\ref{eq:os2_with_dual})$_{3,4}$ can be interpreted as a preconditioning step. However, for the time being, we focus on the ICCD method (\ref{eq:os2}) and its analysis.

Since, the definition of the optimality systems (\ref{eq:os2_with_dual}) and (\ref{eq:os2}) involves the trace on $\Gamma_k$ of the solution $u_\ell$, which depends on $\lambda_\ell$, in the next Section, we estimate the $L^2-$norm of $u_\ell$ in terms of the $L^2-$norm of the trace $\lambda_\ell$. These estimates will be used in Sect. \ref{sec:convergence-analysis} to prove the convergence rate of ICDD.


\subsection{$L^2-$ estimates of the internal trace of the harmonic extension}\label{sec:estimatesOnTraces}

In this section, we consider an open bounded domain $\domain \subset \mathbb{R}^2$ with Lipschitz boundary $\partial\domain$, a function $\lambda$ defined on a subset $\Gamma \subset \partial\domain$, and the \emph{harmonic extension} $u^\lambda\in H^1_{0,\partial\domain\setminus\Gamma}(\domain)$ of $\lambda$, i.e., the solution of $-\nabla\cdot(\nu\nabla u^\lambda)+\gamma u^\lambda=0$ in $\domain$, with $u^\lambda=\lambda$ on $\Gamma$ and $u^\lambda=0$ on $\partial \domain\setminus \Gamma$. Then, we estimate the $L^2-$norm of the trace of $u^\lambda$ on a set $\gamma_t$ of co-dimension 1 internal to the domain $\domain$. A sample schematic representation of the setting that we consider is given in Fig.~\ref{fig:gammat}. (In the following, $\domain$ will play the role of either subdomain $\Omega_k$, $k=1,2$.)

\begin{figure}[bht]
\begin{center}
\includegraphics[width=0.4\textwidth]{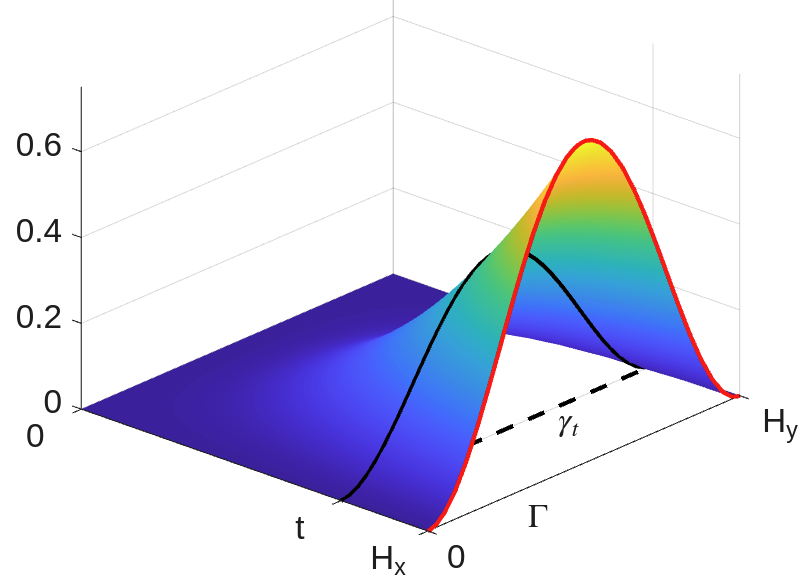}
\end{center}
\caption{The harmonic extension $u^\lambda$ of $\lambda$ (in red) in the domain $\domain = (0,H_x) \times (0,H_y)$. The set $\gamma_t$ and the trace of $u^\lambda$ on $\gamma_t$ are highlighted.}
\label{fig:gammat}
\end{figure}

We can prove the following result, which is a consequence of the more general Lemma 3.6 of \cite{loisel_szyld_nummath}. Indeed, although the latter deals with more general domains and traces on a 1-dimensional variety than we consider here, the novelty of our result is the very sharp expression of the function $\mathcal{C}_M(t)$ appearing in (\ref{eq:ls}) with respect to the size $H$ of the domain, the operator coefficients $\nu$ and $\gamma$, and the variable coordinate $t$, which is related to the overlap width.

\begin{theorem}\label{thm:ls}
Let $\domain=(0,H_x)\times(0,H_y)$ be a rectangle, $\gamma_t=\{(x,y)\in\overline{\domain}:\ x=t\}$ and $\Gamma=\{H_x\}\times(0,H_y)$ (see Fig. \ref{fig:gammat}).
Given $\lambda\in H^{1/2}_{00}(\Gamma)$, let $u^\lambda\in H^1(\domain)$ be the harmonic extension of $\lambda$ in $\domain$.
Then, there exists a positive and bounded function $\mathcal{C}_M:(0,H_x)\to(0,1)$ such that, for any $ t\in(0,H_x)$, it holds
\begin{equation}\label{eq:ls}
\|u^\lambda\|_{L^2(\gamma_t)}\leq \mathcal{C}_M(t) \|\lambda\|_{L^2(\Gamma)}.
\end{equation}

More precisely, if $H=diam(\domain)=\max\{H_x,H_y\}$ and $H_t=\max\{t,H_y\}$, then
\begin{equation}\label{eq:CLS_behaviour}
\mathcal{C}_M(t)=\left(1-\frac{1}{\frac{\|\nu\|_\infty}{\underline \nu} + \mathcal{C}(t)}\right)^{1/2}<1,
\end{equation}
and
\begin{equation}\label{eq:barCt}
\mathcal{C}(t) = \frac{c}{2} \, \frac{\|\nu\|_\infty}{c_1} \, \frac{\max\{H_t,H_t^{-1}\}}{H_x-t}\,,
\end{equation}
where the constant $c>0$ only depends on the shape of $\domain$, while $c_1>0$ is defined as 
\begin{equation}\label{eq:c1}
c_1 =
\left\{
\begin{array}{ll}
4\underline\nu/(4+2t H_x+(H_x-t)^2) & \text{if } \underline\gamma=0 ,\\
\min\{\underline\nu,\underline\gamma\} & \text{if } \underline\gamma \not= 0 ,
\end{array}
\right.
\end{equation}
with $\underline\gamma=\inf_{{\bf x}\in\Omega} \gamma({\bf x})$ and $\underline\nu=\inf_{{\bf x}\in\Omega} \nu({\bf x})$.

\end{theorem}

\begin{proof}
See Appendix \ref{appendix:proof}.
\end{proof}

\medskip

We now introduce the following Assumption.

\begin{assumption}\label{ass_decomposition}
For $k=1,2$, let $\Omega_k \subset \mathbb{R}^2$ be two rectangles of the same size $H_x \times H_y$ and $H = \max \{ H_x, H_y\}$ be their diameter. The rectangles are overlapped as shown in Fig. \ref{fig:domain} (right) such that the interfaces $\Gamma_1$ and $\Gamma_2$ are vertical parallel segments. The distance $\delta = \text{dist}(\Gamma_1, \Gamma_2)>0$ is the width of the overlap and we assume that $\delta \ll H_x$.
\end{assumption}

\begin{corollary}\label{cor:gamma_gamma}
Let Assumptions \ref{ass_decomposition} hold. For $k=1,2$, let $\lambda_k\in \Lambda_k$ and let $u_k^{\lambda_k}\in V_k$ be its harmonic extension in $\Omega_k$.
Then, there exists a positive constant $c_{L,k}>0$ that only depends on the physical quantities that characterize the local differential operator $L_k$, and there exists a positive constant $c_{H,k}>0$ that only depends on the geometrical properties of the local domain $\Omega_k$ such that, for $k=1,2$ and $\ell=3-k$,
\begin{equation}\label{eq:stimal2_gg}
\|u_k^{\lambda_k}\|_{L^2(\Gamma_\ell)}\leq
C_{\delta,k} \|\lambda_k\|_{L^2(\Gamma_k)}
\quad
\text{with} \quad
C_{\delta,k}=\left(\frac{\delta\left(\frac{\|\nu_k\|_\infty}{\underline\nu_k}-1\right)+c_{L,k} \, c_{H,k}}{\delta\frac{\|\nu_k\|_\infty}{\underline\nu_k}+c_{L,k} \, c_{H,k} } \right)^{1/2},
\end{equation}
and where both constants $c_{L,k}$ and $c_{H,k}$ are independent of $\delta$.
\end{corollary}

\begin{proof}
The thesis follows from Theorem \ref{thm:ls}. More precisely, for $k=1,2$, let $\domain = \Omega_k$ and $\delta=\text{dist}(\Gamma_1,\Gamma_2)=H_x-t$ in \eqref{eq:barCt}. Due to Assumption \ref{ass_decomposition}, $\delta \ll H_x$, so $H_t = \max \{ t, H_y \} = \max\{ H_x - \delta, H_y \} \approx \max \{ H_x, H_y \} = H$. Therefore, we can approximate $\mathcal{C}(t)$ with $C_k(t)$ where
\begin{equation*}
C_k(t) = \frac{c}{2} \, \frac{\|\nu_k\|_\infty}{c_1} \, \frac{\max\{H,H^{-1}\}}{\delta}\,.
\end{equation*}
Finally, considering \eqref{eq:c1}, we can define the constants $c_{L,k}$ and $c_{H,k}$ as follows:
\begin{equation*}
c_{L,k} = \left\{
\begin{array}{ll}
\|\nu_k\|_\infty \, (8 \underline{\nu}_k)^{-1} & \text{if } \underline{\gamma}_k = 0, \\
\|\nu_k\|_\infty \, (2\min \{ \underline{\nu}_k, \underline{\gamma}_k \})^{-1} & \text{if } \underline{\gamma}_k \not= 0 ,
\end{array}
\right.
\end{equation*}
and
\begin{equation*}
c_{H,k} = \left\{
\begin{array}{ll}
c\, \max\{H, H^{-1}\} \, (4+2H^2)
& \text{if } \underline{\gamma}_k = 0, \\
c\, \max\{H, H^{-1}\} & \text{if } \underline{\gamma}_k \not= 0,
\end{array}
\right.
\end{equation*}
so that we can express $C_k(t) = c_{L,k} \, c_{H,k} \, \delta^{-1}$. The thesis now follows from \eqref{eq:ls} and \eqref{eq:CLS_behaviour} by the approximation $\mathcal{C}(t) \approx C_k(t)$. 

\end{proof}

\subsection{Discretization and algebraic form}\label{sec:discrete}

With the aim of handling the two subproblems independently of each other, we now introduce independent $hp$-FEM discretizations in each subdomain. This approach is different from the standard one used in the overlapping domain decomposition literature, where a discretization of the global domain $\Omega$ induces those in the local subdomains $\Omega_1$ and $\Omega_2$. Here, the meshes are constructed at the local level, thus leading to two independent set of nodes in $\Omega_1$ and $\Omega_2$ so that, by construction, duplication of nodes occurs inside the overlapping region $\Omega_1 \cap \Omega_2$. This approach leads to completely non-intrusive algorithms that permit to treat each subproblem in an independent way. 

More precisely, in general, we consider two \emph{independent families of triangulations} ${\cal T}_{1,h_1}$  in $\Omega_1$ and ${\cal T}_{2,h_2}$ in $\Omega_2$, so that the local meshes in $\Omega_1$ and  in $\Omega_2$ can be non-conforming on the overlap $\Omega_{1}\cap\Omega_2$, they can be characterized by different mesh sizes $h_1$ and $h_2$, and different polynomial degrees $p_1$ and $p_2$ can be used to define the local finite element spaces. In the rest of the paper, we consider the following Assumption \ref{ass_mesh}.

\begin{assumption}\label{ass_mesh}
    Inside each subdomain $\Omega_k$, if $T_m$ are simplices, we assume that the families of triangulations ${\cal T}_{k,h_k}$ are affine, regular \cite[(3.4.18)]{qvb} and quasi--uniform \cite[(3.5.19)]{qvb}. The interpolation nodes on the edges of each element $T_m$ are chosen as the Legendre-Gauss-Lobatto (LGL) nodes \cite{chqz06}, while the interior points (for $p_k\geq 3$) are the so-called electrostatic points, see \cite[Sect. 2.9.2]{chqz06} and \cite{Hesthaven_1998}.
    
    Instead, if $T_m$ are quads, we require that any $T_m$ is the image of the reference quad $\hat T=[-1,1]^d$ through a $C^1$ diffeomorfism ${\bf F}_{T_m}$ (see, e.g. \cite{qvb}).
    In this case, the interpolation nodes are the tensor product of the LGL nodes, as standard for Spectral Element Methods (SEM).
\end{assumption}

Choosing the nodes as indicated in Assumption \ref{ass_mesh} yields greater stability, since the Lebesgue interpolation constant remains smaller than with equally spaced nodes \cite{Hesthaven_1998}. Furthermore, in the case of quadrilaterals, since the interpolation nodes in each element are exactly the nodes of the LGL quadrature formulas, we can exploit these highly accurate quadrature formulas to assemble both stiffness and mass matrices at a very low computational cost compared with exact integration. Assumption \ref{ass_mesh} will also be needed to establish properties of the mass matrices as in Lemma \ref{lemma:eig_mass}.

\smallskip

The local finite element approximation spaces are (for $k=1,2$):
\begin{equation*}
  X_{k,h_k} = \{ v \in C^0(\overline{\Omega}_k) \, : \, v_{|T_m}
\in {\cal Q}_{p_k}, \: \forall T_m \in {\cal T}_{k,h_k} \},
\qquad
V_{k,h_k}=X_{k,h_k}\cap V_k,  \qquad V_{k,h_k}^0=X_{k,h_k}\cap V_k^0
\end{equation*}
(where ${\cal Q}_{p_k}={\mathbb P}_{p_k}$ if the $T_m$ are simplices and ${\cal Q}_{p_k}={\mathbb Q}_{p_k}\circ {\bf F}^{-1}_{T_m}$ if the $T_m$ are quads), while the spaces of the discrete traces on $\Gamma_1$ and $\Gamma_2$ are
\begin{equation*}
\Lambda_{k,h_k}=\{\lambda=v|_{\Gamma_k}, \ v\in V_{k,h_k}\}, \qquad 
\boldsymbol\Lambda_h=\Lambda_{1,h_1}\times \Lambda_{2,h_2}.
\end{equation*}
Let us remark that at the discrete level, due to the equivalence of all norms, we will be able to work with a standard finite element discretization of the trace space $\boldsymbol{\Lambda}$, without introducing any special treatment for $\widehat{\boldsymbol{\Lambda}}$.

\smallskip

We set $\overline N_k=\textrm{dim}(X_{k,h_k})$,
$N_k=\textrm{dim}(V_{k,h_k})$, $N^0_k=\textrm{dim}(V_{k,h_k}^0)$, and
$n_k=\textrm{dim}(\Lambda_{k,h_k})$. 

We consider (nodal) Lagrange basis functions in both the spaces $V_{k,h_k}$ and $\Lambda_{k,h_k}$.
Let $\{\mathbf x_i^{(k)}\}$ be the nodes of the mesh ${\cal T}_{k,h_k}$, then we define the following sets of indices:
\begin{eqnarray*}
\begin{array}{lll}
{\cal I}_{\overline \Omega_k}=\{1,\ldots, \overline N_k\},&
{\cal I}_{\Omega_k}=\{i\in{\cal I}_{\overline \Omega_k}:\ {\bf x}_i\in
(\Omega_k\setminus\partial\Omega)\},&
{\cal I}_{\Gamma_k}=\{i\in{\cal I}_{\overline \Omega_k}:\ {\bf x}_i\in
(\Gamma_k\setminus\partial\Omega)\}.
\end{array}
\end{eqnarray*}

For $i\in {\cal I}_{\overline \Omega_k}$, let $\{\varphi_i^{(k)}\}$ denote the Lagrange basis functions in $X_{k,h_k}$, while for $i\in {\cal I}_{\Gamma_k}$ let $\{\mu_i^{(k)}\}$ denote the Lagrange basis functions in $\Lambda_{k,h_k}$.

The discrete counterparts $E_{k,h_k}$ of the extension operators $E_k$ (\ref{eq:extension}) are defined as usual in finite element context: if $\mu_i^{(k)}$ is any Lagrange basis function of $\Lambda_{k,h_k}$, then $E_{k,h_k}\mu_i^{(k)}$ is the finite element interpolant that extends $\mu_i^{(k)}$ to $X_{k,h_k}$. It follows that there exists a unique $j=j(i)\in {\cal I}_{\Omega_k}$ such that $E_{k,h_k}\mu_i^{(k)}=\varphi_j^{(k)}$, that is, $\mu_i^{(k)}$ is the trace on $\Gamma_k$ of $\varphi_{j(i)}^{(k)}$.

The discrete counterpart of (\ref{eq:os2}) reads:
for $k=1,2$, find $u_{k,h_k}^0\in V_{k,h_k}^0$ and
$\lambda_{k,h_k}\in           
{\Lambda}_{k,h_k}$ such that
\begin{eqnarray}\label{eq:os2h}
\left\{\begin{array}{l}
a_k(u_{k,h_k}^0,v_{k,h_k})+a_k(E_{k,h_k}\lambda_{k,h_k},v_{k,h_k})=F_k(v_{k,h_k}) 
 \quad \forall v_{k,h_k} \in V_{k,h_k}^0,\\[1mm]
u_{k,h_k}=u_{k,h_k}^0+E_{k,h_k}\lambda_{k,h_k}, \\[1mm]
\lambda_{k,h_k}=\mathscr I_{\Gamma_k} ({u_{\ell,h_\ell}}_{|\Gamma_k}) \quad \mbox{on }\Gamma_k,
\end{array}\right.
\end{eqnarray}
where ${\mathscr I}_{\Gamma_k}: C^0(\Gamma_k)\to \Lambda_{k,h_k}$ is the composite Lagrange interpolation operator of local degree $p_k$ at the nodes
$\{\mathbf x_i^{(k)}\}_{i\in \mathcal I_{\Gamma_k}}$, i.e.,
\begin{equation}\label{eq:Lagrange_interpol}
    ({\mathscr I}_{\Gamma_k}v)(\mathbf x) = 
    \sum_{i\in{\mathcal I}_{\Gamma_k}}  v(\textbf{x}_i^{(k)}) \mu_i^{(k)}(\mathbf x) \qquad \forall \mathbf x\in \Gamma_k
\end{equation}
such that
\begin{equation}\label{eq:interpolation_conditions}
    ({\mathscr I}_{\Gamma_k}v)(\mathbf x_i^{(k)}) = v(\textbf{x}_i^{(k)})
    \qquad \forall i \in \mathcal I_{\Gamma_k}.
    \end{equation}

The solutions $u_{\ell,h_\ell}$ of (\ref{eq:os2h}) (for $\ell=1,2$) can be written as the sum of two contributions, one related to internal dofs $\uv_\ell^0$, and the other to interface dofs $\lv_\ell$:
\begin{equation}\label{eq:uh_expansion}
u_{\ell,h_\ell}(\textbf{x})=\sum_{j\in{\mathcal I}_{\Omega_\ell}}(\uv_\ell^0)_j\varphi^{(\ell)}_j(\textbf{x})
+ \sum_{j\in{\mathcal I}_{\Gamma_\ell}}(\lv_\ell)_j \varphi_j^{(\ell)}(\textbf{x}).
\end{equation}
By the linearity of the interpolation operator $\mathscr I_{\Gamma_k}$ and the interpolation conditions (\ref{eq:interpolation_conditions}), 
we have, for $k\neq \ell\in\{1,2\}$ and for any $ i\in {\cal I}_{\Gamma_k}$:
\begin{eqnarray} 
[\mathscr I_{\Gamma_k} (u_{\ell,h_\ell}{}_{|_{\Gamma_k}})](\textbf{x}_i^{(k)})&=&\sum_{j\in{\mathcal I}_{\Omega_\ell}}(\uv_\ell^0)_j[\mathscr I_{\Gamma_k} (\varphi_j^{(\ell)}{}_{|_{\Gamma_k}})](\textbf{x}_i^{(k)})
+ \sum_{j\in{\mathcal I}_{\Gamma_\ell}}(\lv_\ell)_j [\mathscr I_{\Gamma_k} (\varphi_j^{(\ell)}{}_{|_{\Gamma_k}})](\textbf{x}_i^{(k)}) \nonumber \\[1mm]
&=&\sum_{j\in{\mathcal I}_{\Omega_\ell}}(\uv_\ell^0)_j \, \varphi^{(\ell)}_j(\textbf{x}_i^{(k)})
 + \sum_{j\in{\mathcal I}_{\Gamma_\ell}}(\lv_\ell)_j \varphi_j^{(\ell)}(\textbf{x}_i^{(k)}).
 \label{eq:nc_interpolation}
\end{eqnarray}

Then, we introduce the following arrays:
\begin{eqnarray}\label{eq:matrici}
\begin{array}{ll}
\uv_k^0=[u_{k,h_k}^0({\bf x}_i^{(k)})]_i, \quad
\wv_k^0=[w_{k,h_k}^0({\bf x}_i^{(k)})]_i, & i\in {\cal I}_{\Omega_k}\\[1mm]
\lv_k=[\lambda_{k,h_k}({\bf x}_i^{(k)})]_i, & i\in  {\cal
I}_{\Gamma_k}\\[1mm]
{\mathsf A}_{kk}:\ ({\mathsf A}_{kk})_{i,j}=a_k(\varphi_j^{(k)},
\varphi_i^{(k)}), & i,j\in {\cal I}_{\Omega_k}\\[1mm]
{\mathsf A}_{k\Gamma_k}:\ ({\mathsf A}_{k\Gamma_k})_{i,j}=a_k(\varphi_j^{(k)},
\varphi_i^{(k)}), & i\in {\cal I}_{\Omega_k},\ j\in {\cal I}_{\Gamma_k}\\[1mm]
\displaystyle
{\mathsf T}_{k\ell}^0:\ ({\mathsf T}_{k\ell}^0)_{i,j}=
\varphi_j^{(\ell)}({\bf x}_i^{(k)}),& i\in {\cal I}_{\Gamma_k},\ j\in {\cal I}_{\Omega_\ell} ,\ k\neq \ell \\[1mm]
\displaystyle
{\mathsf T}_{k\ell}^\Gamma:\ ({\mathsf
T}_{k\ell}^\Gamma)_{i,j}=
\varphi_j^{(\ell)}({\bf x}_i^{(k)}),& i\in {\cal I}_{\Gamma_k},\ \ j\in {\cal I}_{\Gamma_\ell},\ k\neq \ell\\[1mm]
\fv_k=F_k(\varphi_i^{(k)}),  & i\in {\cal I}_{\Omega_k}.\\[1mm]
\end{array}
\end{eqnarray}
Here, ${\mathsf T}_{k\ell}^0$ and ${\mathsf T}_{k\ell}^\Gamma$ are \textit{intergrid matrices}, and
we notice that ${\mathsf T}_{kk}^0={\mathsf 0}$ for $k=1,2$ since $(\varphi^{(k)}_j)|_{\Gamma_k}=0$ for any $j\in{\cal I}_{\Omega_k}$, while ${\mathsf T}_{kk}^\Gamma$ is the 
identity matrix of size $n_k$.

Using these notations, the algebraic counterpart of \eqref{eq:nc_interpolation} (that is the array of nodal values of $\uv_\ell$ at the nodes of $\Gamma_k$) can be written as
\begin{equation}\label{eq:alg_nc_interpolation}
    (\uv_\ell)|_{\Gamma_k}=\mathsf T_{kl}^0\uv_\ell^0+\mathsf T_{kl}^{\Gamma}\boldsymbol\lambda_\ell,
\end{equation}
and the matrix form of the discrete optimality system (\ref{eq:os2h}) reads:
\begin{eqnarray}\label{eq:os2h_mat}
\arrayrulecolor{violachiaro}
\renewcommand{\arraystretch}{1.2}
\left[\begin{array}{cc|cc}
{\mathsf A}_{11} & {\mathsf 0} & {\mathsf A}_{1\Gamma_1} &{\mathsf 0} \\
{\mathsf 0} & {\mathsf A}_{22} & {\mathsf 0} &  {\mathsf A}_{2\Gamma_2}\\
\hline
{\mathsf 0} &  -{\mathsf T}_{12}^0 & 
{\mathsf I}_{\Gamma_1} 
& - {\mathsf T}_{12}^\Gamma\\
- {\mathsf T}_{21}^0 & {\mathsf 0} & -{\mathsf T}_{21}^\Gamma & 
\mathsf I_{\Gamma_2} 
\\
\end{array}
\right]
\left[\begin{array}{c}
\uv_1^0\\
\uv_2^0\\
\hline
\lv_1\\
\lv_2
\end{array}
\right]
=
\left[\begin{array}{c}
\fv_1\\
\fv_2\\
\hline
\zerov \\
\zerov \\
\end{array}
\right].
\end{eqnarray}
Setting
\begin{eqnarray}\label{eq:mat_blocchi}
\begin{array}{l}
{\mathsf A}_{II}=\left[\begin{array}{cc}
{\mathsf A}_{11} &{\mathsf 0} \\
 {\mathsf 0} & {\mathsf A}_{22}\\
\end{array} \right],\quad 
{\mathsf A}_{I\Gamma}=\left[\begin{array}{cc}
{\mathsf A}_{1\Gamma_1} &{\mathsf 0} \\
 {\mathsf 0} & {\mathsf A}_{2\Gamma_2}\\
\end{array} \right],\\[4mm]
{\mathsf T}^0=\left[\begin{array}{cc}
{\mathsf 0} & {\mathsf T}_{12}^0\\
{\mathsf T}_{21}^0 & {\mathsf 0}
\end{array} \right], \quad
{\mathsf T}^\Gamma=\left[\begin{array}{cc}
{\mathsf 0} & {\mathsf T}_{12}^\Gamma\\
{\mathsf T}_{21}^\Gamma & {\mathsf 0}
\end{array} \right], \\[4mm]

\uv^0=\left[\begin{array}{l}
\uv_1^0 \\
\uv_2^0
\end{array} \right], \quad 
\lbv=\left[\begin{array}{l}
\lv_1 \\
\lv_2\\
\end{array} \right], \quad
\fbv=\left[\begin{array}{l}
\fv_1 \\
\fv_2
\end{array} \right],
\end{array}
\end{eqnarray}
we obtain the compact form
\begin{eqnarray}\label{eq:os2h_matCompact}
\left[\begin{array}{cc}
{\mathsf A}_{II} & {\mathsf A}_{I\Gamma} \\
-{\mathsf T}^0 & {\mathsf I}-{\mathsf T}^\Gamma
\end{array}\right]
\left[\begin{array}{c}
\uv^0\\
\lbv\end{array}\right]
=
\left[\begin{array}{c}
\fbv\\
\zerov
\end{array}
\right],
\end{eqnarray}
where ${\mathsf I}$ denotes the identity matrix of size $n_1+n_2$.
If we eliminate the unknown $\uv^0$, we obtain the Schur complement system of (\ref{eq:os2h_matCompact}) with respect to the control (or interface) variables:
\begin{equation}\label{eq:schur_2}
{\mathsf \Sigma} \lbv=\chibv,
\end{equation}
where
\begin{equation}\label{eq:schur_2_mat_rhs}
{\mathsf \Sigma}=\mathsf I- {\mathsf T}^\Gamma+{\mathsf T}^0{\mathsf A}^{-1}_{II}
{\mathsf A}_{I\Gamma}
\quad\mbox{and} \quad
\chibv= {\mathsf T}^0{\mathsf A}_{II}^{-1}{\fbv}
\end{equation}
are the Schur complement matrix and the corresponding right--hand side associated with the minimization of the cost functional $J_\Gamma$ \eqref{eq:JGamma}.

\begin{remark}
The matrix $\mathsf A_{II}$ is block-diagonal and $\mathsf A_{11}$ and $\mathsf A_{22}$ do not share any entries. This is because the degrees of freedom corresponding to the two subdomains belong to two distinct sets, even when the two meshes match on the overlap. As explained at the beginning of Sect. \ref{sec:discrete}, this is in contrast with what is usually done in classical overlapping Schwarz methods, where the dofs on the overlap are only counted once. 
\end{remark}

In the case the two meshes $\mathcal{T}_{1,h_1}$ and $\mathcal{T}_{2,h_2}$ are conforming on $\overline{\Omega_1 \cap \Omega_2}$ (i.e., the two triangulations and the polynomial degrees coincide on the overlap), the interface equations \eqref{eq:os2h}$_3$ become
\begin{equation}\label{eq:interfaceConforming}
\lambda_{k,h_k} = u_{\ell,h_\ell \vert \Gamma_k}\,,
\quad k,\ell = 1,2, \; k \not= \ell,
\end{equation}
where, at the right-hand side, we have the restriction of $u_{\ell,h_\ell}$ to the nodes on $\Gamma_k$. The interpolation matrices $\mathsf{T}^0_{k\ell}$ ($k \not= \ell$) (see \eqref{eq:matrici}) become indeed restriction matrices from ${\cal I}_{\Omega_\ell}$ to ${\cal I}_{\Gamma_k}$, i.e.
\begin{eqnarray}\label{eq:T0_conforming}
({\mathsf T}_{k\ell}^0)_{ij}=({\mathsf R}_{\Gamma_k\ell})_{ij}=\left\{
\begin{array}{ll}
1 & \mbox{if }{\bf x}_j^{(\ell)}\in\Gamma_k\\
0 & \mbox{otherwise.}
\end{array}\right.
\end{eqnarray}
Moreover, since we have assumed that $\text{dist}(\Gamma_1,\Gamma_2)>0$, the interfaces $\Gamma_1$ and $\Gamma_2$ do not intersect and there holds
\begin{eqnarray}\label{eq:Tgamma_conforming}
\mathsf T^{\Gamma}=\mathsf 0.
\end{eqnarray}
In this case, system (\ref{eq:os2h_mat}) reduces to
\begin{equation}\label{eq:os2_alg_matrix_conf}
\arrayrulecolor{violachiaro}
\renewcommand{\arraystretch}{1.2}
\left[\begin{array}{cc|cc}
{\mathsf A}_{11} & {\mathsf 0} & {\mathsf A}_{1\Gamma_1} &{\mathsf 0} \\
{\mathsf 0} & {\mathsf A}_{22} & {\mathsf 0} &  {\mathsf A}_{2\Gamma_2}\\
\hline
{\mathsf 0} &  -{\mathsf R}_{\Gamma_12} & 
{\mathsf I}_{\Gamma_1} &  {\mathsf 0}\\
- {\mathsf R}_{\Gamma_21} & {\mathsf 0} & {\mathsf 0} & 
{\mathsf I}_{\Gamma_2}
\end{array}
\right]
\left[\begin{array}{c}
\uv_1^0\\
\uv_2^0\\
\hline
\lv_1\\
\lv_2
\end{array}
\right]
=
\left[\begin{array}{c}
\fv_1\\
\fv_2\\
\hline
\zerov\\
\zerov\\
\end{array}
\right],
\end{equation}
while the Schur complement matrix ${\mathsf\Sigma}$ and the associated right hand side $\mathsf {\boldsymbol\chi}$ become
\begin{equation}\label{eq:schur_2_conforming}
{\mathsf\Sigma} =
{\mathsf I} + {\mathsf R}_{\Gamma I}\, {\mathsf A}_{II}^{-1}\,{\mathsf A}_{I\Gamma},
\quad
{\mathsf {\boldsymbol\chi}}=
{\mathsf R}_{\Gamma I} \, {\mathsf A}_{II}^{-1} \, {\fbv},
\quad
\text{with } {\mathsf R}_{\Gamma I}=\left[\begin{array}{cc}
{\mathsf 0} & {\mathsf R}_{\Gamma_12}\\
{\mathsf R}_{\Gamma_21}& {\mathsf 0}\end{array}\right].
\end{equation}

\medskip
The Schur complement system \eqref{eq:schur_2} can be assembled and solved by direct methods when the number of degrees of freedom in each subdomain is small. This, however, is seldom the case in real-life applications, so iterative methods are typically used to solve it. Since $\mathsf \Sigma$ is not symmetric, we consider matrix-free Krylov methods like GMRES \cite{saad_schultz} for which neither ${\mathsf \Sigma}$ nor the corresponding right-hand side ${\mathsf{ \boldsymbol\chi}}$ need to be assembled. Indeed, each iteration of the Krylov method would require performing matrix-vector multiplications with the matrix ${\mathsf \Sigma}$ that, in particular, involves computing ${\mathsf A}_{II}^{-1}{\mathsf A}_{I\Gamma}\zetabv$ for any vector $\zetabv= [\zetav_1,\zetav_2]^T$ defined at the nodes on the interfaces $\Gamma_1$ and $\Gamma_2$. Recalling the definition \eqref{eq:mat_blocchi}$_1$ of the matrices ${\mathsf A}_{II}$ and ${\mathsf A}_{I\Gamma}$, this corresponds to solving local problems in $\Omega_k$ with Dirichlet boundary condition on $\Gamma_k$ described by the nodal values contained in the vector $\zetav_k$, for $k=1,2$. This can be performed by any efficient algebraic solver for inverting the local stiffness matrices ${\mathsf A}_{kk}$ or by any available existing software that can solve local Dirichlet problems independently inside each subdomain. Therefore, ICDD provides a very flexible framework that can be implemented in a fully non-intrusive way, allowing to reuse existing software to deal with each local subproblem.

\medskip

For the sake of completeness, we provide the algebraic form of the dual ICDD method (\ref{eq:os2_with_dual}).
By using the notations introduced in Sect. \ref{sec:discrete}, and exploiting \eqref{eq:alg_nc_interpolation} to formulate the last two equations of (\ref{eq:os2_with_dual}), the algebraic form of the \emph{dual ICDD method} (\ref{eq:os2_with_dual}) reads
\begin{eqnarray}\label{eq:os2h_dual_mat}
\arrayrulecolor{violachiaro}
\renewcommand{\arraystretch}{1.2}
\left[\begin{array}{cc|cc|cc}
{\mathsf A}_{11} & {\mathsf 0} & 
 {\mathsf 0} & {\mathsf 0} &
 {\mathsf A}_{1\Gamma_1} &{\mathsf 0} \\
{\mathsf 0} & {\mathsf A}_{22} &
 {\mathsf 0} & {\mathsf 0} &
{\mathsf 0} &  {\mathsf A}_{2\Gamma_2}\\
\hline
{\mathsf 0} &  -\mathsf A_{1\Gamma_1}{\mathsf T}_{12}^0 &
\mathsf A_{11} & \mathsf 0 &  
\mathsf A_{1\Gamma_1}  & - \mathsf A_{1\Gamma_1}{\mathsf T}_{12}^\Gamma\\
-\mathsf A_{2\Gamma_2} {\mathsf T}_{21}^0 & {\mathsf 0} &
\mathsf 0 & \mathsf A_{22} &
-\mathsf A_{2\Gamma_2}{\mathsf T}_{21}^\Gamma & 
\mathsf A_{2\Gamma_2} \\
\hline
-\mathsf T_{12}^\Gamma \mathsf T_{21}^0 & -\mathsf T_{12}^0 &
\mathsf 0 & \mathsf T_{12}^0 & 
\mathsf I -\mathsf T_{12}^\Gamma \mathsf T_{21}^\Gamma & \mathsf 0\\
-\mathsf T_{21}^0 & -\mathsf T_{21}^\Gamma \mathsf T_{12}^0 &
\mathsf T_{21}^0 & \mathsf 0&
\mathsf 0 & \mathsf I-\mathsf T_{21}^\Gamma \mathsf T_{12}^\Gamma
\\
\end{array}
\right]
\left[\begin{array}{c}
\uv_1^0\\
\uv_2^0\\
\hline
\pv_1^0\\
\pv_2^0\\
\hline
\lv_1\\
\lv_2
\end{array}
\right]
=
\left[\begin{array}{c}
\fv_1\\
\fv_2\\
\hline
0\\
0\\
\hline
0\\
0\\
\end{array}
\right],
\end{eqnarray}
or equivalently, 
\begin{eqnarray*}
\left[\begin{array}{ccc}
\mathsf A_{II} & \mathsf 0 & \mathsf A_{I\Gamma}\\
-\mathsf A_{I\Gamma}\mathsf T^0 & \mathsf A_{II} & \mathsf A_{I\Gamma}(\mathsf I-\mathsf T^\Gamma)\\
-(\mathsf I+\mathsf T^\Gamma)\mathsf T^0 & \mathsf T^0 & (\mathsf I+\mathsf T^\Gamma)(\mathsf I-\mathsf T^\Gamma)
\end{array}\right]
\left[\begin{array}{c}
\uv^0\\
\pv^0\\
\lbv
\end{array}\right]=
\left[\begin{array}{c}
\fbv^0\\
\mathsf 0\\
\mathsf 0
\end{array}\right].
\end{eqnarray*}
Notice that the stiffness matrices of the dual problems (in the third and fourth equations of (\ref{eq:os2h_dual_mat})) coincide with the matrices appearing in the first two equations. Thus, no extra storage is required to solve the dual problems whose unknowns are $\pv^0_1$ and $\pv^0_2$.

The Schur complement system of (\ref{eq:os2h_dual_mat}) with respect to the control (interface) variable $\boldsymbol{\lambda}$ is
\begin{equation}\label{eq:schur_dual}
    \mathsf \Sigma_D \boldsymbol\lambda = \boldsymbol\chi_D\,,
\end{equation}
where
\begin{equation}\label{eq:schur_dual_mat_rhs}
\mathsf{\Sigma}_D = (2\mathsf{I} - \mathsf{\Sigma}) \, \mathsf{\Sigma}
\quad \text{and} \quad
\boldsymbol{\chi}_D = (2\mathsf{I} - \mathsf{\Sigma}) \, \boldsymbol{\chi} \, .
\end{equation}

It is clear that each GMRES iteration to solve the dual system \eqref{eq:os2h_dual_mat} or \eqref{eq:schur_dual} costs about twice as much as one GMRES iteration for solving \eqref{eq:os2h_mat} or \eqref{eq:schur_2}. Thus, using the dual ICDD method is advantageous only if it halves the number of iterations at least. We will further discuss this in Sect. \ref{sec:numerical_results}.

\subsection{Comparison with Schwarz-type methods}\label{sec:comparison_Schwarz}

In the case of conforming discretizations on the overlap, ICDD shares similarities with some Schwarz methods.

First, notice that system \eqref{eq:os2_alg_matrix_conf} coincides with system (1.10) in \cite{sbg} that reformulates the Schwarz method by highlighting the role of interface and internal degrees of freedom. However, the most common approach in the literature is to formulate the Schwarz method as a preconditioner instead of as a solver. More precisely, considering the differential problem \eqref{eq:globalProblem}, and introducing a suitable discretization on a global triangulation ${\cal T}_h$ in $\Omega$, one obtains the corresponding linear system
\begin{equation}\label{eq:global_linear_system}
    {\mathsf A}\,\uv=\mathsf \fv.
\end{equation}
This is solved by a suitable iterative method (e.g., CG, GMRES) using Schwarz-type preconditioners such as the Additive Schwarz \cite{toselli-widlund} and the Restricted Additive Schwarz \cite{Cai-Sarkis1999} preconditioners.

In the formulation of these methods, the two sub-triangulations $\mathcal T_{1,h}$ and $\mathcal T_{2,h}$ of the overlapping subdomains $\Omega_1$ and $\Omega_2$ are obtained as the restriction of the global triangulation ${\cal T}_h$ of $\Omega$, and, differently from the ICDD approach, no duplication of nodes occurs in the overlapping region. Then, for $k = 1,2$, let $W_k$ be the set of dofs in the open subdomain $\Omega_k$ and $N_k$ its dimension, while $W$ of dimension $N$ is the set of dofs of $\uv$.
Moreover, let ${\mathsf R}_k\in\mathbb R^{N_k\times N}$ be the restriction matrices that map the array of global dofs to the local ones, and the local matrices ${\mathsf A}_{kk}={\mathsf R}_k\,{\mathsf A}\,{\mathsf R}_k^T$. 

Then, the \emph{Additive Schwarz (AS)} preconditioner is defined by 
\begin{equation}\label{eq:preco_as}
{\mathsf P_{AS}^{-1}=\mathsf R_1^T\mathsf A_{11}^{-1}\mathsf R_1+\mathsf R_2^T\mathsf A_{22}^{-1}\mathsf R_2},
\end{equation}
and the system 
\begin{equation}\label{eq:pas_global}
\mathsf P_{AS}^{-1}\,{\mathsf A}\,\uv=\mathsf P_{AS}^{-1}\,\fv
\end{equation}
is typically solved by the Conjugate Gradient method because both $\mathsf A$ and $\mathsf P_{AS}$ are symmetric and positive definite matrices. 

The \emph{Restricted Additive Schwarz (RAS)} preconditioner introduced in \cite{Cai-Sarkis1999} is a better-performing variant of AS. The idea of RAS is to replace the matrices $\mathsf R_k^T$ in (\ref{eq:preco_as}) with suitable prolongation matrices $\widetilde{\mathsf R}_k\in\mathbb R^{N\times N_k}$, which extend the array of local dofs to the global one, such that
\begin{equation*}
    \widetilde {\mathsf R}_1 \mathsf R_1+\widetilde {\mathsf R}_2 \mathsf R_2=\mathsf I.
\end{equation*}
The original definition of $\widetilde {\mathsf R}_k$ is given in \cite{cai_dryja_sarkis}, and other choices, including suitable weights, are proposed in \cite{frommer_szyld}.
After defining 
\begin{equation*}
{\mathsf P_{RAS}^{-1}=\widetilde{\mathsf R}_1\mathsf A_{11}^{-1}\mathsf R_1+\widetilde{\mathsf R}_2\mathsf A_{22}^{-1}\mathsf R_2},
\end{equation*}
solving (\ref{eq:global_linear_system}) with RAS means solving
\begin{equation}\label{eq:pras_global}
\mathsf P_{RAS}^{-1}\,{\mathsf A}\,\uv = \mathsf P_{RAS}^{-1}\,\fv
\end{equation}
by a suitable Krylov method.
Notice that, although $\mathsf A$ is a symmetric positive definite matrix, $\mathsf{P}_{RAS}$ is not.   

Remark that for both the AS and the RAS preconditioners, the array $\uv$ contains volume degrees of freedom, thus the dimension of the systems  (\ref{eq:pas_global}) and \eqref{eq:pras_global} is much larger than that of (\ref{eq:schur_2_conforming}).

An approach that uses only the interface degrees of freedom as in (\ref{eq:schur_2_conforming}), is the Substructured Restricted Additive Schwarz (SRAS) method proposed in \cite{SRAS2022}. This is a variant of RAS that works directly on the substructure $S$ of the decomposition, which is the set of the dofs belonging to $\cup_k (\partial\Omega_k\setminus\partial\Omega)$.
To recall the formulation of this method, we introduce the restriction matrix $\overline {\mathsf R}: W\to S$, which maps the dofs from the global space to the substructure, and the prolongation matrix $\mathsf P: S\to W$ that extends arrays with substructure dofs to the global volume arrays. A possible choice for $\overline {\mathsf P}$ is $\overline {\mathsf P}=\overline {\mathsf R}^T$. Then, denoting by $\uv_S$ the array containing the dofs belonging to $S$, the SRAS method reads:
\begin{equation}\label{eq:preco_sras}
\overline {\mathsf R}\,\mathsf P_{RAS}^{-1}{\mathsf A}\overline {\mathsf P}\,\uv_S=\overline {\mathsf R}\,\mathsf P_{RAS}^{-1}\,\fv
 \end{equation}
and system \eqref{eq:preco_sras} is solved by a suitable Krylov method. In this context, $\uv_S$ plays the role of the vector $\boldsymbol{\lambda}$ in ICDD (see \eqref{eq:schur_2_conforming}).
In Appendix \ref{app:ICDD-SRAS}, we show that, in the case of conformal discretizations (see Assumption \ref{ass_conforming}), ICDD and SRAS coincide. Therefore, the convergence analysis of ICDD that we develop in Sect. \ref{sec:convergence-analysis} also applies to SRAS, thus extending the theoretical results of \cite{SRAS2022}. From an implementation point of view, it must be noted that ICDD does not require assembling the global matrix $\mathsf A$, computing the global residual, and sharing the residual among the different processors that manage the local subproblems. This further reduces the overall computational cost as only the interface degrees of freedom need to be exchanged.

\section{Spectral properties of the ICDD matrices}\label{sec:convergence-analysis}

In this section, we analyze the convergence of GMRES iterations to solve the Schur complement system \eqref{eq:schur_2}. In particular, we are interested in measuring the convergence rate versus the overlap width $\delta$ and the discretization parameters $h_k$ and $p_k$, i.e., the local mesh size and local polynomial degree. Our estimates will also take into account the size of the subdomains. Although this information is not crucial for the case of two-subdomains that we primarily consider in this work, the estimate is relevant, considering that the ICDD method can be extended to more general decompositions with more than two subdomains, even including cross points. The analysis is carried out in the 2D case and under the assumption that the local triangulations are conforming in the overlapping region.

This section is organised as follows. First, in Sect. \ref{sec:weak_schur}, we rewrite the interface equations \eqref{eq:interfaceConforming} in weak form as this will permit us to provide a variational interpretation of the Schur complement matrix $\mathsf{\Sigma}$. Then, we estimate the eigenvalues of the matrices related to $\mathsf{\Sigma}$ (Sect. \ref{sec:analysisSchur}) exploiting the estimate on the trace of the state solutions on the interfaces, provided in Sect. \ref{sec:estimatesOnTraces}.

\subsection{Weak formulation of the interface conditions}\label{sec:weak_schur}

We introduce the following Assumption \ref{ass_conforming}.

\begin{assumption}\label{ass_conforming}
Let the local meshes $\mathcal{T}_{h,h_k}$ be conforming on the overlapping region $\overline{\Omega_1 \cap \Omega_2}$, then we set $h=h_1=h_2$ in $\overline{\Omega_1 \cap \Omega_2}$ and $p=p_1=p_2$ in $\Omega_1 \cup \Omega_2$.
\end{assumption}

For $k=1,2$, the pair of interface equations \eqref{eq:interfaceConforming} can be reformulated weakly as
\begin{equation*}
\int_{\Gamma_1}(\lambda_{1,h} - u_{2,h}{}_{|\Gamma_1})\mu_{1,h} d\Gamma
+\int_{\Gamma_2}(\lambda_{2,h} - u_{1,h}{}_{|\Gamma_2}) \mu_{2,h} d\Gamma
=0 \qquad \forall
(\mu_{1,h},\mu_{2,h})\in{\boldsymbol\Lambda}_h.
\end{equation*}

Thanks to (\ref{eq:uh_expansion}) and by expanding $\lambda_{k,h}$ with respect to the basis $\mu_j^{(k)}$, we have, for $i \in \mathcal{I}_{\Gamma_k}$,
\begin{eqnarray}\label{eq:weak-difference-on-Gamma}
    \begin{array}{ll}
    \displaystyle \int_{\Gamma_k}(\lambda_{k,h}
    - u_{\ell,h}{}_{|\Gamma_k}) \mu_i^{(k)} 
    & \displaystyle = \int_{\Gamma_k}
    \Big(\sum_{j\in{\mathcal I}_{\Gamma_k}}(\lv_k)_j \mu^{(k)}_j
    -\sum_{j\in{\mathcal I}_{\Omega_\ell}}(\uv_\ell^0)_j \varphi^{(\ell)}_j{}_{|\Gamma_k} - \sum_{j\in{\mathcal I}_{\Gamma_\ell}}(\lv_\ell)_j \varphi^{(\ell)}_j{}_{|\Gamma_k} \Big)\mu^{(k)}_i\\[2mm]
    & \displaystyle =
    \sum_{j\in{\mathcal I}_{\Gamma_k}}(\lv_k)_j
    \int_{\Gamma_k} \mu^{(k)}_j \mu^{(k)}_i
    - \sum_{j\in{\mathcal I}_{\Omega_\ell}}(\uv_\ell^0)_j
    \int_{\Gamma_k} \varphi^{(\ell)}_j{}_{|\Gamma_k}\,  \mu^{(k)}_i\\[2mm]
    &\displaystyle \qquad - \sum_{j\in{\mathcal I}_{\Gamma_\ell}}(\lv_\ell)_j
    \int_{\Gamma_k} \varphi^{(\ell)}_j{}_{|\Gamma_k}\, \mu^{(k)}_i.
\end{array}
\end{eqnarray}

The integrals $\int_{\Gamma_k} \mu^{(k)}_j \mu^{(k)}_i$ involve basis functions defined on the same mesh, so we denote by ${\mathsf M}_{\Gamma_k}$ the mass matrix on the interface  $\Gamma_k$, whose entries are
\begin{equation}\label{eq:mass_exact}
({\mathsf M}_{\Gamma_k})_{i,j}=\int_{\Gamma_k}\mu_j^{(k)}\,
\mu_i^{(k)}, \quad i,j\in{\mathcal I}_{\Gamma_k}.
\end{equation}

Moreover, we set, for $k\neq \ell$,
\begin{eqnarray*}
\begin{array}{ll}
\displaystyle (\widetilde{\mathsf R}_{\Gamma_k\ell})_{i,j} & \displaystyle := \int_{\Gamma_k} \varphi_j^{(\ell)}{}_{|\Gamma_k}\, \mu_i^{(k)} = 
\displaystyle \int_{\Gamma_k}\sum_{m\in {\mathcal I}_{\Gamma_k}} \varphi_j^{(\ell)}(\mathbf x_m^{(k)}) \mu_m^{(k)}\, \mu_i^{(k)} \\[1mm]
 &=\displaystyle
\sum_{m\in {\mathcal I}_{\Gamma_k}} ({\mathsf R}_{\Gamma_k\ell})_{m,j}
\int_{\Gamma_k}\mu_m^{(k)}\mu_i^{(k)}=(\mathsf M_{\Gamma_k}{\mathsf R}_{\Gamma_k \ell})_{i,j},
\qquad i\in\mathcal I_{\Gamma_k}, \; j\in\mathcal I_{\Omega_\ell},
\end{array}
\end{eqnarray*}
so that
\begin{equation*}
\widetilde{\mathsf{R}}_{\Gamma_k \ell} = \mathsf M_{\Gamma_k}{\mathsf R}_{\Gamma_k \ell} \, .
\end{equation*}
Finally, since $\varphi^{(\ell)}_j{}_{|\Gamma_k} = 0$ for any $j\in \mathcal I_{\Gamma_\ell}$ due to the hypothesis that $\Gamma_1\cap\Gamma_2=\emptyset$ and to the conformity of the meshes, the last integral in \eqref{eq:weak-difference-on-Gamma} is identically null.

It follows that
\begin{equation}\label{eq:weak-difference-on-Gamma-1}
    \int_{\Gamma_k}(\lambda_{k,h} - u_{\ell,h}{}_{|\Gamma_k}) \mu_i^{(k)}=
    (\mathsf M_{\Gamma_{k}} \, \lv_k - \widetilde{\mathsf R}_{\Gamma_k \ell} \, \uv_\ell^0 )_i, \qquad \mbox{for any }i\in {\mathcal I_{\Gamma_k}}.
\end{equation}
Now, by setting
\begin{equation*}
\widetilde{\mathsf R}_{\Gamma I} = \left[\begin{array}{cc}
{\mathsf 0} & \widetilde{\mathsf R}_{\Gamma_1 2}\\
\widetilde{\mathsf R}_{\Gamma_2 1} & {\mathsf 0}
\end{array} \right], 
\end{equation*}
the counterpart of \eqref{eq:os2_alg_matrix_conf} for the weak form of the interface equations becomes
\begin{eqnarray*}
\left[\begin{array}{cc}
{\mathsf A}_{II} & {\mathsf A}_{I\Gamma} \\
-\widetilde{\mathsf R}_{\Gamma I} & \mathsf M_\Gamma
\end{array}\right]
\left[\begin{array}{c}
\uv^0\\
\lbv\end{array}\right]
=
\left[\begin{array}{c}
\fbv\\
\zerov
\end{array}
\right]
\end{eqnarray*}
where 
\begin{equation*}
{\mathsf M}_\Gamma=\left[\begin{array}{cc}
{\mathsf M}_{\Gamma_1} & {\mathsf 0}\\
{\mathsf 0} & {\mathsf M}_{\Gamma_2}
\end{array} \right].
\end{equation*}

By eliminating the unknown $\uv^0$, the \textit{weak} Schur complement system with respect to the interface control variables becomes:
\begin{equation}\label{eq:schur_2_tilde}
\widetilde{\mathsf \Sigma} \lbv=\widetilde{\mathsf{\boldsymbol\chi}},
\end{equation}
where
\begin{equation}\label{eq:schur_2_mat_rhs_tilde}
\widetilde{\mathsf \Sigma} = \mathsf M_\Gamma + \widetilde{\mathsf R}_{\Gamma I}{\mathsf A}^{-1}_{II}
{\mathsf A}_{I\Gamma}
\quad\mbox{and} \quad
\widetilde{\mathsf{\boldsymbol\chi}}= \widetilde{\mathsf R}_{\Gamma I}{\mathsf A}_{II}^{-1}{\fbv}.
\end{equation}

\smallskip

The nomenclature \textit{weak} for the Schur complement $\widetilde{\mathsf{\Sigma}}$ refers to the weak formulation of the interface equation \eqref{eq:weak-difference-on-Gamma-1} that was used to obtain it. Moreover, it is straightforward to see that
\begin{equation}\label{eq:sigma-chi-strong-weak}
    \widetilde{\mathsf \Sigma}={\mathsf M}_\Gamma \mathsf \Sigma
    \qquad \text{and} \qquad
    \widetilde{\mathsf{\boldsymbol\chi}}={\mathsf M}_\Gamma \mathsf{\boldsymbol\chi}.
\end{equation}

\smallskip

The following theorem characterizes the weak Schur complement operator $\widetilde{\mathsf\Sigma}$ defined in (\ref{eq:schur_2_mat_rhs_tilde}). 

\begin{theorem}
For $k=1,2$ and for any $\zeta_{k,h}\in \Lambda_{k,h}$, let
$u_{k,h}^{\zeta_{k}}$ be the \emph{discrete harmonic extension} of $\zeta_{k,h}$
to $\Omega_k$, i.e. the function $u_{k,h}^{\zeta_{k}}\in V_{k,h}$ s.t. 
$u_{k,h}^{\zeta_{k}}=\zeta_{k,h}$ on $\Gamma_k$
and 
\begin{equation*}
a_k(u_{k,h}^{\zeta_{k}},v_{h})=
\int_{\Omega_k} \left( \nu_k \, \nabla u_{k,h}^{\zeta_{k}} \cdot \nabla v_{h} +
\gamma_k \,  u_{k,h}^{\zeta_{k}} v_{h} \right) =
0 \qquad \forall v_{h}\in V_{k,h}^0.    
\end{equation*}
Then,
\begin{equation}\label{def:char_Sigma_2}
    {\zetabv^T \,\widetilde{\mathsf\Sigma} \,\zetabv}
=\int_{\Gamma_1} (u_{1,h}^{\zeta_1} - u_{2,h}^{\zeta_2}{}_{|\Gamma_1})\zeta_{1,h}
+
\int_{\Gamma_2}(u_{2,h}^{\zeta_2} - u_{1,h}^{\zeta_1}{}_{|\Gamma_2})\zeta_{2,h},
\end{equation}
where $\zetabv=\begin{bmatrix}\zetav_1 ,\, \zetav_2\end{bmatrix}^T$, and $\zetav_k=[\zeta_{k,h}(\mathbf x)_i^{(k)}]_{i\in \mathcal I_{\Gamma_k}}$.
\end{theorem}
\begin{proof}
Let us denote by $\uv_k^\zeta=[u_{k,h}^{\zeta_k}({\bf x}_i^{(k)})]_{i\in {\cal I}_{\Omega_k}}$ the array of the degrees of freedom internal to $\Omega_k$ of the harmonic extension of $\zeta_{k,h}$. Using the notations introduced in (\ref{eq:matrici}), $\uv_k^\zeta$ is the solution of the algebraic system
\begin{equation*}
 {\mathsf A}_{kk} \uv_k^\zeta+{\mathsf A}_{k\Gamma_k}\zetav_k=\zerov, 
\end{equation*}
thus $\uv_k^\zeta=-{\mathsf A}_{kk}^{-1}{\mathsf A}_{k\Gamma_k}\zetav_k$.
After setting 
$\uv^\zeta=\begin{bmatrix} \uv_1^\zeta, \,\uv_2^\zeta \end{bmatrix}^T$, we can write $\uv^\zeta=-{\mathsf A}_{II}^{-1}{\mathsf A}_{I\Gamma}\zetabv$, and thanks to (\ref{eq:schur_2_mat_rhs_tilde}) and (\ref{eq:weak-difference-on-Gamma-1}), it holds
\begin{eqnarray*}
\zetabv^T \widetilde{\mathsf\Sigma} \, \zetabv &=& \displaystyle \zetabv^T (\mathsf M_\Gamma \zetabv + \widetilde{\mathsf R}_{\Gamma I}{\mathsf A}_{II}^{-1}{\mathsf A}_{I\Gamma}\zetabv)=\zetabv^T (\mathsf M_\Gamma \zetabv - \widetilde{\mathsf R}_{\Gamma I}\uv^\zeta)
= \begin{bmatrix}\zetav_1 ,\, \zetav_2\end{bmatrix}
\begin{bmatrix}
 {\mathsf M}_{\Gamma_1} \zetav_1 - \widetilde {\mathsf R}_{\Gamma_1 2} \uv_2^\zeta\\
   {\mathsf M}_{\Gamma_2} \zetav_2 - \widetilde {\mathsf R}_{\Gamma_2 1} \uv_1^\zeta
\end{bmatrix}
\\[3mm]
&=& \displaystyle
\int_{\Gamma_1}(\zeta_{1,h} - u_{2,h}^{\zeta_2}{}_{|\Gamma_1})\zeta_{1,h} +
\int_{\Gamma_2}(\zeta_{2,h} - u_{1,h}^{\zeta_1}{}_{|\Gamma_2})\zeta_{2,h}.
\end{eqnarray*}
Recalling that $u_{k,h}^{\zeta_k}=\zeta_{k,h}$ on $\Gamma_k$, the thesis follows.
\end{proof}

\subsection{Analysis of the weak Schur complement matrix}\label{sec:analysisSchur}

\medskip

We now proceed to analyze the Schur complement matrix $\widetilde{\mathsf \Sigma}$ defined in \eqref{eq:sigma-chi-strong-weak}.

\medskip

The following Lemma estimates the spectral properties of the interface mass matrix ${\mathsf M}_{\Gamma_k}$.

\begin{lemma}\label{lemma:eig_mass}
Let Assumptions \ref{ass_mesh} and \ref{ass_conforming} hold and let $\mathsf M_{\Gamma_k}$ be the interface mass matrix defined in (\ref{eq:mass_exact}). 
There exist two positive constants $\widehat{c}_1$ and $\widehat{c}_2$ independent of $h$ and of $p$, such that, for any $\zetav_k\in\mathbb R^{n_k}\setminus\ \{ \zerov \}$, it holds
\begin{equation}\label{eq:eig_mass_general}
\widehat{c}_1 \frac{h}{p^2}\leq \frac{\zetav_k^T{\mathsf M}_{\Gamma_k}\zetav_k}{\zetav_k^T \zetav_k}\leq \widehat{c}_2 \frac{h}{p}, \qquad k=1,2.
\end{equation}
\end{lemma}

\begin{proof}
The proof uses some classical results collected in Appendix \ref{appendix:mass}. Let $\mathsf M_{{LGL},\Gamma_k}$ be the mass matrix on $\Gamma_k$ defined using the discrete inner product (\ref{eq:LGL-innerproduct}) on each element of $\Gamma_k$. Then, thanks to \eqref{eq:extrema_eig_MLGL} and scaling arguments, there exist two positive constants $\tilde{c}_1$ and $\tilde{c}_2$ independent of $h$ and of $p$, such that, for any $\zetav_k\in\mathbb R^{n_k}\setminus\ \{ \zerov \}$, it holds
\begin{equation*}
\tilde{c}_1 \frac{h}{p^2}\leq \frac{\zetav_k^T{\mathsf M}_{LGL,\Gamma_k}\zetav_k}{\zetav_k^T \zetav_k}\leq \tilde{c}_2 \frac{h}{p}, \qquad k=1,2.
\end{equation*}
Finally, in view of (\ref{eq:normEquivalenceMatrix}), the thesis follows with 
$\widehat c_1=\tilde c_1$ and $\widehat c_2=3\tilde c_2$.
\end{proof}

\medskip

We can now prove the following theorem that provides an upper bound for the norm
$\|\widetilde{\mathsf \Sigma}\|_2$.

\begin{theorem}\label{thm:upperbound_max_eig}
Let Assumptions \ref{ass_mesh}--\ref{ass_conforming} be satisfied. Then,
\begin{equation}\label{eq:upperbound_Sigmatilde}
    \|\widetilde{\mathsf{\Sigma}}\|_2 \leq h \, \overline{c}(p),
\end{equation}
where
\begin{equation}\label{eq:overlineCp}
    \overline{c}(p) = 
    \frac{\widehat{c}_2}{p}
    \sqrt{ 1 +  \frac{\widehat{c}_2}{\widehat{c}_1} p } \,,
\end{equation}
and the constant $\widehat{c}_2$ was introduced in Lemma \ref{lemma:eig_mass}.
\end{theorem}

\begin{proof}
Knowing that
\begin{equation}\label{eq:thm6_step0}
    \|\widetilde{\mathsf \Sigma}\|_2 = 
    \sqrt{\lambda_{max}(\widetilde{\mathsf \Sigma}^T\widetilde{\mathsf \Sigma})} =
    \sqrt{\sup_{\mathsf 0\neq \zetabv \in\mathbb R^n} \frac{\zetabv^T\widetilde{\mathsf \Sigma}^T\widetilde{\mathsf \Sigma}\,\zetabv}{\zetabv^T\zetabv}} \, ,
\end{equation}
we start by estimating $\zetabv^T\widetilde{\mathsf \Sigma}^T\widetilde{\mathsf \Sigma}\,\zetabv$.

By definition \eqref{eq:schur_2_mat_rhs_tilde} of the matrix $\widetilde{\mathsf \Sigma}$, there holds $\widetilde{\mathsf \Sigma} \zetabv = \mathsf{M}_\Gamma \boldsymbol{d}$ with
\begin{equation*}
    \boldsymbol{d} = 
    \begin{bmatrix}
    \boldsymbol{d}_1 \\ \boldsymbol{d}_2    
    \end{bmatrix} =
    \begin{bmatrix}
        \zetabv_1 - \boldsymbol{t}_{12} \\
        \zetabv_2 - \boldsymbol{t}_{21}
    \end{bmatrix}
    \quad \text{and} \quad
    \boldsymbol{t}_{k\ell} = - \mathsf{R}_{\Gamma_k \ell}\, \mathsf{A}_{\ell\ell}^{-1} \, \mathsf{A}_{\ell\Gamma_\ell} \, \zetabv_\ell
    \quad \text{for} \quad
    k,\ell=1,2, \; k\not= \ell\,,
\end{equation*}
and where we remark that the vector $\boldsymbol{t}_{k\ell}$ is the algebraic counterpart of $u_{\ell,h\vert\Gamma_k}$.
Then,
\begin{equation*}
\zetabv^T \widetilde{\mathsf{\Sigma}}^T \widetilde{\mathsf{\Sigma}}\zetabv
= \boldsymbol{d}^T \mathsf{M}_\Gamma^2 \boldsymbol{d} 
= \boldsymbol{d}_1^T \mathsf{M}_{\Gamma_1}^2 \boldsymbol{d}_1 
+ \boldsymbol{d}_2^T \mathsf{M}_{\Gamma_2}^2 \boldsymbol{d}_2 \, .
\end{equation*}
Since $\mathsf M_{\Gamma_k}$ is symmetric positive definite, we can introduce its square root $M^{1/2}_{\Gamma_k}$ \cite[Sect. 4.4.2]{chqz06}. Thus, as a consequence of Lemma \ref{lemma:eig_mass}, it holds
\begin{equation*}
    \boldsymbol{d}_k^T \mathsf{M}_{\Gamma_k}^2 \boldsymbol{d}_k =(\mathsf M^{1/2}_{\Gamma_k}\boldsymbol d_k)^T \mathsf M_{\Gamma_k}(\mathsf M^{1/2}_{\Gamma_k}\boldsymbol d_k)\leq \widehat c_2\frac{h}{p}\boldsymbol d_k^T\mathsf M_{\Gamma_k}\boldsymbol d_k
    \leq {\widehat{c}_2}^{\;2} \, \frac{h^2}{p^2} \, \boldsymbol{d}_k^T \boldsymbol{d}_k \qquad k=1,2,
\end{equation*}
and, by the triangle inequality, we obtain
\begin{equation}\label{eq:thm6_step1}
    \zetabv^T \widetilde{\mathsf{\Sigma}}^T \widetilde{\mathsf{\Sigma}}\zetabv
    \leq \widehat{c}_2^{\;2} \, \frac{h^2}{p^2} \,
         \left(
          \| \zetabv_1\|_2^2 +\| \boldsymbol{t}_{12}\|_2^2  +
          \| \zetabv_2\|_2^2 +\| \boldsymbol{t}_{21}\|_2^2 
         \right)\,,
\end{equation}
where $\|\boldsymbol d\|_2=\boldsymbol d^T \boldsymbol d$ is the Euclidean norm.\\
We proceed to estimate $\|\boldsymbol{t}_{k\ell}\|_2^2$ for $k,\ell=1,2$, $k\not= \ell$. Using Lemma \ref{lemma:eig_mass} and the meaning of $\boldsymbol t_{k\ell}$, 
we find
\begin{equation*}
    \|\boldsymbol{t}_{k\ell}\|_2^2 
    \leq 
    \frac{p^2}{\widehat{c}_1 \,h} \boldsymbol{t}_{k\ell}^T \mathsf{M}_{\Gamma_k} \boldsymbol{t}_{k\ell}
    \leq
    \frac{p^2}{\widehat{c}_1 \, h} \| u_{\ell,h} \|_{L^2(\Gamma_k)}^2\,.
\end{equation*}

Using \eqref{eq:stimal2_gg}, we obtain
\begin{equation*}
    \| u_{\ell,h} \|_{L^2(\Gamma_k)}^2
    \leq
    C_{\delta,\ell}^2\,  \| \zeta_{\ell,h} \|_{L^2(\Gamma_\ell)}^2 \, .
\end{equation*}
Moreover, using again Lemma \ref{lemma:eig_mass}, we find
\begin{equation*}
    \| \zeta_{\ell,h} \|_{L^2(\Gamma_\ell)}^2 = \zetabv_\ell^T \, \mathsf{M}_{\Gamma_\ell} \, \zetabv_\ell
    \leq
    \widehat{c}_2 \, \frac{h}{p} \| \zetabv_\ell \|_2^2 \, .
\end{equation*}
Considering that
\begin{equation*}
 C_{\delta,\ell}^2\, < 1 \qquad \forall \delta >0,\ \ell=1,2,
\end{equation*}
we obtain the following bound for the Euclidean norm of $\boldsymbol{t}_{k\ell}$:
\begin{equation}\label{eq:thm6_step2}
    \| \boldsymbol{t}_{k\ell} \|_2^2 \leq \frac{\widehat{c}_2}{\widehat{c}_1} \, p \, \|\zetabv_\ell\|_2^2\,.
\end{equation}
Substituting \eqref{eq:thm6_step2} into \eqref{eq:thm6_step1}, we obtain the estimate
\begin{equation*}
    \zetabv^T \widetilde{\mathsf{\Sigma}}^T\widetilde{\mathsf{\Sigma}}\zetabv
    \leq \widehat{c}_2^{\;2} \, \frac{h^2}{p^2}
    \left(
     \| \zetabv_1\|_2^2 \left( 1 + \frac{\widehat{c}_2}{\widehat{c}_1} \, p \right) + \|\zetabv_2\|_2^2 \left( 1 + \frac{\widehat{c}_2}{\widehat{c}_1} \, p \right) 
    \right) 
    \leq
    h^2 \frac{\widehat{c}_2^{\;2}}{p^2}
    \left( 1 + \frac{\widehat{c}_2}{\widehat{c}_1} p \right) \| \zetabv\|_2^2 \,.
\end{equation*}
Therefore, it holds
\begin{equation*}
    \frac{\zetabv^T \widetilde{\mathsf{\Sigma}}^T\widetilde{\mathsf{\Sigma}}\zetabv}{\zetabv^T\zetabv}
    \leq
    h^2 (\overline{c}(p))^2 \, ,
\end{equation*}
and the thesis follows from \eqref{eq:thm6_step0} with $\overline{c}(p)$ defined as in \eqref{eq:overlineCp}.
\end{proof} 

\medskip

Finally, we can provide a bound for the minimum eigenvalue $\lambda_{min}(\widetilde{\mathsf \Sigma}_s)$ of the symmetric part of the Schur complement matrix $\widetilde{\mathsf \Sigma}$.

\begin{theorem}\label{thm:lowerbound_min_eig}
Let Assumptions \ref{ass_mesh}--\ref{ass_conforming} be satisfied, and let $\widetilde{\mathsf \Sigma}_s=\frac{1}{2} (\widetilde{\mathsf \Sigma}+\widetilde{\mathsf\Sigma}^T)$ be the symmetric part of $\widetilde{\mathsf\Sigma}$. The matrix $\widetilde{\mathsf \Sigma}_s$ is positive definite and there holds
\begin{equation}\label{eq:lowerbound_min_eig}
    \lambda_{min}(\widetilde{\mathsf \Sigma}_s)= 
\inf_{\mathsf 0\neq \zetabv \in\mathbb R^n}\frac{\zetabv^T\widetilde{\mathsf \Sigma}_s\,\zetabv}{\zetabv^T\zetabv}\geq h \, \frac{\delta}{c_{LH} + \max_\ell\frac{\|\nu_\ell\|_\infty}{\underline\nu_\ell}\delta} \, \underline{c}(p) > 0 \, ,
\end{equation}
where
\begin{equation*}
    \underline{c}(p) = \frac{\widehat{c}_1}{2\,p^2} \, ,
\end{equation*}
the constant $\widehat{c}_1$ is defined in Lemma \ref{lemma:eig_mass} and $c_{LH}=\max\{c_{L,1}\,c_{H,1},\,  c_{L,2}\,c_{H,2}\}$.
\end{theorem}

\begin{proof}
By Cauchy--Schwarz inequality, Corollary \ref{cor:gamma_gamma}, and by the algebraic inequality $ab\leq (a^2+b^2)/2$, we have
\begin{eqnarray}
    \displaystyle\int_{\Gamma_1} \left(u_{2,h}^{\zeta_2}{}_{|\Gamma_1}\right)\zeta_{1,h} 
    &\leq& \|u_{2,h}^{\zeta_2}\|_{L^2(\Gamma_1)} \|\zeta_{1,h}\|_{L^2(\Gamma_1)} \label{eq:stima_mista} \\
    &\leq& 
    C_{\delta, 2}\, \|\zeta_{2,h}\|_{L^2(\Gamma_2)}  \|\zeta_{1,h}\|_{L^2(\Gamma_1)} \nonumber \\
    &\leq& \displaystyle \frac{1}{2} 
    C_{\delta, 2}\, \left( \|\zeta_{1,h}\|^2_{L^2(\Gamma_1)}+ \|\zeta_{2,h}\|^2_{L^2(\Gamma_2)} \right). \nonumber
\end{eqnarray}
Similarly, 
\begin{eqnarray*}
\int_{\Gamma_2}
\left(u_{1,h}^{\zeta_1}{}_{|\Gamma_2}\right)\zeta_{2,h}
\leq \frac{1}{2} 
C_{\delta,1}\, \left( \|\zeta_{1,h}\|^2_{L^2(\Gamma_1)}+ \|\zeta_{2,h}\|^2_{L^2(\Gamma_2)} \right).
\end{eqnarray*}
Then, thanks to (\ref{def:char_Sigma_2}),
\begin{eqnarray*}
\begin{array}{ll}
\zetabv^T \widetilde{\mathsf\Sigma} \, \zetabv & 
\displaystyle = \int_{\Gamma_1} (u_{1,h}^{\zeta_1}-u_{2,h}^{\zeta_2})\zeta_{1,h} + \int_{\Gamma_2} (u_{2,h}^{\zeta_2}- u_{1,h}^{\zeta_1})\zeta_{2,h}\\[2mm] 
&\displaystyle = \|\zeta_{1,h}\|^2_{L^2(\Gamma_1)}+ \|\zeta_{2,h}\|^2_{L^2(\Gamma_2)} -\int_{\Gamma_1}u_{2,h}^{\zeta_2}\,\zeta_{1,h}
- \int_{\Gamma_2} u_{1,h}^{\zeta_1}\,\zeta_{2,h} \\[2mm]
&  \geq \displaystyle \left(1
-\frac{1}{2}C_{\delta,1}-\frac{1}{2}C_{\delta,2}\, \, \right)
(\|\zeta_{1,h}\|^2_{L^2(\Gamma_1)}+ \|\zeta_{2,h}\|^2_{L^2(\Gamma_2)}).
\end{array}
\end{eqnarray*}
It is straightforward to show that for $k=1,2$
\begin{equation}\label{eq:step3}
    1-C_{\delta,k}\, 
    \geq 
    \frac{\delta}{2(c_{L,k}\,c_{H,k}+\frac{\|\nu_k\|_\infty}{\underline\nu_k}\delta)}
    \geq \frac{\delta}{2(c_{LH}+\max_\ell\frac{\|\nu_\ell\|_\infty}{\underline\nu_\ell}\delta)},
\end{equation}
so that,
\begin{equation}\label{eq:below_Sigma_Gamma}
\zetabv^T \widetilde{\mathsf\Sigma} \, \zetabv \geq \frac{\delta}{2(c_{LH}+\max_\ell\frac{\|\nu_\ell\|_\infty}{\underline\nu_\ell}\delta)} \,
\zetabv^T  \mathsf M_{\Gamma} \, \zetabv.
\end{equation}
Since
\begin{equation*}
    \zetabv^T \widetilde{\mathsf\Sigma}_s \, \zetabv=\zetabv^T \widetilde{\mathsf\Sigma} \, \zetabv   \qquad \forall \zetabv \in \mathbb R^n,
\end{equation*}
and applying (\ref{eq:below_Sigma_Gamma}) and \eqref{eq:eig_mass_general}, there holds
\begin{eqnarray*}
    \frac{\zetabv^T\widetilde{\mathsf \Sigma}_s\, \zetabv}{\zetabv^T\zetabv} =
    \frac{\zetabv^T\widetilde{\mathsf \Sigma}\,\zetabv}{\zetabv^T\mathsf M_\Gamma\, \zetabv}\,
    \frac{\zetabv^T\mathsf M_\Gamma\,\zetabv}{\zetabv^T\zetabv} \geq \frac{\widehat{c}_1\,\delta}{2\, (c_{LH}+\max_\ell\frac{\|\nu_\ell\|_\infty}{\underline\nu_\ell}\delta)} \,\frac{h}{p^2}.
\end{eqnarray*}
We finally conclude that
\begin{equation*}
 \inf_{\mathsf 0\neq \zetabv\in \mathbb R^n}
\frac{\zetabv^T \widetilde {\mathsf \Sigma}_s \, \zetabv}{\zetabv^T\zetabv}\geq
\frac{\widehat{c}_1\,\delta}{2\, (c_{LH}+\max_\ell\frac{\|\nu_\ell\|_\infty}{\underline\nu_\ell}\delta)} \,\frac{h}{p^2}.
\end{equation*}
\end{proof}

As a consequence of Theorem \ref{thm:lowerbound_min_eig}, we can conclude that the weak Schur complement matrix $\widetilde{\mathsf \Sigma}$ is positive real according to the following Definition \ref{def:A_real_posdef}.

\begin{definition}\label{def:A_real_posdef}
\cite{saad_schultz} A real matrix $\mathsf A\in\mathbb R^{n\times n}$ is said \emph{positive real} if its symmetric part $\mathsf A_s=\frac{1}{2} (\mathsf A+\mathsf A^T)$ is positive definite.
\end{definition}

\subsection{Numerical verification of the theoretical estimates}\label{sec:numericalValidation}

In this section, we numerically verify  the estimates of Theorems \ref{thm:upperbound_max_eig} and \ref{thm:lowerbound_min_eig}. We consider the operator $L$ as in \eqref{eq:operatorL} in the domain $\Omega=(0,2)\times(0,1)$ and homogeneous Dirichlet conditions on the boundary $\partial\Omega$ and with different choices for the physical parameters $\nu$ and $\gamma$. In Sect. \ref{sec:numericalAnalysisConforming}, we first consider the case of conforming discretizations in the overlap studied in Sect. \ref{sec:analysisSchur}, for which we set $\Omega_1=(0,1+\delta/2)\times(0,1)$ and $\Omega_2=(1-\delta/2,2)\times(0,1)$, where $\delta$ will be specified later. In Sect. \ref{sec:numericalAnalysisNonConforming}, we explore numerically the case of non-conforming discretizations using a different decomposition as specified later.

\subsubsection{Numerical tests for conforming discretizations}\label{sec:numericalAnalysisConforming}


\paragraph{Test \#1. Continuous coefficients and conforming discretizations on the overlap.} 
We consider the differential operator \eqref{eq:operatorL} with $\nu=\gamma=1$ 
and three conforming discretizations: FEM $\mathbb P_1-\mathbb P_1$, FEM $\mathbb P_3-\mathbb P_3$, and SEM $\mathbb Q_p-\mathbb Q_p$, that we name \emph{Test 1A}, \emph{Test 1B}, and \emph{Test 1C}, respectively. 
We introduce structured meshes and denote by $h_x$ and $h_y$ the lengths of the element sides along the $x$ and $y$ directions, respectively. 

First, we want to assess the behavior of $\|\widetilde{\mathsf \Sigma}\|_2$ and $\lambda_{min}(\widetilde{\mathsf \Sigma}_s)$  versus the overlap width $\delta$. In Tests 1A and 1B, the mesh is obtained by taking $h_y=0.04$ and 
\begin{eqnarray*}
    h_x=\left\{\begin{array}{ll}
    0.04 & \mbox{in the subregions with }x\in (0,0.8)\cup(1.2,2)\\
    0.02 & \mbox{in the subregions with }x\in (0.8,0.94)\cup(1.08,1.2)\\
    0.002 & \mbox{in the subregion with }x\in (0.94,1.08),
    \end{array}\right.
\end{eqnarray*}
while in Test 1C, $p=6$, $h_y=0.1$, and 
\begin{eqnarray*}
    h_x=\left\{\begin{array}{ll}
    \frac{1-\delta/2}{10} & \mbox{in }\Omega\setminus (\Omega_1 \cap \Omega_2) \\
    \delta & \mbox{in } \Omega_1 \cap \Omega_2.
    \end{array}\right.
\end{eqnarray*}
Numerical results shown in the top row of Fig. \ref{fig:test1} confirm the theoretical estimates of Theorems \ref{thm:upperbound_max_eig} and \ref{thm:lowerbound_min_eig}, that is, $\|\widetilde{\mathsf \Sigma}\|_2$ is bounded independently of $\delta$, while $\lambda_{min}(\widetilde{\mathsf \Sigma}_s)=\mathcal O(\delta)$ when $\delta\to 0$.
\begin{figure}
    \centering
    \includegraphics[width=\linewidth]{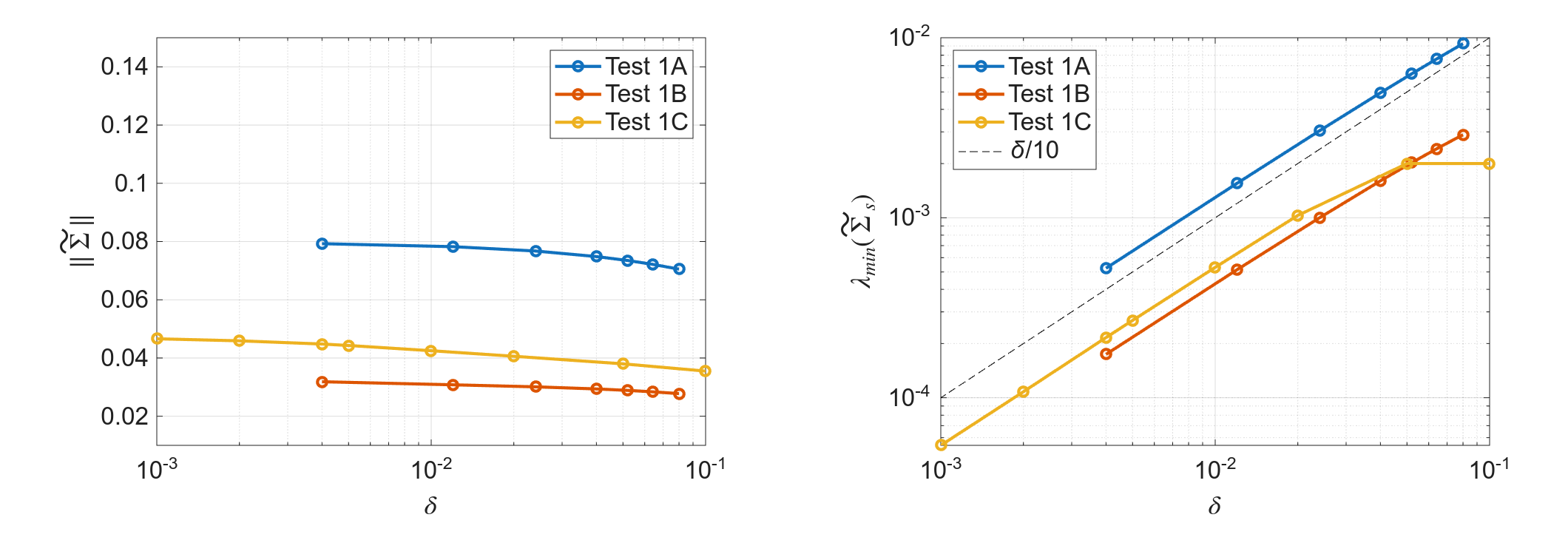}
    \includegraphics[width=\linewidth]{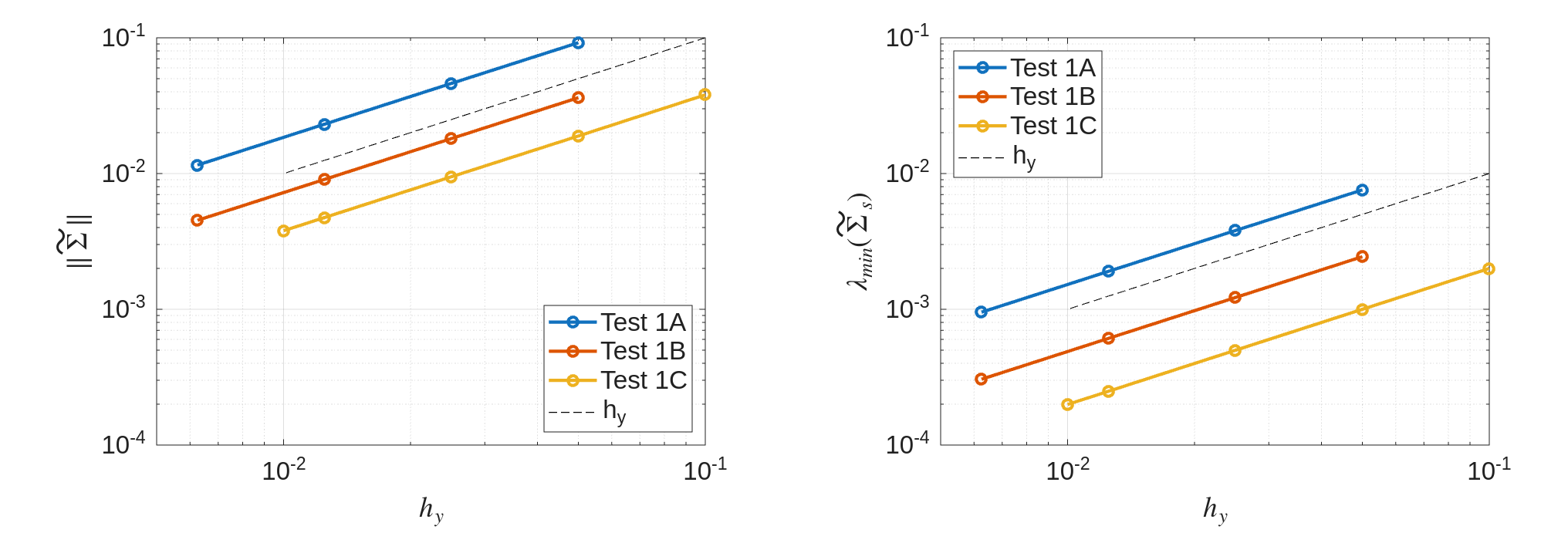}
    \includegraphics[width=\linewidth]{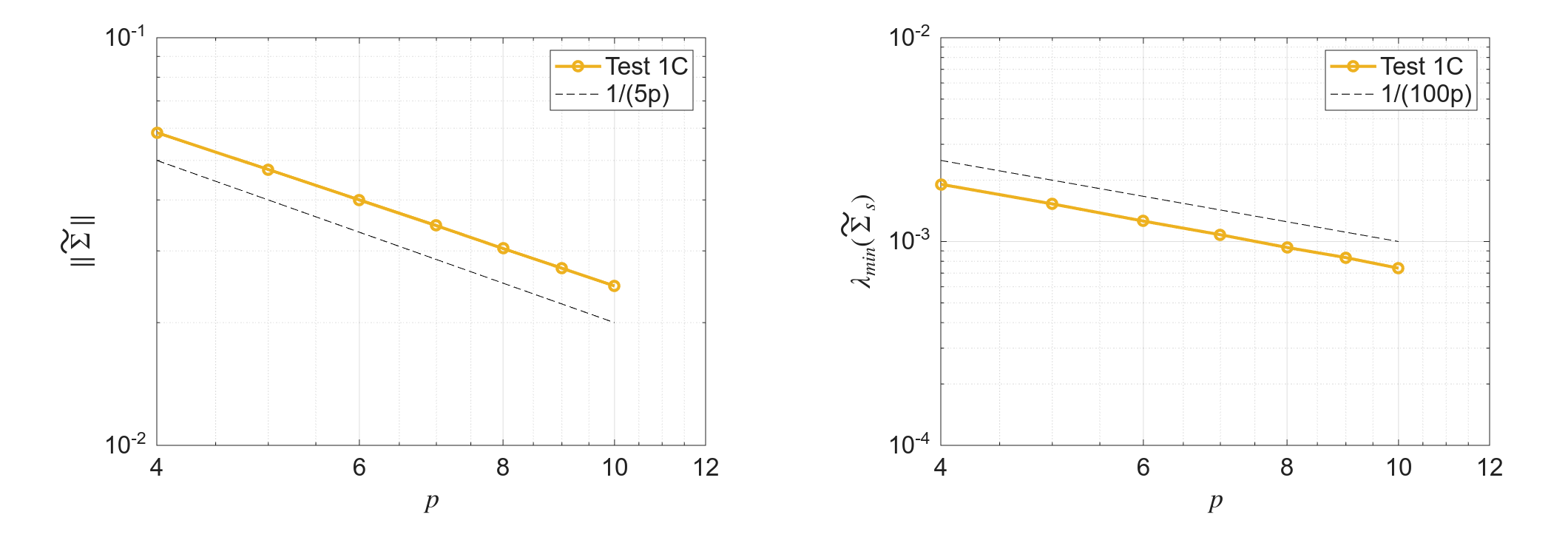}
    \caption{\emph{Test \#1.} Conforming discretizations on the overlap and continuous coefficients. $\|\widetilde{\mathsf \Sigma}\|_2$ (left) and $\lambda_{min}(\widetilde{\mathsf \Sigma}_s)$ (right) versus the overlap width $\delta$ (top), the mesh-size $h_y$ (center), and the polynomial degree $p$ (bottom). Test 1A, Test 1B, and Test 1C refer to FEM $\mathbb P_1-\mathbb P_1$, FEM $\mathbb P_3-\mathbb P_3$, and SEM $\mathbb Q_p-\mathbb Q_p$ discretizations, respectively. In the analysis versus $\delta$ and $h_y$ (top two rows), we set $p=6$ for Test 1C.}
    \label{fig:test1}
\end{figure}

Then, we evaluate $\|\widetilde{\mathsf \Sigma}\|_2$ and $\lambda_{min}(\widetilde{\mathsf \Sigma}_s)$ versus the mesh size. We note that both quantities are independent of $h_x$, so we do not report graphs in this case. Instead, when $h_y$ varies, both $\|\widetilde{\mathsf \Sigma}\|_2$ and $\lambda_{min}(\widetilde{\mathsf \Sigma}_s)$ behave proportionally to $h_y$, see Fig. \ref{fig:test1}, center row. This validates the theoretical estimates of Theorems \ref{thm:upperbound_max_eig} and \ref{thm:lowerbound_min_eig}.
In the latter tests, for Test 1A and Test 1B, we have considered $\delta=0.05$ and $h_x=0.025$, while for Test 1C, $\delta=0.05$, $p=6$, and $h_x \approx 0.05$. 

Finally, in Fig. \ref{fig:test1}, bottom, we plot $\|\widetilde{\mathsf \Sigma}\|_2$ and $\lambda_{min}(\widetilde{\mathsf \Sigma}_s)$ versus the polynomial degree $p$ for Test 1C. In this case, we have considered $\delta=0.025$, $h_y=0.1$, $h_x=(1-\delta/2)/10$ outside the overlap region, and $h_x=\delta$ in the overlap region.
Numerical results confirm Theorems \ref{thm:upperbound_max_eig} and \ref{thm:lowerbound_min_eig} also in this case, although in a less stringent way. Indeed, from (\ref{eq:upperbound_Sigmatilde})--(\ref{eq:overlineCp}), we have $\|\widetilde{\mathsf \Sigma}\|_2\leq h\overline c(p)$ with $\overline c(p)=\mathcal O(p^{-1/2})$ when $p\to\infty$. Numerical results show that $\|\widetilde{\mathsf \Sigma}\|_2\sim p^{-1}\leq p^{-1/2}$.
As concerns $\lambda_{min}(\widetilde{\mathsf \Sigma}_s)$, we obtain  $\lambda_{min}(\widetilde{\mathsf \Sigma}_s)\sim \frac{c}{p}\geq \frac{c}{p^2}$.

\medskip

\paragraph{Test \#2. Discontinuous coefficient $\nu$ and conforming discretizations on the overlap.}  
In this test, we consider $\gamma = 1$ and
\begin{eqnarray}\label{eq:nu_disc}
\nu=\left\{\begin{array}{ll}
1 & x\leq 1,\\
10^{\alpha} & x>1, \quad \text{with } \alpha\in\{0,\pm 1,\pm 2,\pm 3\},
\end{array}\right.
\end{eqnarray}
and we consider the same discretizations as in Tests 1A and 1C.

\smallskip

In the analysis versus $\delta$, we consider FEM $\mathbb{P}_1 - \mathbb{P}_1$ like in Test 1A. The graphs in the top row of Fig. \ref{fig:test2} show that $\|\widetilde{\mathsf \Sigma}\|_2$ remains bounded even though it mildly depends on the ratio $\|\nu\|_\infty/\underline{\nu}$. Regarding $\lambda_{min}(\widetilde{\mathsf \Sigma}_s)$, we see that the larger the ratio $\|\nu\|_\infty/\underline{\nu}$ the weaker the dependence on $\delta$.  
As we will see in Sect. \ref{sec:numerical_results} when analyzing GMRES iterations, the behavior of $\lambda_{min}(\widetilde{\mathsf \Sigma}_s)$ positively impacts the convergence of ICDD when the coefficient $\nu$ has strong variations. More precisely, the number of ICDD iterations turns out to be mildly dependent on $\delta$.

\smallskip

In the analysis versus $h_y$, we consider the same configuration as in Test 1A.
The graphs in the middle row of Fig. \ref{fig:test2} show that $\|\widetilde{\mathsf \Sigma}\|_2$ mildly depends on $h$. As concerns $\lambda_{min}(\widetilde{\mathsf \Sigma}_s)$, the larger the ratio $\|\nu\|_\infty/\underline{\nu}$, the larger $\lambda_{min}(\widetilde{\mathsf \Sigma}_s)$. Moreover, when $\alpha=\pm 3$, $\lambda_{min}(\widetilde{\mathsf \Sigma}_s)$ depends on $\delta$ more weakly than in the other cases.

\smallskip

Finally, in the analysis versus $p$, we consider SEM $\mathbb Q_p$ discretization as in Test 1C, with $h_x=\delta=0.025$, in the overlap, $h_x=(1-\frac{\delta}{2})/10$, outside the overlap, and $h_y=0.1$.
The graphs in the bottom row of Fig. \ref{fig:test2} show that $\|\widetilde{\mathsf \Sigma}\|_2$ mildly depends on $\|\nu\|_\infty/\underline{\nu}$ and behaves like $p^{-1}$, while
$\lambda_{min}(\widetilde{\mathsf \Sigma}_s)$ behaves like $p^{-2}$ when 
$\nu$ is discontinuous, confirming the theoretical estimate of Theorem \ref{thm:lowerbound_min_eig}. 

\begin{figure}
    \centering
    \includegraphics[width=\linewidth]{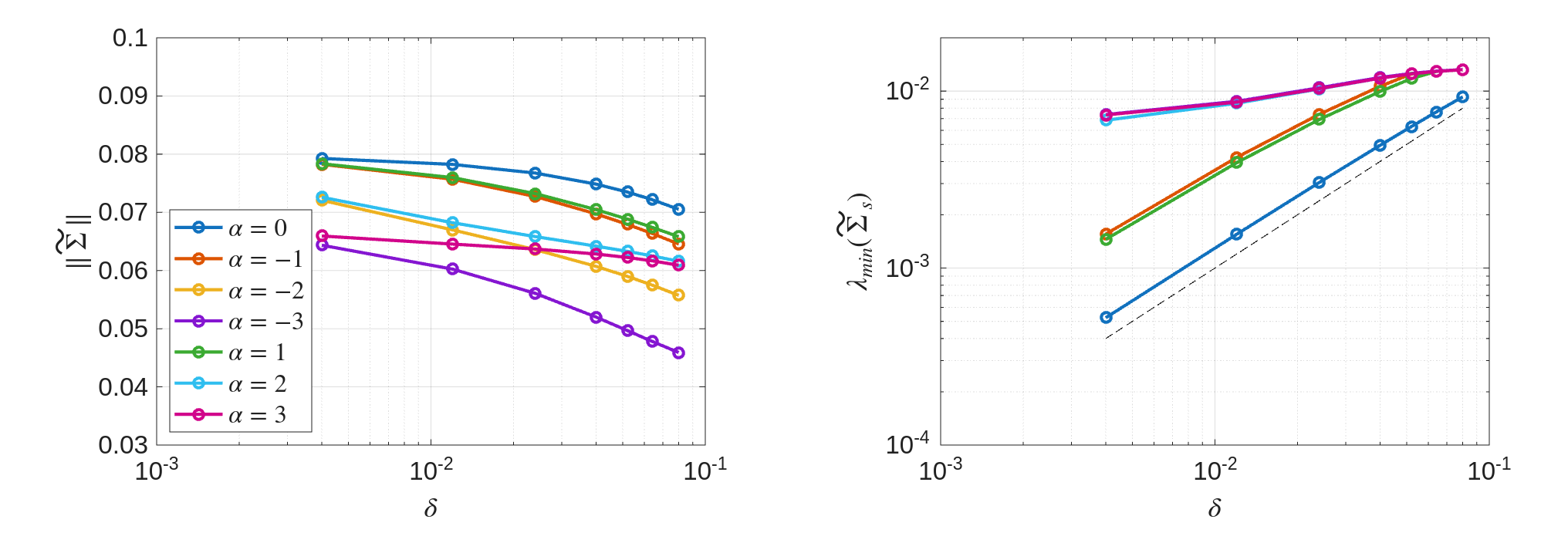}
    \includegraphics[width=\linewidth]{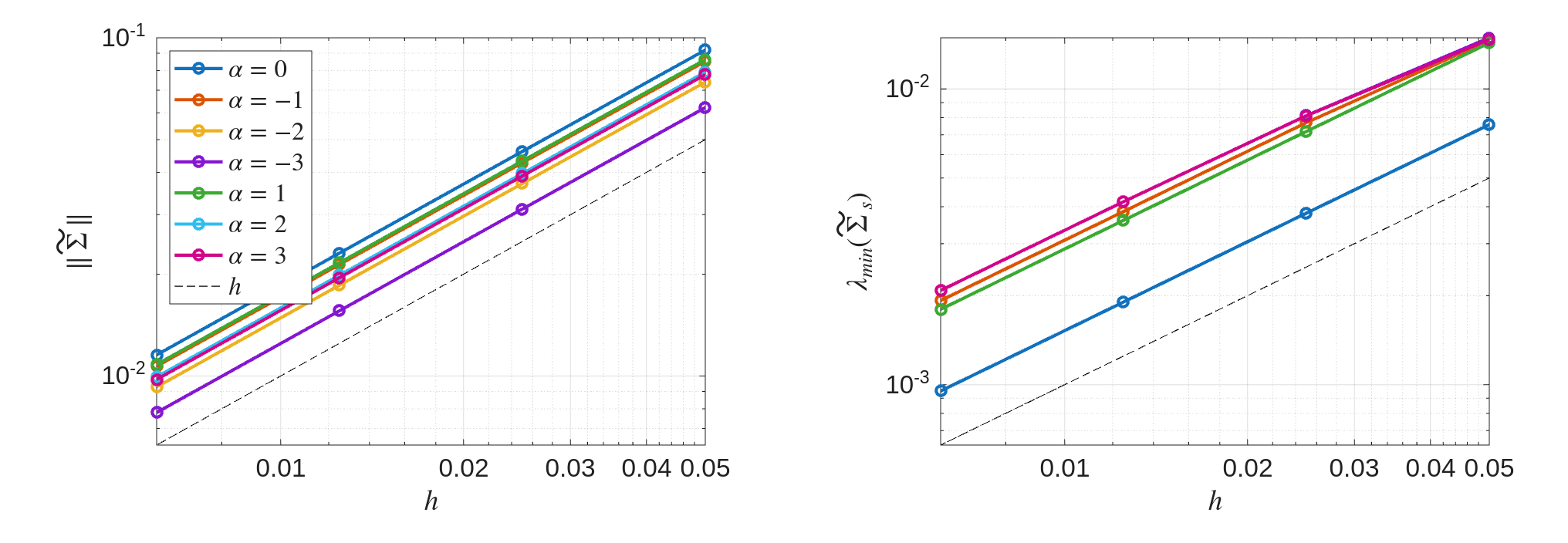}
    \includegraphics[width=\linewidth]{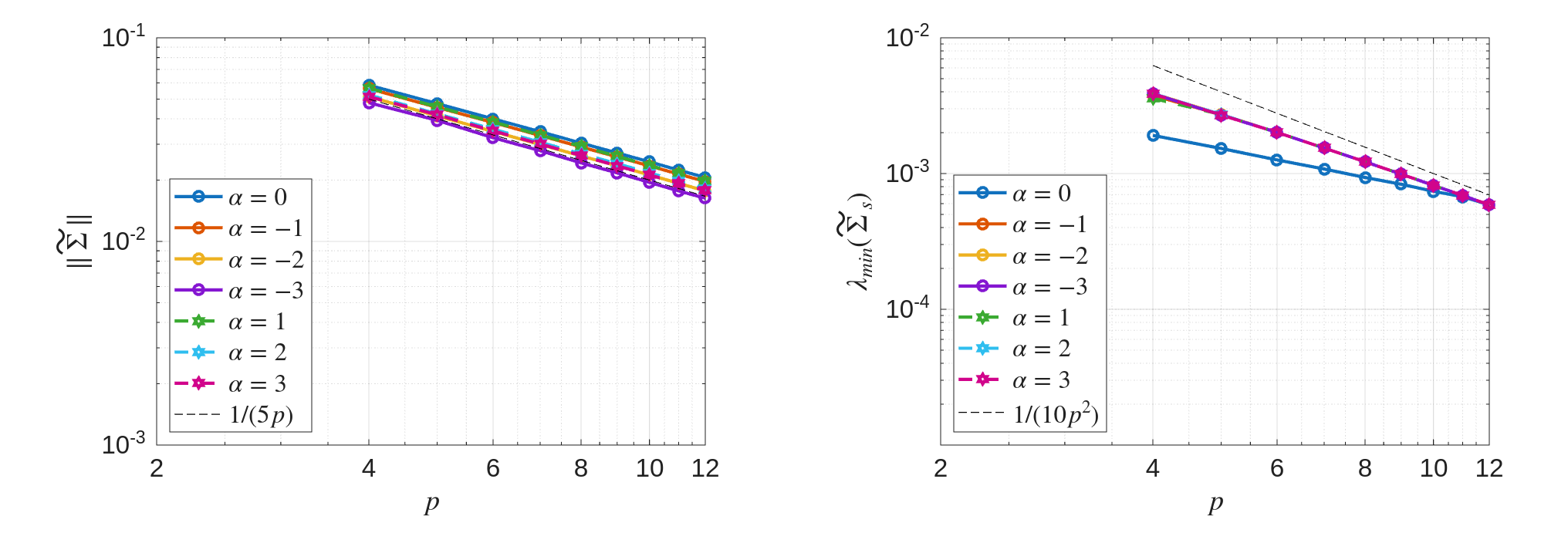}
    \caption{\emph{Test \#2.} Discontinuous coefficient $\nu$ (see (\ref{eq:nu_disc})) and conforming discretizations on the overlap.  $\|\widetilde{\mathsf \Sigma}\|_2$ (left) and $\lambda_{min}(\widetilde{\mathsf \Sigma}_s)$ (right) versus the overlap width $\delta$ (top), the mesh size $h$ (middle), and the polynomial degree $p$ (bottom).}
    \label{fig:test2}
\end{figure}

\medskip

\paragraph{Test \#3. Discontinuous coefficient $\gamma$ and conforming discretizations on the overlap.} In this test, $\nu = 1$ and
\begin{eqnarray}\label{eq:gamma_disc}
\gamma=\left\{\begin{array}{ll}
1 & x\leq 1,\\
10^{\alpha} & x>1, \quad \text{with } \alpha\in\{0,\pm 1,\pm 2, \pm 3\},
\end{array}\right.
\end{eqnarray}
and we use the same discretization of Test 1C (SEM $\mathbb Q_p$ in both subdomains).

In Fig. \ref{fig:test3}, we show $\|\widetilde{\mathsf \Sigma}\|_2$ and 
$\lambda_{min}(\widetilde{\mathsf \Sigma}_s)$ versus the overlap width $\delta$ (top), versus $h$ (middle), and versus $p$ (bottom). They almost behave as in Test \#2, with $\lambda_{min}(\widetilde{\mathsf \Sigma}_s)$ less sensitive to the ratio $\|\gamma\|_\infty/\underline{\gamma}$ than to $\|\nu\|_\infty/\underline{\nu}$.
 
\begin{figure}
    \centering
    \includegraphics[width=\linewidth]{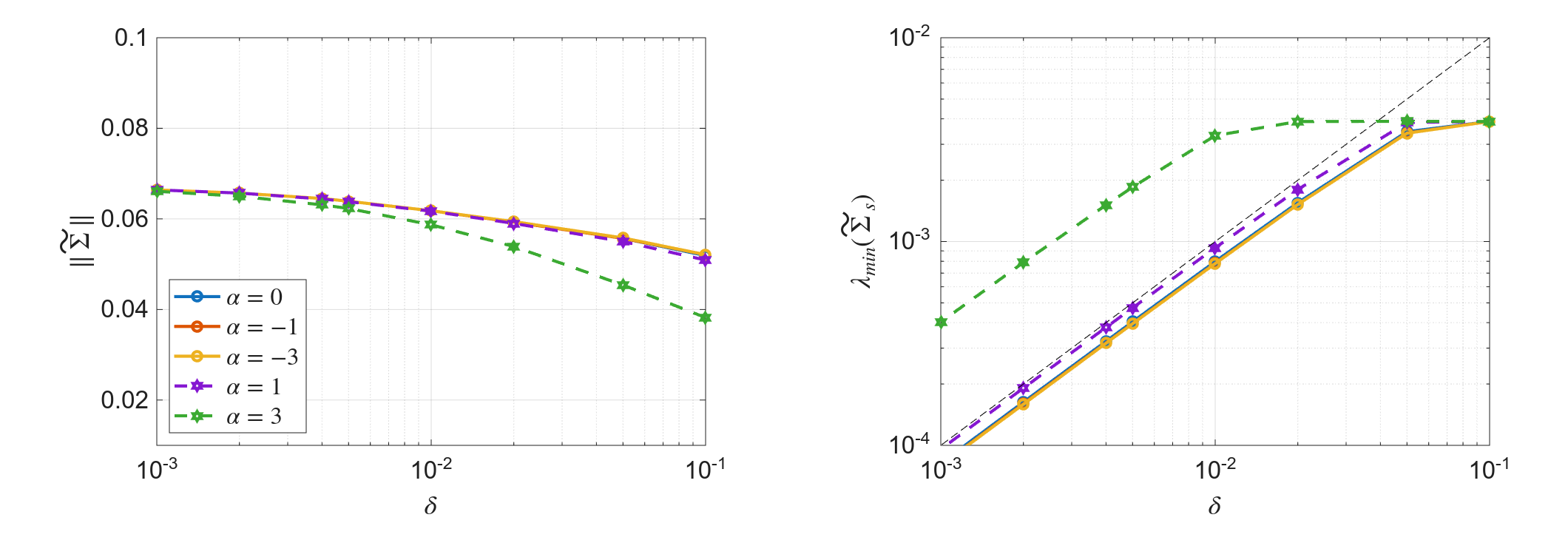}
    \includegraphics[width=\linewidth]{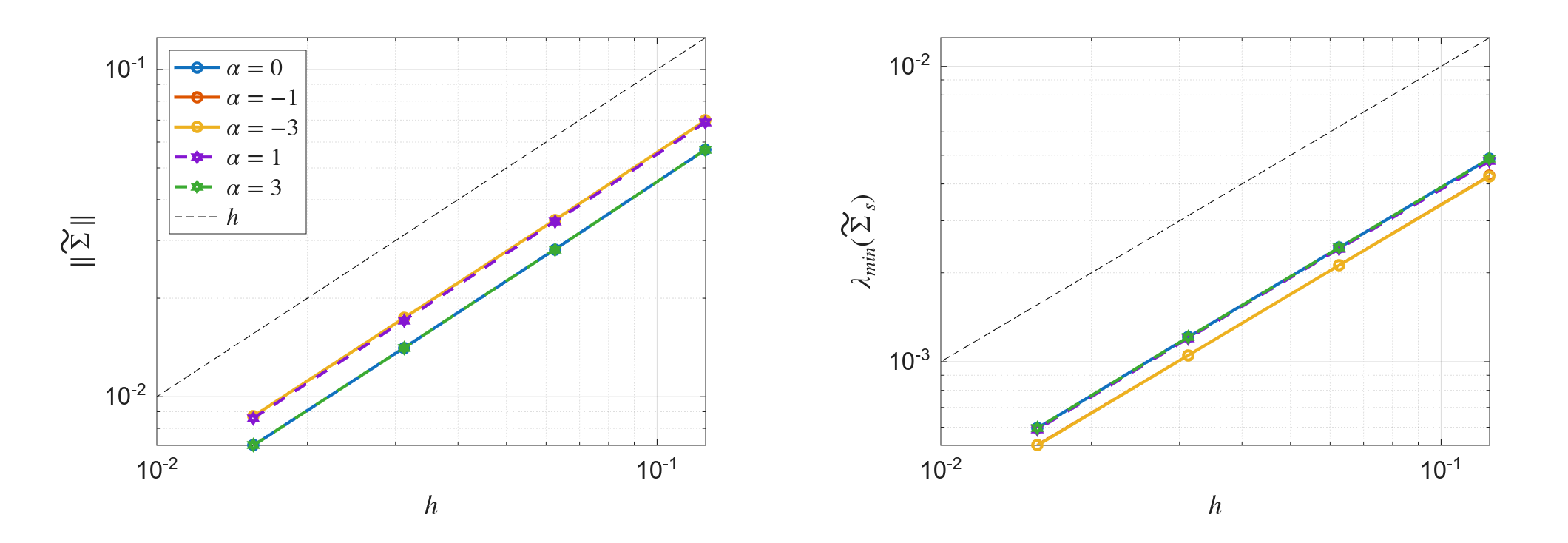}
    \includegraphics[width=\linewidth]{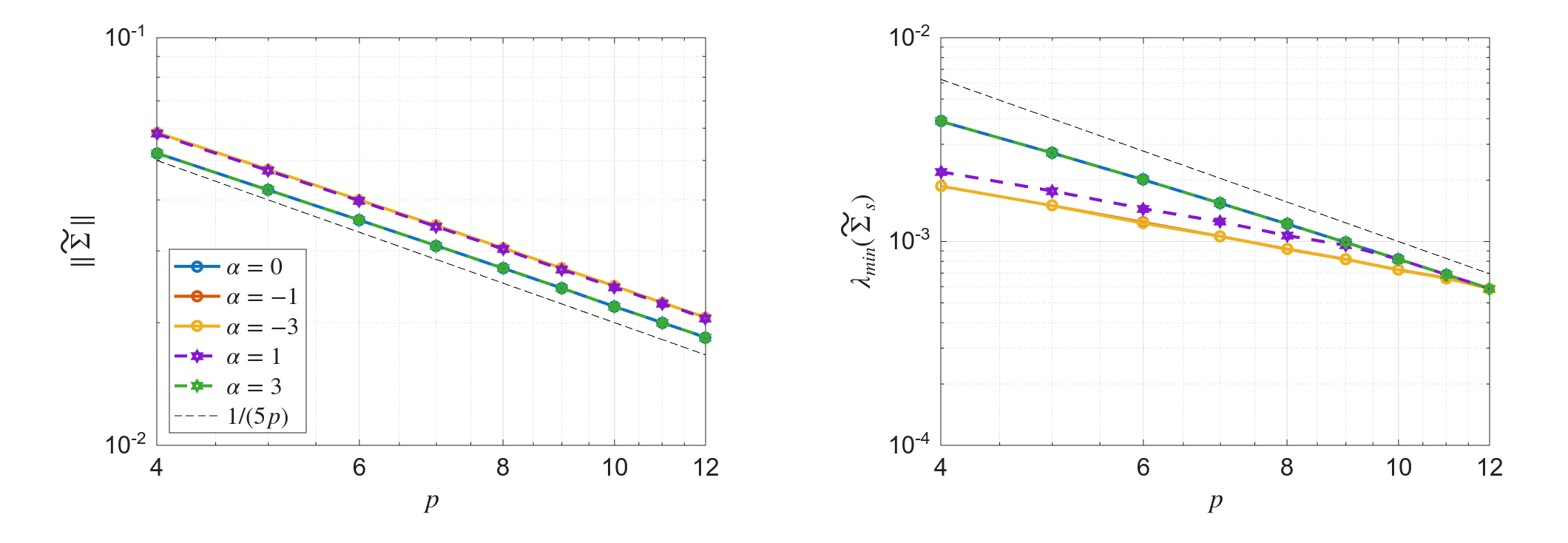}
    \caption{\emph{Test \#3.} Discontinuous coefficient $\gamma$ (see (\ref{eq:gamma_disc})) and conforming discretizations on the overlap.  $\|\widetilde{\mathsf \Sigma}\|_2$ (left) and $\lambda_{min}(\widetilde{\mathsf \Sigma}_s)$ (right) versus the overlap width $\delta$ (top), the mesh-size $h$ (center), the polynomial degree $p$ (bottom). In the analysis versus $\delta$ and $h$ (top two rows), we set $p=4$.}
    \label{fig:test3}
\end{figure}

\subsubsection{Numerical tests for non-conforming discretizations in the overlap}\label{sec:numericalAnalysisNonConforming}

Assumption \ref{ass_conforming} is crucial to guarantee that $\widetilde{\mathsf{\Sigma}}$ is positive real. Indeed, while for non-conforming discretizations in the overlap one can still prove a bound similar to \eqref{eq:upperbound_Sigmatilde} for the norm $\|\widetilde{\mathsf{\Sigma}}\|_2$, it is no longer possible to guarantee that $\lambda_{min} ( \widetilde{\mathsf{\Sigma}}_s ) > 0$ as we explain in Appendix \ref{app:lambdamin}.

However, although essential for the analysis, as we saw in Sect. \ref{sec:discrete} mesh conformity is not needed to set up the ICDD method that can easily accommodate non-conforming meshes through suitable interpolation operators (or intergrid matrices at the algebraic level).

Therefore, for the sake of completeness, in this section, we show the behavior of $\|\widetilde{\mathsf{\Sigma}}\|_2$ and $\lambda_{min} (\widetilde{\mathsf{\Sigma}}_s )$ for some non-conforming settings. This numerical investigation will be useful to interpret the corresponding convergence results in Sect. \ref{sec:numerical_results}.

All tests presented in this section consider the differential operator \eqref{eq:operatorL} with $\nu = \gamma = 1$ and homogeneous Dirichlet boundary conditions in the domain $\Omega=(0,2)\times(0,1)$.

\medskip

\paragraph{Test \#4. Continuous coefficient, non-conforming discretization on the overlap.} 
Different types of non-conformity are considered.

\smallskip

In \emph{Test 4A}, we set $\Omega_1=(0,1+\delta/2)\times(0,1)$ and $\Omega_2=(1-\delta/2,2)\times(0,1)$ with $\delta$ that will be specified later. Then we consider two geometrically conforming FEM discretizations, but with different polynomial degrees: $\mathbb P_1$ in $\Omega_1$ and $\mathbb P_2$ in $\Omega_2$. The mesh sizes $h_x$ and $h_y$ are those used in Test 1A.

\smallskip

In \emph{Test 4B}, we set $\Omega_1=(0,1+\delta)\times(0,1)$, $\Omega_2=(1,2)\times(0,1)$, and we consider FEM $\mathbb Q_1$ in $\Omega_1$ and $\mathbb P_1$ in $\Omega_2$, with uniform $h$ in both subdomains. We take $h\approx 0.025$ when we analyze the behavior versus $\delta$, and $\delta=0.025$ for the analysis versus $h$.

\smallskip

In \emph{Test 4C}, we consider the same domain decomposition as in Test 4B, but with $\mathbb Q_2$ in $\Omega_1$. As in Test 4B, we take $\mathbb P_1$ elements in $\Omega_2$, uniform $h$ in both subdomains, and we set $h\approx 0.025$ for the analysis versus $\delta$, and $\delta=0.025$ for the analysis versus $h$.

\smallskip

Finally, in \emph{Test 4D}, we take $\Omega_1=(0,1+\delta/2)\times(0,1)$, $\Omega_2=(1-\delta/2,2)\times(0,1)$, and SEM $\mathbb Q_p$ in $\Omega_1$ and $\mathbb Q_{p+1}$ in $\Omega_2$, with uniform meshes in both subdomains, but with different mesh sizes $h_1$ and $h_2$.
In the analysis versus $\delta$, we consider SEM $\mathbb Q_4$ with $h_1=1/10$ in $\Omega_1$ and $\mathbb Q_5$ with $h_2=1/12$ in $\Omega_2$.
For the analysis versus $h$, we consider $\delta=0.05$, SEM $\mathbb Q_4$ with $h_1=1/ne$ in $\Omega_1$ and $\mathbb Q_5$ with $h_2=1/(ne+2)$ in $\Omega_2$, and varying $ne$.
Finally, for the analysis versus $p$, we consider $\delta=0.025$, SEM $\mathbb Q_p$ with $h_1=1/10$ in $\Omega_1$ and $\mathbb Q_{p+1}$ with $h_2=1/12$ in $\Omega_2$.

\smallskip

Note that both geometries and discretizations are non-conforming in Tests 4B, 4C, and 4D.

\smallskip 

In Fig. \ref{fig:test4} top, we show the behavior of $\|\widetilde{\mathsf \Sigma}\|_2$ and 
$\lambda_{min}(\widetilde{\mathsf \Sigma}_s)$ versus the overlap width $\delta$. Even though Theorems \ref{thm:upperbound_max_eig} and \ref{thm:lowerbound_min_eig} are proved only in the conforming setting, numerical results show that $\|\widetilde{\mathsf \Sigma}\|_2$ is bounded independently of $\delta$ and $\lambda_{min}(\widetilde{\mathsf \Sigma}_s)=\mathcal O(\delta)$, in this configuration too.

For what concerns the behavior with respect to $h$, ICDD behaves as in the conforming case in Tests 4A and 4D, whereas the behavior differs in Tests 4B and 4C. Indeed, in the latter cases, both $\lambda_{min}(\widetilde{\mathsf{\Sigma}}_s)$ and $\|\widetilde{\mathsf{\Sigma}}\|_2$ behave like $\mathcal{O}(h^q)$ with the same $q>1$. However, this will not impact the convergence rate of GMRES because, as we will see in Theorem \ref{thm:convergenceGMRES_weak} this depends on the ratio between the two quantities.

We remark that in Tests 4B and 4C the eigenvalues $\lambda_{min} (\widetilde{\mathsf{\Sigma}}_s)$ become negative for some value of $h$ as we have showed in Appendix \ref{app:lambdamin}. 

Finally, the behavior with respect to $p$ is as in the conforming case.

\begin{figure}
    \centering
    \includegraphics[width=\linewidth]{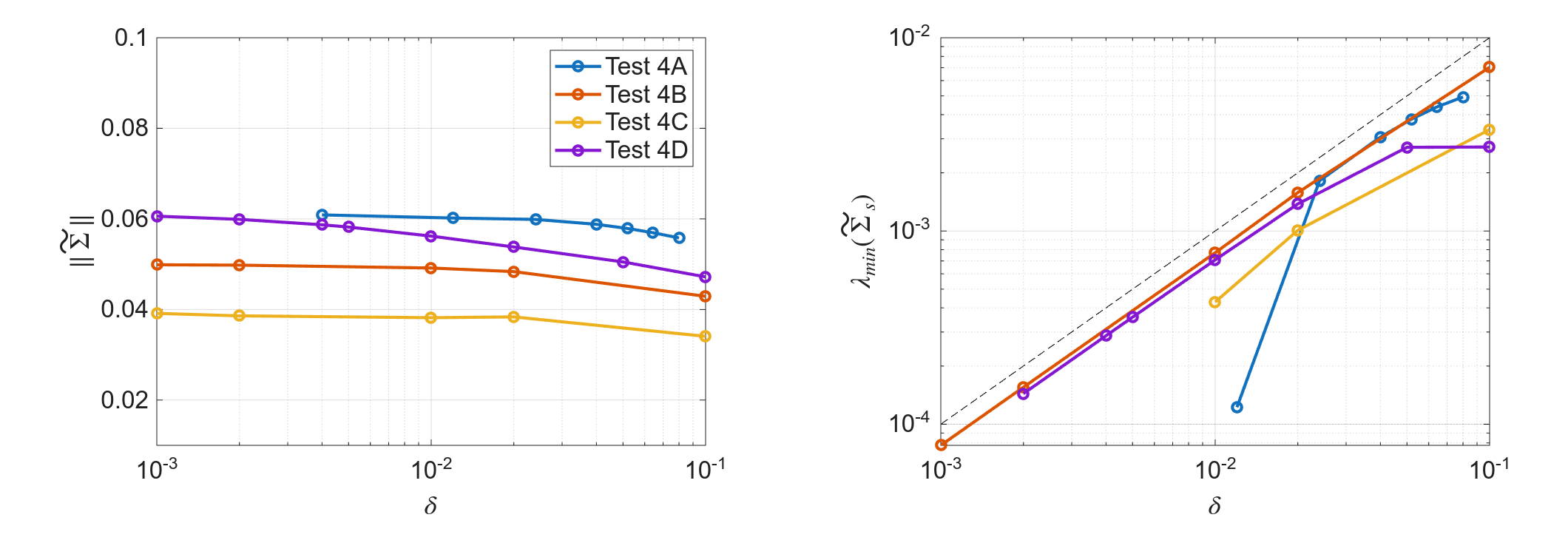}
    \includegraphics[width=\linewidth]{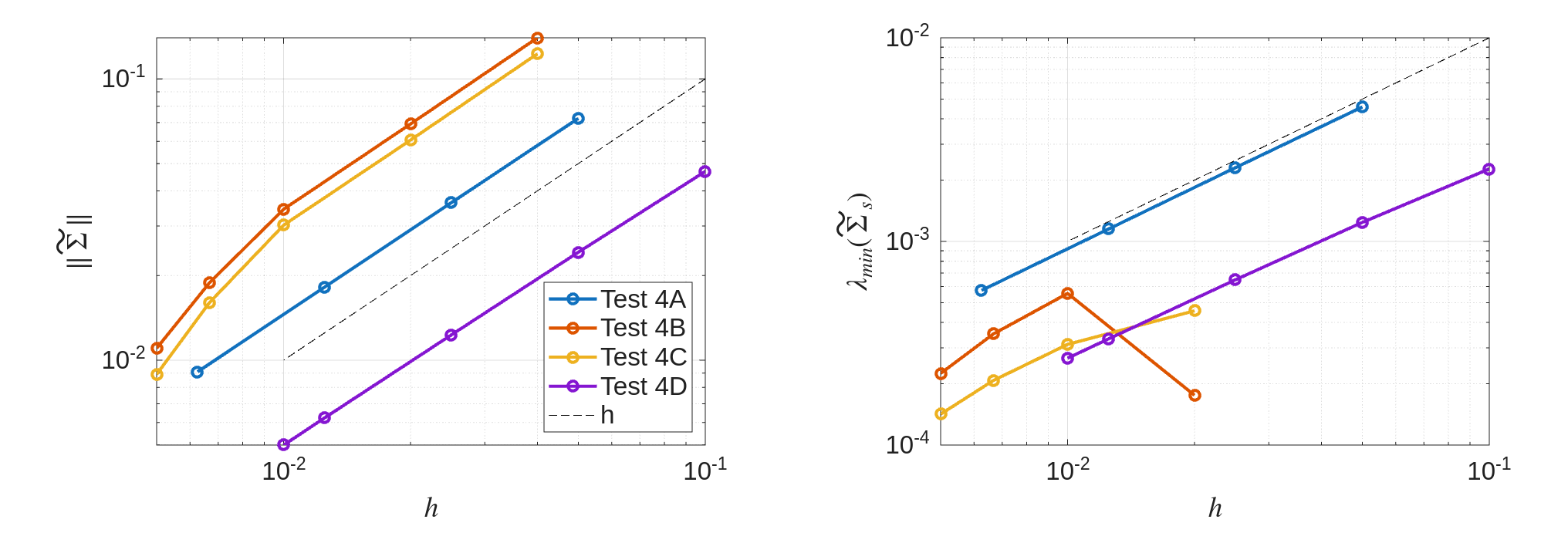}
    \includegraphics[width=\linewidth]{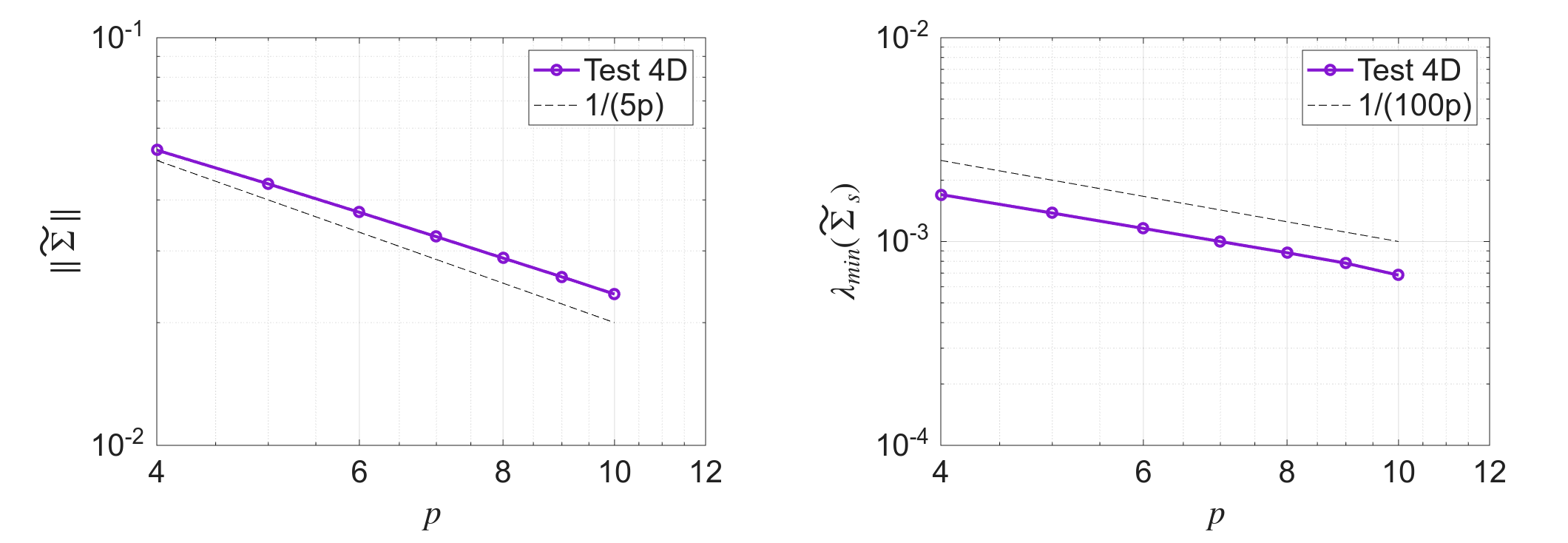}
    \caption{\emph{Test \#4.} Non-conforming discretizations on the overlap.  $\|\widetilde{\mathsf \Sigma}\|_2$ (left) and $\lambda_{min}(\widetilde{\mathsf \Sigma}_s)$ (right) versus the overlap width $\delta$ (top), the mesh-size $h=h_1$ (center), the polynomial degree $p$ (bottom). Test 4A, Test 4B, Test 4C, and Test 4D refer to FEM $\mathbb P_1-\mathbb P_2$, $\mathbb Q_1-\mathbb P_1$, $\mathbb Q_2-\mathbb P_1$, and SEM $\mathbb Q_p-\mathbb Q_{p+1}$ discretizations, respectively.}
    \label{fig:test4}
\end{figure}

\section{Convergence of GMRES iterations for the ICDD interface systems}\label{sec:convergenceGMRES}

In this section, we derive several estimates for the number of GMRES iterations needed to solve the Schur complement system associated with the ICDD method using theoretical results on the convergence of GMRES \cite{beckermann2006}. We begin by considering the weak Schur complement system \eqref{eq:schur_2_tilde}.

\begin{theorem}\label{thm:convergenceGMRES_weak}
Let Assumptions \ref{ass_mesh}--\ref{ass_conforming} be satisfied. Let  $\widetilde{\boldsymbol{r}}^{(m)}$ be the residual of the weak Schur complement system \eqref{eq:schur_2_tilde} at the $m$th step of GMRES iterations. Then, 
\begin{equation*}
\frac{\|\widetilde{\boldsymbol{r}}^{(m)}\|_2}{\|\widetilde{\boldsymbol{r}}^{(0)}\|_2}\leq \left( 2 + \frac{2}{\sqrt{3}} \right)(2+\gamma_\beta)\,\gamma_\beta^m,
\end{equation*}
where
\begin{equation*}
\cos(\beta) =\frac{\lambda_{min}(\widetilde{\mathsf \Sigma}_s)}{\|\widetilde{\mathsf \Sigma}\|_2} \qquad \mbox{ and }\qquad \gamma_\beta=2 \sin \left( \frac{\beta}{4 - 2\beta/\pi} \right) 
\end{equation*}
with $\beta\in[0,\pi/2)$, and $\|\cdot\|_2$ denotes the Euclidean norm.

Moreover, for $\delta \to 0$, the minimal number of iterations required for the relative residual to become smaller than an assigned tolerance $\epsilon$ can be estimated as
\begin{equation}\label{eq:estimate_iterations_weak}
m \geq \frac{3\sqrt{3}}{4} \,c_{LH} \, \frac{\overline{c}(p)}{\underline{c}(p)} \, \log \left( \frac{4(1+\sqrt{3})}{\epsilon} \right) \, \delta^{-1} \, ,
\end{equation}
where the constant $c_{LH}$ is defined in Theorem \ref{thm:upperbound_max_eig}, while $\overline{c}(p)$ and $\underline{c}(p)$, both independent of $\delta$ and $h$, are defined in Theorem \ref{thm:lowerbound_min_eig}.
\end{theorem}
\begin{proof}
The proof is based on Theorem 2.1 of \cite{beckermann2006} that we report in Appendix \ref{appendix:gmres} for the sake of clarity. The Schur complement matrix $\widetilde{\mathsf \Sigma}$ is positive real as a consequence of Theorem \ref{thm:lowerbound_min_eig}. Using \eqref{eq:upperbound_Sigmatilde} and \eqref{eq:lowerbound_min_eig}, we get
\begin{equation*}
    \cos(\beta) \geq \frac{\underline{c}(p)}{\overline{c}(p)} 
    \frac{\delta}{c_{LH} + \max_\ell\frac{\|\nu_\ell\|_\infty}{\underline\nu_\ell}\delta}\,.
\end{equation*}
Thus, $0 < \beta \leq \overline{\beta}$ with
\begin{equation*}
 \overline{\beta} = \arccos\left( \frac{\underline{c}(p)}{\overline{c}(p)} \frac{\delta}{c_{LH} + \max_\ell\frac{\|\nu_\ell\|_\infty}{\underline\nu_\ell}\delta} \right) ,
\end{equation*}
and
\begin{equation*}
    \gamma_\beta  \leq\gamma_{\overline \beta}= 2 \sin \left( \frac{\overline{\beta}}{4 - 2\overline{\beta}/\pi} \right).
\end{equation*}
We notice that $\gamma_{\overline \beta}<1$ because
$\overline{\beta}\in[0,\pi/2)$ implies $\overline\beta/(4-2\overline\beta/\pi)\in[0,\pi/6). $

Since the right hand side of (\ref{eq:res_gmres}) is an increasing function with respect to $\gamma_\beta$, (\ref{eq:res_gmres}) implies that
\begin{equation*}
    \frac{\|\boldsymbol{r}^{(m)}\|_2}{\|\boldsymbol{r}^{(0)}\|_2}\leq (2+2/\sqrt{3})(2+\gamma_{\overline\beta})\gamma_{\overline\beta}^m.
\end{equation*}
Thus, if 
\begin{equation*} 
         m\geq 
     \frac{\log\epsilon -\log((2+2/\sqrt{3})(2+\gamma_{\overline\beta}))} {\log\gamma_{\overline\beta}}
\end{equation*}
then $\frac{\|\boldsymbol{r}^{(m)}\|_2}{\|\boldsymbol{r}^{(0)}\|_2}\leq \epsilon$.

Then, \eqref{eq:estimate_iterations_weak} is obtained by applying the Taylor expansion to the right-hand side of the previous inequality around $\delta = 0$.

\end{proof}

Since $\overline c(p)\sim p^{-1/2}$ and $\underline c(p)\sim p^{-2}$, thus we deduce that the minimum number of iterations required by GMRES increases with $p$, asymptotically like $p^{3/2}$, while it is independent of the mesh size $h$.

\medskip

Numerical results showing the number of GMRES iterations versus $\delta$, $h$, and $p$ will be presented in Sect. \ref{sec:numerical_results}.

\bigskip

We consider now the case of the Schur complement system \eqref{eq:schur_2} associated with the strong form of the interface equations \eqref{eq:os2}$_3$.

\begin{corollary}\label{cor:iterations_strong}
   Let Assumptions \ref{ass_mesh}--\ref{ass_conforming} be satisfied and
    $\boldsymbol{r}^{(m)}$ be the residual of the strong Schur complement system \eqref{eq:schur_2} at the $m$th GMRES iteration. For a fixed tolerance $\epsilon$,when $\delta\to 0$, the number $m$ of GMRES iterations needed to guarantee that
    \begin{equation}\label{eq:bound_residual_strong}
        \frac{\|\boldsymbol{r}^{(m)}\|_2}{\|\boldsymbol{r}^{(0)}\|_2}
        \leq \epsilon
    \end{equation}
    is
    \begin{equation}\label{eq:estimate_iter_strong}
    m\geq\frac{3\sqrt{3}}{4}c_{LH} \frac{\overline c(p)}{\underline c(p)}\log\left(\frac{4\,\widehat{c}_2\,(1+\sqrt{3})\,p}{\widehat{c}_1\,\epsilon}\right)\delta^{-1}\,.
\end{equation}
\end{corollary}

\begin{proof}
Using the relationship \eqref{eq:sigma-chi-strong-weak}, there holds $\widetilde{\boldsymbol{r}}^{(m)} = \mathsf{M}_\Gamma \boldsymbol{r}^{(m)}$ or, equivalently, $\boldsymbol{r}^{(m)} = \mathsf{M}_\Gamma^{-1} \widetilde{\boldsymbol{r}}^{(m)}$ since $\mathsf{M}_\Gamma$ is symmetric and positive definite. Therefore, for any $m\geq 0$, $\|\boldsymbol{r}^{(m)}\|_2 = \|\mathsf{M}_\Gamma^{-1}\widetilde{\boldsymbol{r}}^{(m)}\|_2=\|\widetilde{\boldsymbol{r}}^{(m)}\|_{\mathsf{M}_\Gamma^{-2}}$, 
where, for any symmetric positive definite matrix $\mathsf{B} \in \mathbb{R}^{n_k \times n_k}$, $\|\zetav_k\|_{\mathsf{B}}$ denotes the matrix norm   $\|\zetav_k\|_{\mathsf{B}} = ( \zetav_k^T\mathsf{B}\,\zetav_k )^{1/2}$ associated with the scalar product $(\zetav_k,\zetav_k)_{\mathsf{B}}=(\mathsf{B}\zetav_k,\zetav_k)_2$, $(\cdot,\cdot)_2$ being the standard Euclidean scalar product.

For any $\zetav_k\in\mathbb R^{n_k}\setminus\ \{ \zerov \}$, as a consequence of Lemma \ref{lemma:eig_mass} and Assumptions \ref{ass_conforming} there holds
\begin{equation}\label{eq:lemma2_alg1_bis}
    \widehat{c}_1 \frac{h}{p^2} \|\zetav_k\|_2^2 \leq \| \zetav_k\|_{\mathsf{M}_{\Gamma_k}}^2 \leq \widehat{c}_2 \frac{h}{p} \| \zetav_k\|_2^2\,, \qquad k=1,2.
\end{equation}
By taking $\zetav_k=\mathsf M_{\Gamma_k}\etav_k$ in (\ref{eq:lemma2_alg1_bis}) we obtain
\begin{equation}\label{eq:lemma2_alg2_bis}
    \widehat{c}_1 \frac{h}{p^2} \|\etav_k\|_{\mathsf M_{\Gamma_k^{-1}}}^2 \leq \| \etav_k\|_{2}^2 \leq \widehat{c}_2 \frac{h}{p} \| \etav_k\|_{\mathsf M_{\Gamma_k^{-1}}}^2\,, \qquad k=1,2, 
\end{equation}
and similarly, by taking $\etav_k=\mathsf M_{\Gamma_k}\psibv_k$ in (\ref{eq:lemma2_alg2_bis}) we obtain
\begin{equation}\label{eq:lemma2_alg3_bis}
    \widehat{c}_1 \frac{h}{p^2} \|\psibv_k\|_{\mathsf M_{\Gamma_k^{-2}}}^2 \leq \| \psibv_k\|_{\mathsf M_{\Gamma_k^{-1}}}^2 \leq \widehat{c}_2 \frac{h}{p} \| \psibv_k\|_{\mathsf M_{\Gamma_k^{-2}}}^2\,, \qquad k=1,2. 
\end{equation}
Thus, by applying (\ref{eq:lemma2_alg1_bis}), (\ref{eq:lemma2_alg2_bis}), and (\ref{eq:lemma2_alg3_bis}), it holds 
\begin{equation*}
\frac{\|{\mathbf r}^{(m)}\|_2}{\|{\mathbf r}^{(0)}\|_2}=
\frac{\|\widetilde{\mathbf r}^{(m)}\|_{\mathsf M_{\Gamma}^{-2}}}{\|\widetilde{\mathbf r}^{(0)}\|_{\mathsf M_{\Gamma}^{-2}}}\leq \frac{\frac{1}{\hat c_1}\frac{p^2}{h}\|\widetilde{\mathbf r}^{(m)}\|_2}{\frac{1}{\hat c_2}\frac{p}{h}\|\widetilde{\mathbf r}^{(0)}\|_2}=
\frac{\hat c_2}{\hat c_1} p \frac{\|\widetilde{\mathbf r}^{(m)}\|_2}{\|\widetilde{\mathbf r}^{(0)}\|_2}.
\end{equation*}
Requiring 
\begin{equation*}
    \frac{\|\widetilde{\mathbf r}^{(m)}\|_2} {\|\widetilde{\mathbf r}^{(0)}\|_2}\leq \frac{\widehat c_1}{\widehat c_2}\frac{1}{p}\epsilon,
\end{equation*}
guarantees (\ref{eq:bound_residual_strong}) and 
provides the estimates (\ref{eq:estimate_iter_strong}) in view of Theorem \ref{thm:convergenceGMRES_weak}.
\end{proof}

As in the weak case, the number of iterations grows like $\delta^{-1}$ when $\delta$ vanishes, and is independent of the mesh size $h$. About the dependence on the polynomial degree $p$, the estimate (\ref{eq:estimate_iter_strong}) has a logarithmic factor more than (\ref{eq:estimate_iterations_weak}). However, numerical results of Section \ref{sec:numerical_results} show that the number of iterations is almost independent of the polynomial degree $p$ (see Figure \ref{fig:test1-4_iterations_p}).

\subsection{Numerical results}\label{sec:numerical_results}

We now assess numerically the theoretical results proved in Theorem \ref{thm:convergenceGMRES_weak} and Corollary \ref{cor:iterations_strong}, by showing the number of GMRES iterations required to solve the Schur complement system associated with ICDD. We also analyze the performance of the weak ICDD method (\ref{eq:schur_2_tilde}), the dual ICDD method (\ref{eq:schur_dual}), and the weak dual ICDD method, i.e., the variant of (\ref{eq:schur_dual}) where the interface conditions are imposed weakly. For the sake of clarity, we formulate the Schur complement systems associated with these four variants of ICDD as particular instances of the algebraic equation 
\begin{equation}\label{eq:ICDD_compact}
\mathsf{B} \, \mathsf{\Sigma} \, \boldsymbol{\lambda} = \mathsf{B} \, \boldsymbol{\chi}
\end{equation}
where $\mathsf\Sigma$ and $\boldsymbol  \chi$ are defined in (\ref{eq:schur_2_mat_rhs}), while $\mathsf B$ takes a different form depending on the method as defined in Table \ref{tab:icdd-compact-form}.

\begin{table}[h]
    \centering
\begin{tabular}{lcl}
Method         & $\mathsf{B}$& Equations \\
\hline
ICDD           & $\mathsf{I}$& (\ref{eq:schur_2})--(\ref{eq:schur_2_mat_rhs})\\
Weak ICDD      & $\mathsf{M}_\Gamma$& (\ref{eq:schur_2_tilde})--(\ref{eq:schur_2_mat_rhs_tilde})\\
Dual ICDD      & $2\mathsf{I} - \mathsf{\Sigma}$& (\ref{eq:schur_dual})--(\ref{eq:schur_dual_mat_rhs})\\
Weak dual ICDD & $\mathsf{M}_\Gamma ( 2 \mathsf{I} - \mathsf{\Sigma})$& \\
\hline
\end{tabular}
    \caption{Different variants of ICDD.}
    \label{tab:icdd-compact-form}
\end{table}

Unless otherwise specified, in the following, we consider problem \eqref{eq:globalProblem} in $\Omega=(0,2)\times(0,1)$ with right-hand side 
\begin{equation*}
    f(x,y)=\left\{\begin{array}{rl}
    -200 &\mbox{ in } (0,0.9]\times(0,0.4]\\
    0 &\mbox{ in } (0,0.9]\times(0.4,1)\\
    0 &\mbox{ in } (0.9,2)\times(0,0.4]\\
    200 &\mbox{ in } (0.9,2)\times(0.4,1).
    \end{array}
    \right.
\end{equation*}

The Schur complement systems \eqref{eq:ICDD_compact} are solved by GMRES without restart, with tolerance $\varepsilon=10^{-9}$ for the stopping test, and null initial guess.

\smallskip

We first consider the same settings as in Tests \#1--\#3 introduced in Sect. \ref{sec:numericalAnalysisConforming}.

\medskip

In Fig. \ref{fig:test1_iterations}, we report the number of GMRES iterations versus $\delta$ and $h_y$ required to solve the interface system \eqref{eq:ICDD_compact} considering the discretization of Test \#1. From the plots in the top row of Fig. \ref{fig:test1_iterations}, we see that the convergence rate of both ICDD and weak ICDD depends on $\delta$, but in a weaker way than established by Theorem \ref{thm:convergenceGMRES_weak} and Corollary \ref{cor:iterations_strong}, for all three different discretizations (Test 1A with $\mathbb P_1-\mathbb P_1$, Test 1B with $\mathbb P_3-\mathbb P_3$, and Test 1C with $\mathbb Q_6-\mathbb Q_6$). Both dual forms of ICDD require fewer iterations to converge, but, remembering that each iteration of the dual form costs about twice than one iteration of ICDD, the advantage of using the dual form is evident only when $\delta$ is very small. In particular, it is worth noting that the convergence rate of the dual weak ICDD is almost independent of $\delta$.

For the analysis with respect to $h_x$ and $h_y$, we confirm the theoretical result of Theorem \ref{thm:convergenceGMRES_weak} and Corollary \ref{cor:iterations_strong}, for which GMRES converges independently of the mesh size. For the sake of space, in the bottom row of Fig. \ref{fig:test1_iterations}, we only report the number of iterations versus $h_y$. 
Finally, the iterations versus $p$ for Test 1C are plotted in the left picture of Fig. \ref{fig:test1-4_iterations_p}. The convergence estimate of Theorem \ref{thm:convergenceGMRES_weak} about weak ICDD is confirmed, even though iterations grow a bit less than $p^{3/2}$. On the other hand, Corollary \ref{cor:iterations_strong} seems to overestimate the number of GMRES iterations for ICDD. As a matter of fact, in (\ref{eq:estimate_iter_strong}), $m \geq c p^{3/2}\log p$, while from  Fig. \ref{fig:test1-4_iterations_p}, we observe that $m$ is independent of $p$. This mismatch may be due to the fact that to prove (\ref{eq:estimate_iter_strong}) we need to perform the analysis in a weak setting and then return to the strong form (applying twice the norm equivalence (\ref{eq:eig_mass_general})), while at the algebraic level, we directly solve the Schur complement system in its strong form.

\begin{figure}
    \centering
    \includegraphics[width=\linewidth]{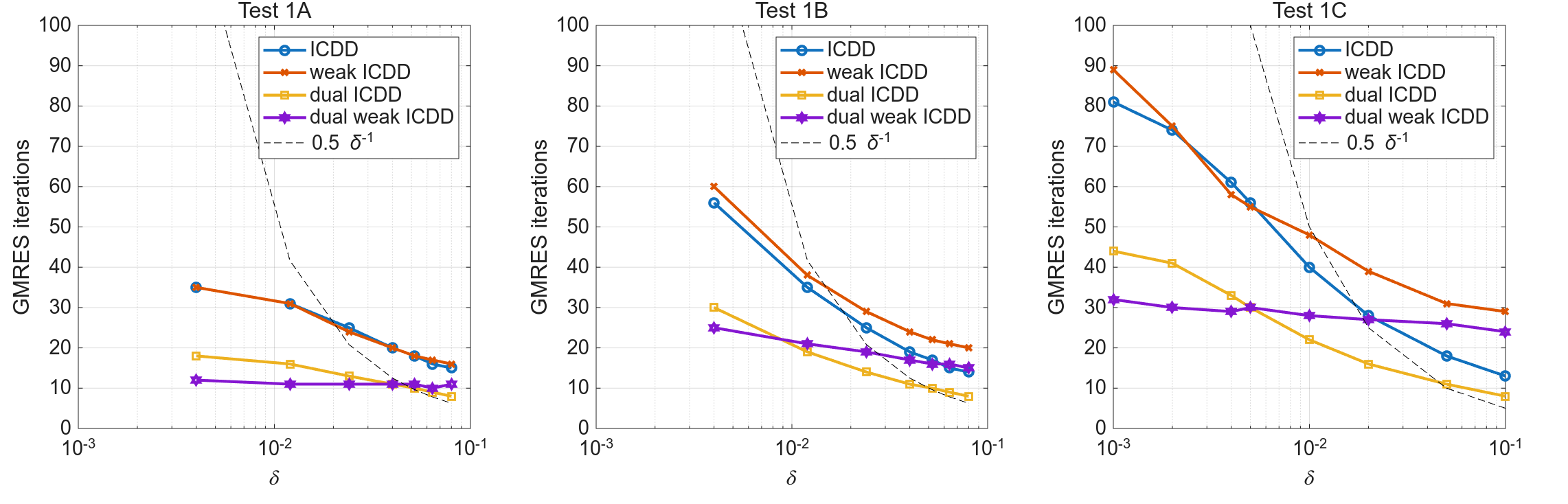}
    \includegraphics[width=\linewidth]{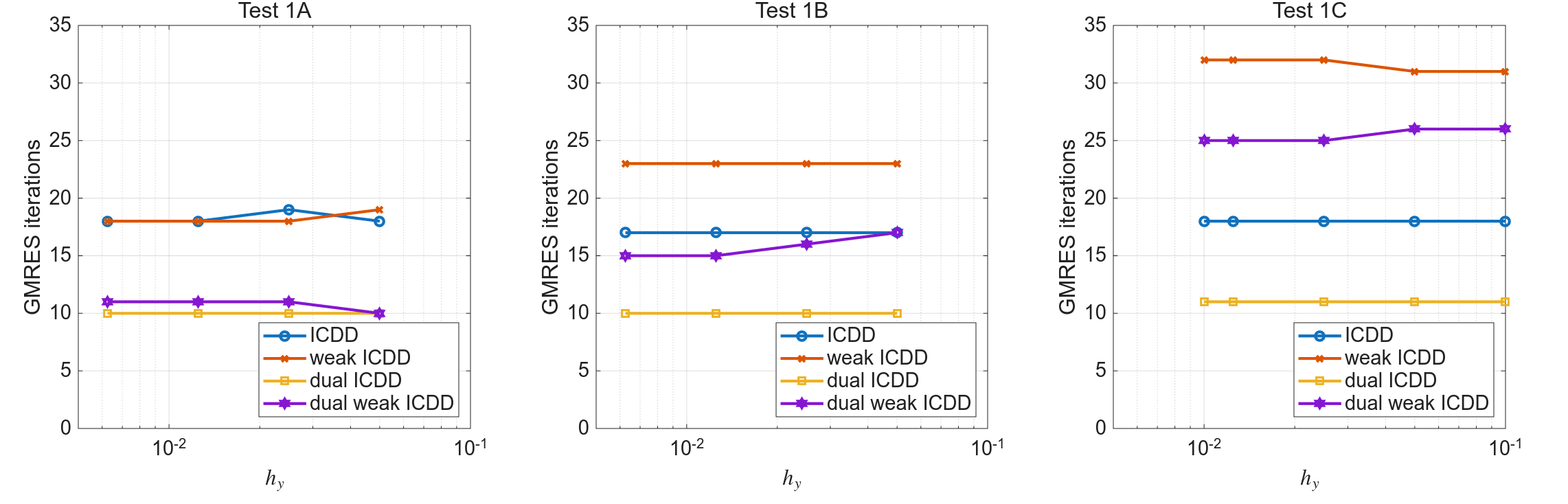}
    \caption{\emph{Test \#1.} GMRES iterations for Test 1A (left), Test 1B (center), Test 1C (right), versus $\delta$ (top row) and $h_y$  (bottom row). Test 1A, Test 1B, and Test 1C refer to FEM $\mathbb P_1-\mathbb P_1$, $\mathbb P_3-\mathbb P_3$, and SEM $\mathbb Q_6-\mathbb Q_6$ discretizations, respectively.}
    \label{fig:test1_iterations}
\end{figure}

\begin{figure}
    \centering
    \includegraphics[width=\linewidth]{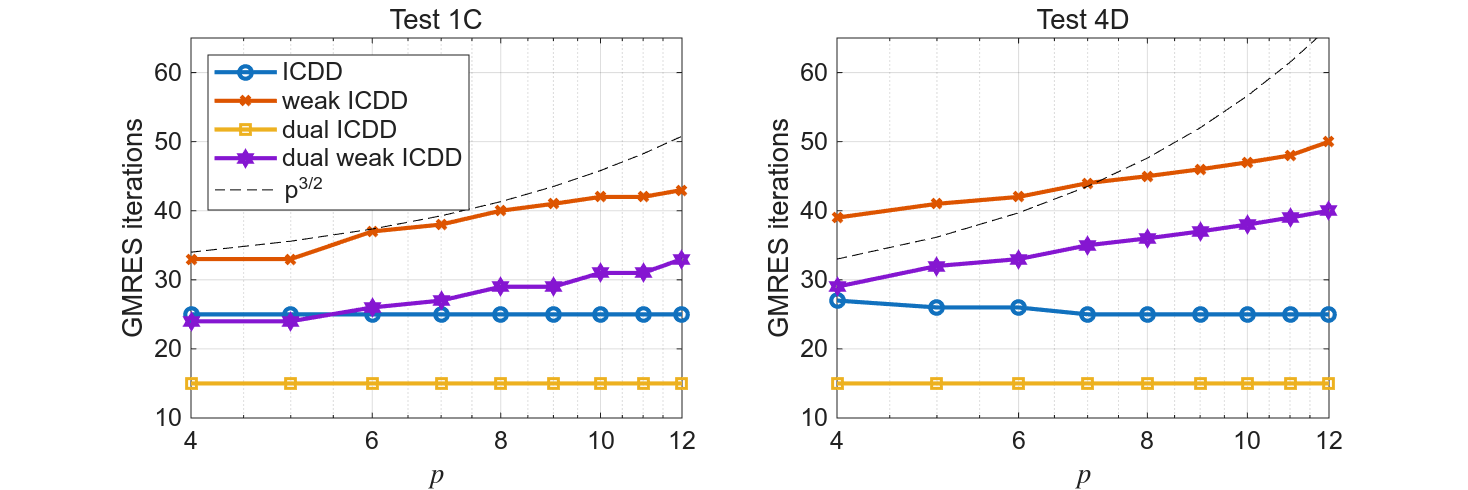}
    \caption{\emph{Test \#1 and Test \#4.} GMRES iterations for Test 1C (left), and Test 4D (right), versus $p$.}
    \label{fig:test1-4_iterations_p}
\end{figure}

\medskip

In Fig. \ref{fig:test2_iterations_delta} and \ref{fig:test2_iterations_hp}, we report the number of GMRES iterations required in the case of 
Test \#2 (see Sect. \ref{sec:numericalAnalysisConforming}). In particular, in Fig. \ref{fig:test2_iterations_delta}, we plot iterations of ICDD and of its other three variants versus $\delta$ for $\alpha\in\{0,-1,-2,-3\}$. 
As in Test \#1, also in this case iterations depend on $\delta$, but in a weaker way than established by Theorem \ref{thm:convergenceGMRES_weak} and Corollary \ref{cor:iterations_strong}. Moreover, the dual ICDD forms require fewer iterations, but are more efficient only for very small values of $\delta$. Furthermore, we note that as the ratio $\|\nu\|_\infty/\underline{\nu}=10^{|\alpha|}$ increases, the number of iterations tends to depend less on $\delta$, making ICDD the most convenient form for any $\delta$ considered. 

In Fig. \ref{fig:test2_iterations_hp}, the GMRES iterations of ICDD with respect to $h_y$ (left) and $p$ (right) are shown for all values of $\alpha$. In agreement with the results of Test \#1, also in this case the GMRES iterations are independent of both $h$ and $p$.

\begin{figure}
    \centering
    \includegraphics[width=\linewidth]{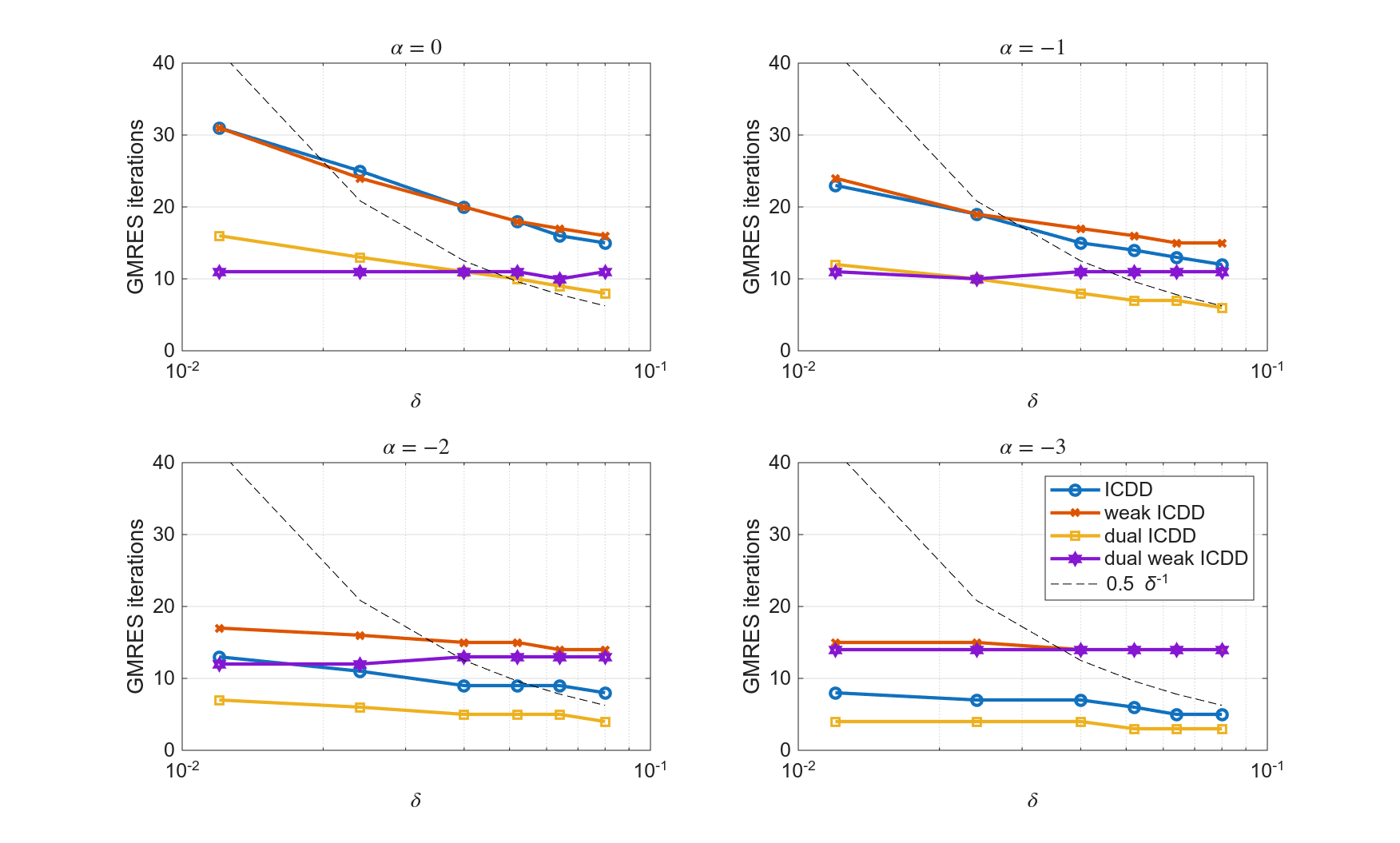}
    \caption{\emph{Test \#2.} GMRES iterations versus the overlap width $\delta$ for the case with discontinuous $\nu$ \eqref{eq:nu_disc}. FEM $\mathbb P_1$ discretization in both subdomains. The legend is the same for all the subplots.}
    \label{fig:test2_iterations_delta}
\end{figure}
\begin{figure}
    \centering
    \includegraphics[width=0.9\linewidth]{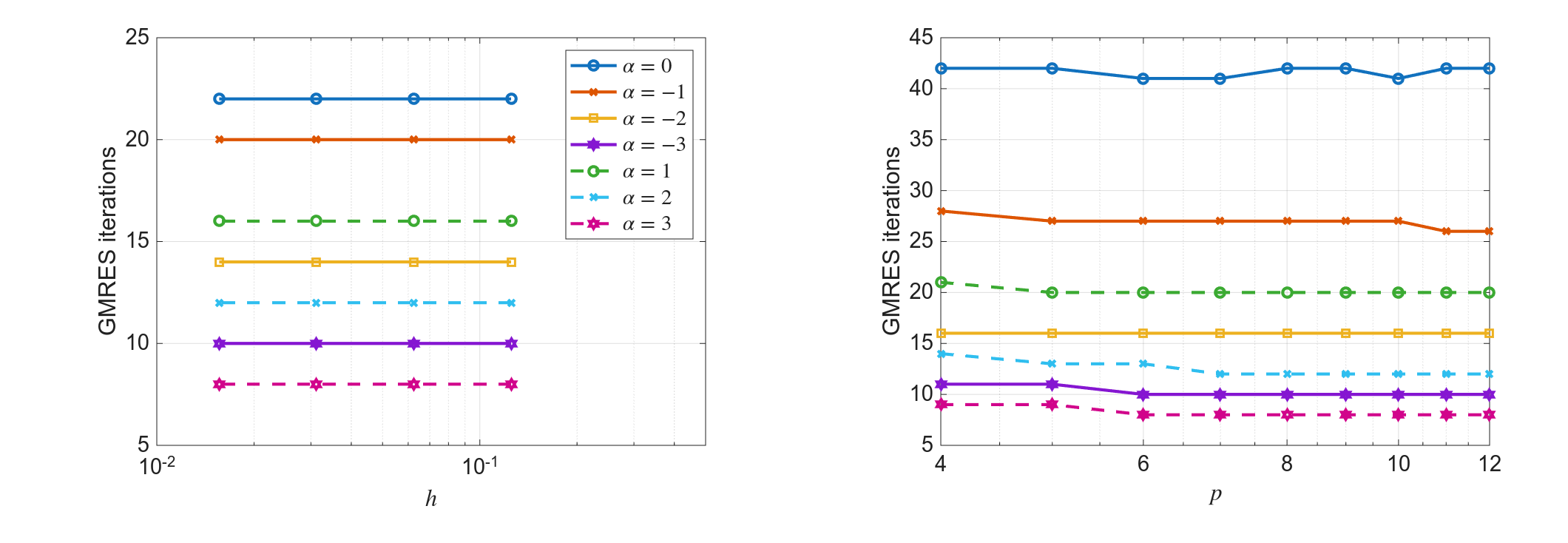}
    \caption{\emph{Test \#2.} GMRES iterations of ICDD versus $h$ for FEM $\mathbb P_1$ discretizations (left) and versus $p$ for SEM $\mathbb Q_p$ discretizations (right) in the case of discontinuous $\nu$ \eqref{eq:nu_disc}.}
    \label{fig:test2_iterations_hp}
\end{figure}

\medskip

Finally, analogous results are obtained for Test \#3 which features the discontinuous coefficient $\gamma$ \eqref{eq:gamma_disc} as shown in Fig. \ref{fig:test3_iterations_delta}. In particular, we observe that the number of iterations is almost independent of the jump of the coefficient $\gamma$.

\begin{figure}
    \centering
    \includegraphics[width=\linewidth]{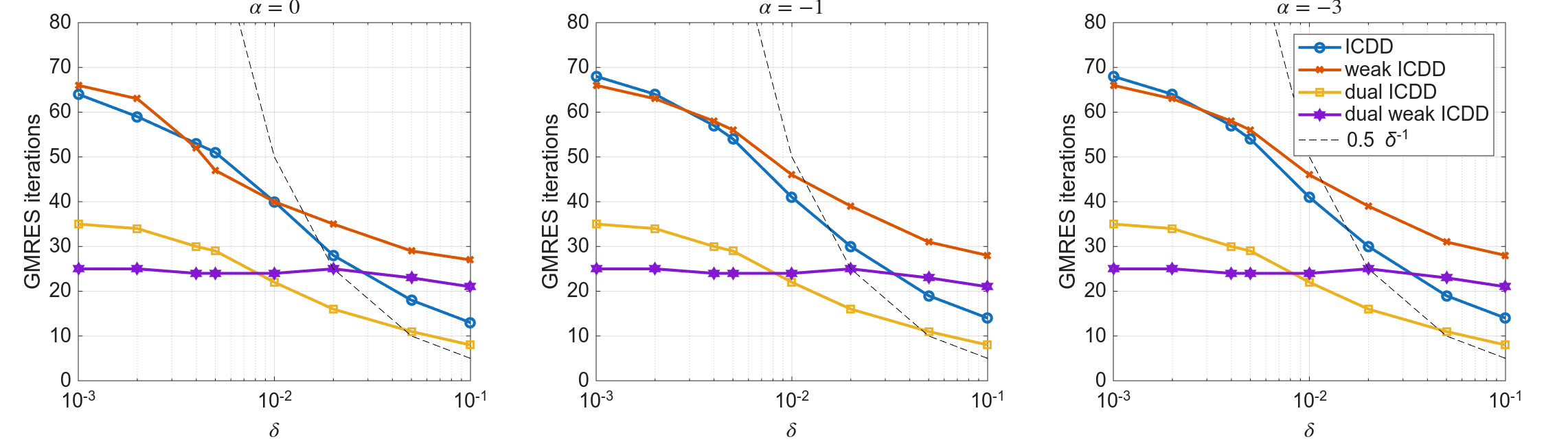}
    \caption{\emph{Test \#3.} GMRES iterations versus $\delta$ for the case with discontinuous $\gamma$ \eqref{eq:gamma_disc}. SEM $\mathbb Q_4$ discretizations.}
    \label{fig:test3_iterations_delta}
\end{figure}

\medskip

We now study numerically the behavior of ICDD iterations in the non-conforming setting, starting from Test \# 4 of Sect. \ref{sec:numericalAnalysisNonConforming}. Although this case is not covered by the convergence theory,
numerical results in Fig. \ref{fig:test4_iterations} show that ICDD and the other three variants perform similarly to the conforming case (Test \#1).

\begin{figure}
    \centering
    \includegraphics[width=\linewidth]{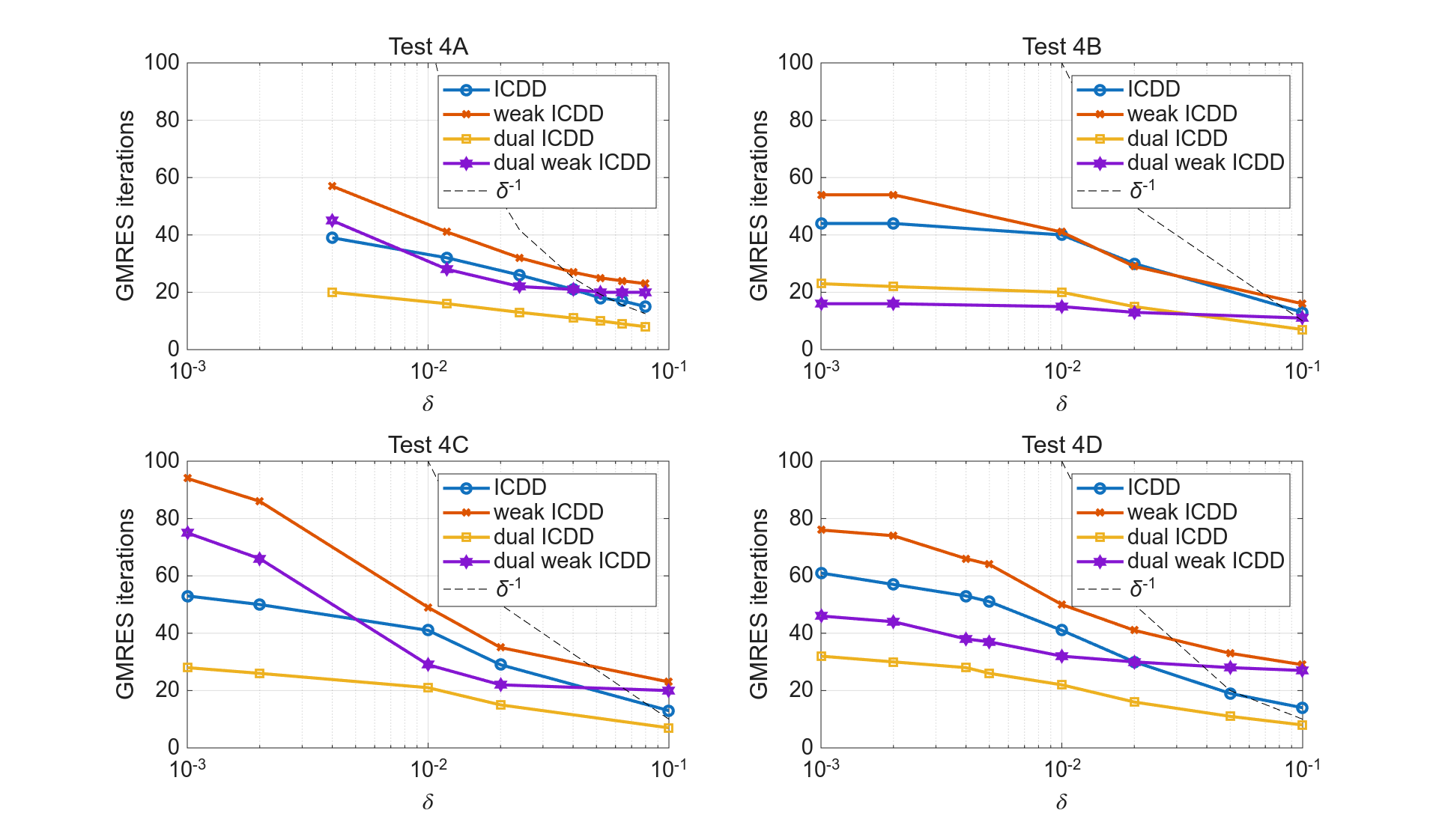}
    \includegraphics[width=\linewidth]{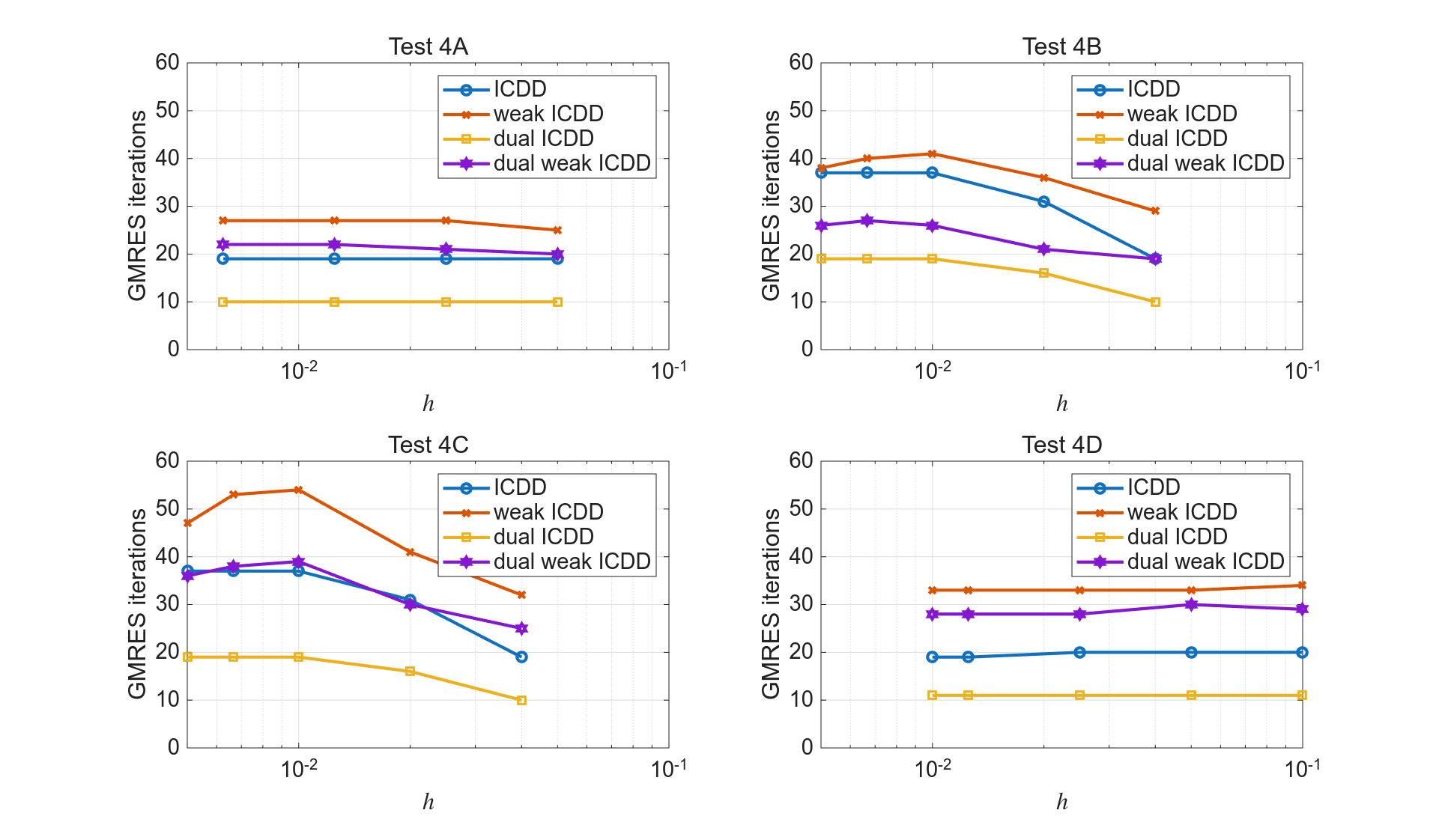}
    \caption{\emph{Test \#4.} GMRES iterations versus $\delta$ and $h$. Different non-conforming discretizations on the overlap.}
    \label{fig:test4_iterations}
\end{figure}

\medskip

We conclude by presenting a final non-conforming \textbf{Test \#5}, where we consider a geometrical configuration with intersecting interfaces $\Gamma_1$ and $\Gamma_2$ that is not covered by the theory of Sects. \ref{sec:convergence-analysis} and \ref{sec:convergenceGMRES}. More precisely, let $\Omega$ be the domain formed by the union of the rectangle $\Omega_1=(0,1.4+\overline{x})\times(-0.8,0.8)$ and the circle $\Omega_2$ with center $(2,0)$ and radius $1$ as shown in Fig. \ref{fig:test5_geometry_solution} (left). The interfaces $\Gamma_1$ and $\Gamma_2$ intersect at the two points of coordinates $(1.4,\pm 0.8)$, and $\delta$ corresponds to the maximum distance between $\Gamma_1$ and $\Gamma_2$.
We consider a rectangular structured mesh in $\Omega_1$ with $\mathbb Q_1$ discretization, and a uniform unstructured triangular mesh in $\Omega_2$ with either $\mathbb P_1$ (Test 5A) or $\mathbb P_2$ (Test 5B) finite elements. Clearly, the discretization is non-conforming on the overlap.
The solution of problem \eqref{eq:globalProblem} for $\nu = \gamma = 1$ and
\begin{equation*}\label{eq:rhs}
    f(x,y)=\left\{\begin{array}{rl}
    -200 & \text{for } (x,y) \in \Omega \text{ with } x \leq 0.9 \text{ and } y \leq 0.4\\
    0 & \text{for } (x,y) \in \Omega \text{ with } x \leq 0.9 \text{ and } y > 0.4 \\
    0 & \text{for } (x,y) \in \Omega \text{ with } x > 0.9 \text{ and } y \leq 0.4\\
    200 & \text{for } (x,y) \in \Omega \text{ with } x > 0.9 \text{ and } y > 0.4
    \end{array}
    \right.
\end{equation*}
is shown in Fig. \ref{fig:test5_geometry_solution} (right), while in Fig. \ref{fig:test5_iterations}, we report the number of GMRES iterations versus $\delta$. In the considered test cases, the number of iterations is independent of $\delta$ for all variants of ICDD.

\begin{figure}
    \centering
\begin{tikzpicture}[scale=2]
    \draw[thick] (0,-0.8) rectangle (1.5,0.8);
    \draw[thick] (2,0) circle (1);
    \draw[dashed] (1,0.0)--(1.5,0.0);
    \node[below left] at (0.3,-0.8) {$(0,-0.8)$};
    \node[above right] at (1.45,0.6) {$(1.4+\overline{x},0.8)$};
    \node[right] at (2,0) {$(2,0)$};
    \node[right] at (1.5,0.4) {$\Gamma_1$};
    \node[left] at (1.0,0.0) {$\Gamma_2$};
    \node[above] at (1.25,0.0) {$\delta$};
    \fill (2,0) circle (0.5pt);
    \fill (1.4,-0.8) circle (1pt);
    \fill (1.4,+0.8) circle (1pt);
\end{tikzpicture}
\qquad
\includegraphics[width=0.4\textwidth]{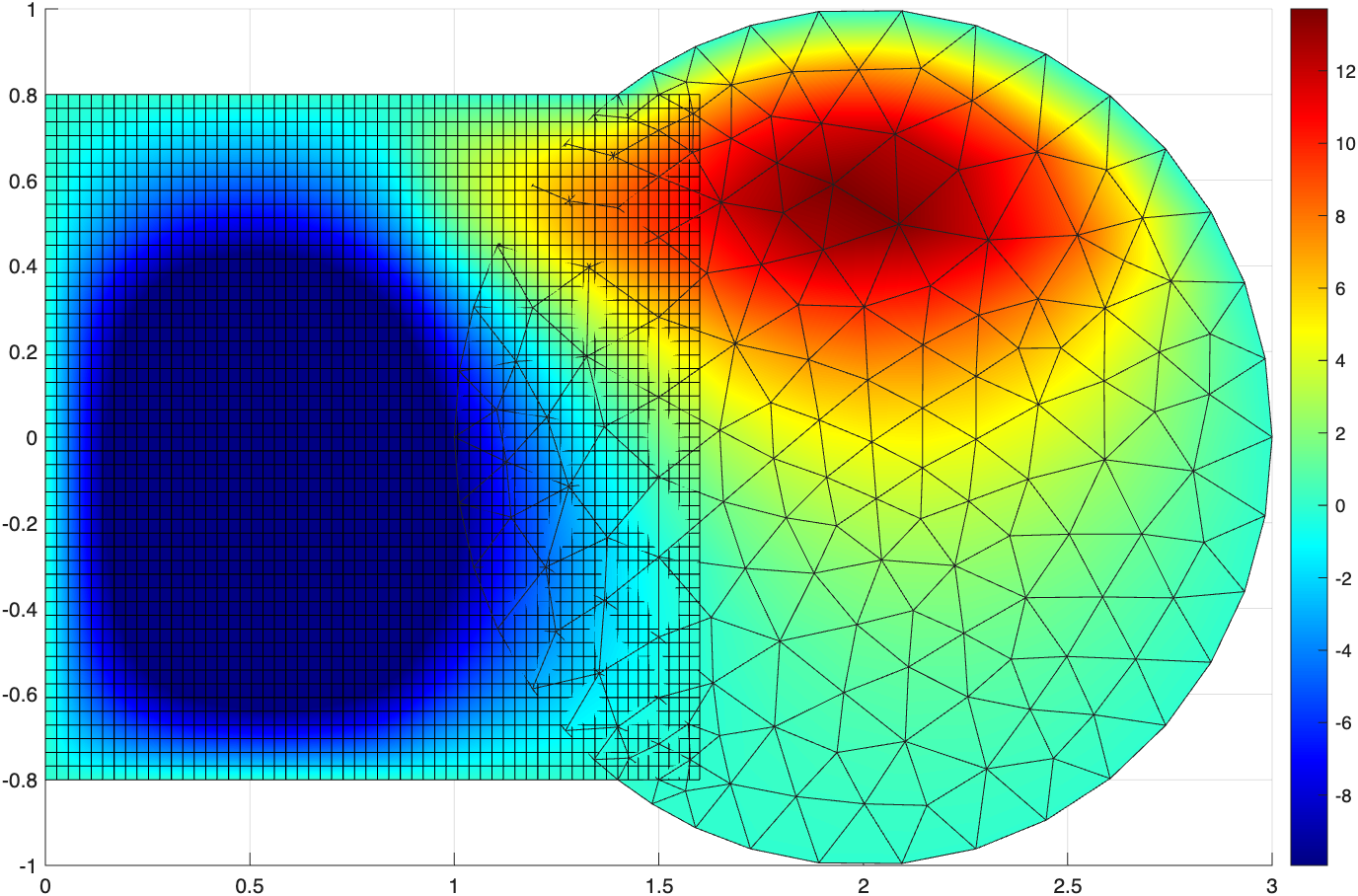}
\caption{\emph{Test \#5.} Geometrical setting (left) and computed solution for $\mathbb{Q}_1-\mathbb{P}_1$ elements (right).}
\label{fig:test5_geometry_solution}
\end{figure}

\begin{figure}
    \centering
    \includegraphics[width=\linewidth]{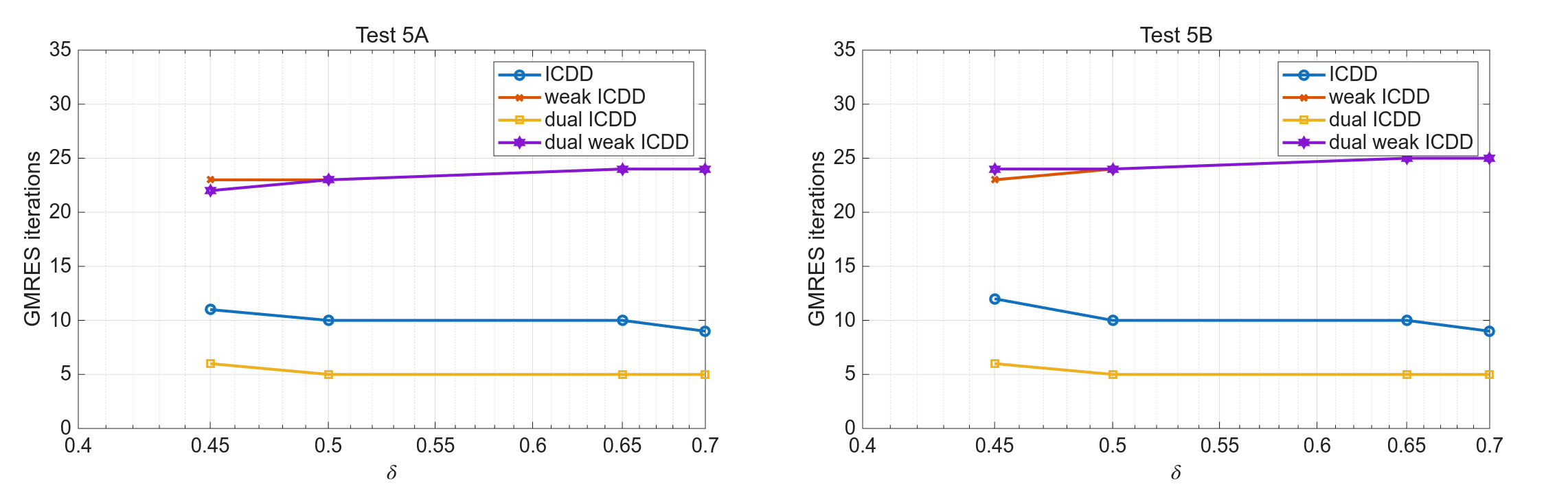}
    \caption{\emph{Test \#5.} GMRES iterations versus $\delta$. Different non-conforming discretizations on the overlap.}
    \label{fig:test5_iterations}
\end{figure}

\section{Conclusions and future work}\label{sec:conclusions}

In this paper, we have analyzed the spectral properties and convergence rate of the ICDD method for linear elliptic PDEs featuring possibly discontinuous coefficients, in 2D domains decomposed into two overlapping subregions. The ICDD method is formulated as an optimal control problem in which we minimize the jump between the subdomain solutions on the internal boundaries, named interfaces. The corresponding optimality system can be rewritten as a Schur complement equation whose unknowns are the degrees of freedom on the interfaces. 
In the particular case of conforming discretizations in the overlapping region, we analyzed the spectral properties of the Schur complement matrix. Then, we proved that the convergence rate of ICDD is independent of the mesh size $h$, behaves at most like $\mathcal O(\delta^{-1})$ when the overlap width $\delta$ tends to zero, and at most like $\mathcal O(p^{3/2}\log p)$ for increasing values of the local polynomial degree $p$. Numerical results verify the theoretical estimates about $\delta$ and $h$, and show that ICDD iterations are in fact independent of the polynomial degree $p$. Moreover, when local physical coefficients exhibit large variations between the two interfaces, ICDD iterations also become independent of $\delta$.
An analogous behavior has been observed numerically for non-conforming discretizations too, although this case is not covered by the theory presented in this work.

Since in the case of conforming discretizations on the overlap, the optimality system of ICDD coincides with the Substructured Restricted Additive Schwarz (SRAS) method \cite{SRAS2022}, the estimates we have proved in this work can be applied to analyze the convergence rate of SRAS, as well.

The ICDD method can be extended to consider more than two subdomains, including configurations with overlaps shared among three or more subdomains. In this case, we minimize the sum of the jumps of the solutions of any pair of adjacent subdomains on the corresponding interfaces. Thus, the analysis for the simple configuration with two subdomains provides insight into more general configurations.
In the case of many subdomains, the analysis versus the size $H$ of the subdomains is fundamental to understanding the scalability of ICDD. We note that in the present study, the dependence on $H$ is not discussed explicitly but is accounted for by the constant $c_{LH}$ that enters all relevant estimates. More precisely, when $H\to 0$, $c_{LH}$ behaves like $H^{-1}$. Such a result is consistent with classical convergence rate estimates of Schwarz methods without a coarse level. As in classical Schwarz methods, a coarse level should also be considered in ICDD to eliminate such a dependence, and this theoretical extension will be the object of future work. However, preliminary numerical results in the case of multiple subdomains have already been published in \cite{dgq_ell1}. In particular, in Test \#1 of \cite{dgq_ell1}, we applied the dual ICDD method (\ref{eq:schur_dual})--(\ref{eq:schur_dual_mat_rhs}) without a coarse level to a configuration with 16 subdomains and jumping coefficients, and obtained convergence in less than 30 iterations in all cases. In Test \#2 of \cite{dgq_ell1}, we reported the number of iterations for the dual ICDD method versus the local polynomial degrees for SEM discretizations in the case of non-conformity on the overlaps. These results agree with the results obtained in this work for the case of only two subdomains, namely that the number of ICDD iterations is independent of the local polynomial degrees.

\section*{Acknowledgements} The first author acknowledges funding through the EPSRC grant EP/V027603/1 and the QJMAM Fund for Applied Mathematics (grant QJMAM2023-R1). The second author is a member of GNCS – INdAM, and has received support from the project PRIN, MUR, Italy, CUP 20227K44ME, and ``Fondo di Ateneo per attivit\`a internazionali'' of Universit\`a di Brescia (Italy).

\appendix 

\section{Appendix}\label{sec:appendix}

\subsection{Equivalence between SRAS and ICDD for conforming meshes}\label{app:ICDD-SRAS}

Consider a conforming discretization in the overlap $\Omega_1 \cap \Omega_2$ as shown in Fig.~\ref{fig:exampleOverlap}.

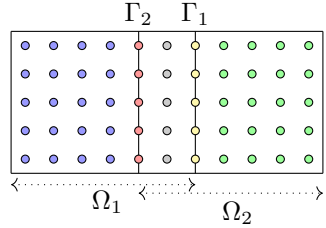
\begin{figure}[bht]
\begin{center}
 	\begin{tikzpicture}[scale=0.75]
        \draw (-0.5,-0.5)--(5.0,-0.5);
        \draw (-0.5,2.0)--(5.0,2.0);
        \draw (-0.5,-0.5)--(-0.5,2.0);
        \draw (5.0,-0.5)--(5.0,2.0);
        \draw (1.75,-0.5)--(1.75,2.0);
        \draw (2.75,-0.5)--(2.75,2.0);
        %
        \fill[blue!40] (-0.25, -0.25) circle (2pt);
        \draw (-0.25, -0.25) circle (2pt);
        \fill[blue!40] ( 0.25, -0.25) circle (2pt);
        \draw ( 0.25, -0.25) circle (2pt);
        \fill[blue!40] ( 0.75, -0.25) circle (2pt);
        \draw ( 0.75, -0.25) circle (2pt);
        \fill[blue!40] ( 1.25, -0.25) circle (2pt);
        \draw ( 1.25, -0.25) circle (2pt);
        \fill[red!40] ( 1.75, -0.25) circle (2pt);
        \draw ( 1.75, -0.25) circle (2pt);
        \fill[gray!40] ( 2.25, -0.25) circle (2pt);
        \draw ( 2.25, -0.25) circle (2pt);
        \fill[yellow!40] ( 2.75, -0.25) circle (2pt);
        \draw ( 2.75, -0.25) circle (2pt);
        \fill[green!40] ( 3.25, -0.25) circle (2pt);
        \draw ( 3.25, -0.25) circle (2pt);
        \fill[green!40] ( 3.75, -0.25) circle (2pt);
        \draw ( 3.75, -0.25) circle (2pt);
        \fill[green!40] ( 4.25, -0.25) circle (2pt);
        \draw ( 4.25, -0.25) circle (2pt);
        \fill[green!40] ( 4.75, -0.25) circle (2pt);
        \draw ( 4.75, -0.25) circle (2pt);
        \fill[blue!40] (-0.25, 0.25) circle (2pt);
        \draw (-0.25, 0.25) circle (2pt);
        \fill[blue!40] ( 0.25, 0.25) circle (2pt);
        \draw ( 0.25, 0.25) circle (2pt);
        \fill[blue!40] ( 0.75, 0.25) circle (2pt);
        \draw ( 0.75, 0.25) circle (2pt);
        \fill[blue!40] ( 1.25, 0.25) circle (2pt);
        \draw ( 1.25, 0.25) circle (2pt);
        \fill[red!40] ( 1.75, 0.25) circle (2pt);
        \draw ( 1.75, 0.25) circle (2pt);
        \fill[gray!40] ( 2.25, 0.25) circle (2pt);
        \draw ( 2.25, 0.25) circle (2pt);
        \fill[yellow!40] ( 2.75, 0.25) circle (2pt);
        \draw ( 2.75, 0.25) circle (2pt);
        \fill[green!40] ( 3.25, 0.25) circle (2pt);
        \draw ( 3.25, 0.25) circle (2pt);
        \fill[green!40] ( 3.75, 0.25) circle (2pt);
        \draw ( 3.75, 0.25) circle (2pt);
        \fill[green!40] ( 4.25, 0.25) circle (2pt);
        \draw ( 4.25, 0.25) circle (2pt);
        \fill[green!40] ( 4.75, 0.25) circle (2pt);
        \draw ( 4.75, 0.25) circle (2pt);
        \fill[blue!40] (-0.25, 0.75) circle (2pt);
        \draw (-0.25, 0.75) circle (2pt);
        \fill[blue!40] ( 0.25, 0.75) circle (2pt);
        \draw ( 0.25, 0.75) circle (2pt);
        \fill[blue!40] ( 0.75, 0.75) circle (2pt);
        \draw ( 0.75, 0.75) circle (2pt);
        \fill[blue!40] ( 1.25, 0.75) circle (2pt);
        \draw ( 1.25, 0.75) circle (2pt);
        \fill[red!40] ( 1.75, 0.75) circle (2pt);
        \draw ( 1.75, 0.75) circle (2pt);
        \fill[gray!40] ( 2.25, 0.75) circle (2pt);
        \draw ( 2.25, 0.75) circle (2pt);
        \fill[yellow!40] ( 2.75, 0.75) circle (2pt);
        \draw ( 2.75, 0.75) circle (2pt);
        \fill[green!40] ( 3.25, 0.75) circle (2pt);
        \draw ( 3.25, 0.75) circle (2pt);
        \fill[green!40] ( 3.75, 0.75) circle (2pt);
        \draw ( 3.75, 0.75) circle (2pt);
        \fill[green!40] ( 4.25, 0.75) circle (2pt);
        \draw ( 4.25, 0.75) circle (2pt);
        \fill[green!40] ( 4.75, 0.75) circle (2pt);
        \draw ( 4.75, 0.75) circle (2pt);
        \fill[blue!40] (-0.25, 1.25) circle (2pt);
        \draw (-0.25, 1.25) circle (2pt);
        \fill[blue!40] ( 0.25, 1.25) circle (2pt);
        \draw ( 0.25, 1.25) circle (2pt);
        \fill[blue!40] ( 0.75, 1.25) circle (2pt);
        \draw ( 0.75, 1.25) circle (2pt);
        \fill[blue!40] ( 1.25, 1.25) circle (2pt);
        \draw ( 1.25, 1.25) circle (2pt);
        \fill[red!40] ( 1.75, 1.25) circle (2pt);
        \draw ( 1.75, 1.25) circle (2pt);
        \fill[gray!40] ( 2.25, 1.25) circle (2pt);
        \draw ( 2.25, 1.25) circle (2pt);
        \fill[yellow!40] ( 2.75, 1.25) circle (2pt);
        \draw ( 2.75, 1.25) circle (2pt);
        \fill[green!40] ( 3.25, 1.25) circle (2pt);
        \draw ( 3.25, 1.25) circle (2pt);
        \fill[green!40] ( 3.75, 1.25) circle (2pt);
        \draw ( 3.75, 1.25) circle (2pt);
        \fill[green!40] ( 4.25, 1.25) circle (2pt);
        \draw ( 4.25, 1.25) circle (2pt);
        \fill[green!40] ( 4.75, 1.25) circle (2pt);
        \draw ( 4.75, 1.25) circle (2pt);
        \fill[blue!40] (-0.25, 1.75) circle (2pt);
        \draw (-0.25, 1.75) circle (2pt);
        \fill[blue!40] ( 0.25, 1.75) circle (2pt);
        \draw ( 0.25, 1.75) circle (2pt);
        \fill[blue!40] ( 0.75, 1.75) circle (2pt);
        \draw ( 0.75, 1.75) circle (2pt);
        \fill[blue!40] ( 1.25, 1.75) circle (2pt);
        \draw ( 1.25, 1.75) circle (2pt);
        \fill[red!40] ( 1.75, 1.75) circle (2pt);
        \draw ( 1.75, 1.75) circle (2pt);
        \fill[gray!40] ( 2.25, 1.75) circle (2pt);
        \draw ( 2.25, 1.75) circle (2pt);
        \fill[yellow!40] ( 2.75, 1.75) circle (2pt);
        \draw ( 2.75, 1.75) circle (2pt);
        \fill[green!40] ( 3.25, 1.75) circle (2pt);
        \draw ( 3.25, 1.75) circle (2pt);
        \fill[green!40] ( 3.75, 1.75) circle (2pt);
        \draw ( 3.75, 1.75) circle (2pt);
        \fill[green!40] ( 4.25, 1.75) circle (2pt);
        \draw ( 4.25, 1.75) circle (2pt);
        \fill[green!40] ( 4.75, 1.75) circle (2pt);
        \draw ( 4.75, 1.75) circle (2pt);
        %
        \node[black] at (1.2, -1.0) {$\Omega_1$};
        \draw[<->, black,dotted]  (-0.5, -0.7) -- (2.75, -0.7);
        \node[black] at (3.5, -1.2) {$\Omega_2$};
        \draw[<->, black,dotted]  (1.75, -0.85) -- (5.0, -0.85);
        \node[black] at (2.75, 2.3) {$\Gamma_1$};
        \node[black] at (1.75, 2.3) {$\Gamma_2$};
 	\end{tikzpicture}
	\end{center}
  \caption{Schematic representation of a two-domain conforming overlapping decomposition.}
  \label{fig:exampleOverlap}
\end{figure}

Using the notation of \cite{SRAS2022}, the substructure associated with $\Omega_1$ is $S_1=\Gamma_2$, while the substructure associated with $\Omega_2$ is $S_2=\Gamma_1$. The global substructure is $S=S_1\cup S_2$. The linear system (\ref{eq:global_linear_system}) can be written as
\begin{equation}\label{eq:srasGlobal}
\begin{pmatrix}
    \widetilde{\mathsf A}_{11} & \widetilde{\mathsf A}_{1\Gamma_2} & \mathsf 0 & \mathsf 0 & \mathsf 0 \\
    \widetilde{\mathsf A}_{\Gamma_2 1} & \widetilde{\mathsf A}_{\Gamma_2 \Gamma_2} & \widetilde{\mathsf A}_{\Gamma_2 3} & \mathsf 0 & \mathsf 0 \\
    \mathsf 0 & \widetilde{\mathsf A}_{3 \Gamma_2} & \widetilde{\mathsf A}_{3 3} & \widetilde{\mathsf A}_{3 \Gamma_1} & \mathsf 0 \\
    \mathsf 0 & \mathsf 0 & \widetilde{\mathsf A}_{\Gamma_1 3} & \widetilde{\mathsf A}_{\Gamma_1 \Gamma_1} & \widetilde{\mathsf A}_{\Gamma_1 2} \\
    \mathsf 0 & \mathsf 0 & \mathsf 0 & \widetilde{\mathsf A}_{2 \Gamma_1} & \widetilde{\mathsf A}_{2 2}
\end{pmatrix}
\begin{pmatrix}
    \widetilde{\uv}_1 \\
    \widetilde{\uv}_{\Gamma_2} \\
    \widetilde{\uv}_3 \\
    \widetilde{\uv}_{\Gamma_1} \\
    \widetilde{\uv}_2
\end{pmatrix}
=
\begin{pmatrix}
    \widetilde{\fv}_1 \\
    \widetilde{\fv}_{\Gamma_2} \\
    \widetilde{\fv}_3 \\
    \widetilde{\fv}_{\Gamma_1} \\
    \widetilde{\fv}_2
\end{pmatrix}
\end{equation}
where the sub-indices are chosen in the following way: `1' corresponds to the dofs in $\Omega_1 \setminus \overline{\Omega}_2$ (blue circles in Fig.~\ref{fig:exampleOverlap}); `$\Gamma_2$' denotes the dofs on the interface $\Gamma_2$ (red circles in Fig.~\ref{fig:exampleOverlap}); `3' denotes the dofs in the overlapping region $\Omega_1 \cap \Omega_2$ (grey circles in Fig.~\ref{fig:exampleOverlap}); `$\Gamma_1$' indicates the dofs on the interface $\Gamma_1$ (yellow circles in Fig.~\ref{fig:exampleOverlap}); `2' corresponds to the dofs in $\Omega_2 \setminus \overline{\Omega}_1$ (green circles in Fig.~\ref{fig:exampleOverlap}). The block matrices are denoted by $\widetilde{}$ to distinguish them from the notation introduced in Sect. \ref{sec:discrete}. The substructure solution array is $\widetilde{\uv}_S= (\widetilde{\uv}_{\Gamma_2}, \, \widetilde{\uv}_{\Gamma_1})^T$.

The local restriction and prolongation operators are chosen so that they satisfy condition (4) in \cite{SRAS2022}. For example,
\begin{equation*}
    \mathsf R_1 =
     \begin{pmatrix}
         \mathsf I_1 & \mathsf 0 & \mathsf 0 & \mathsf 0 & \mathsf 0 \\
         \mathsf 0 & \mathsf I_{\Gamma_2} & \mathsf 0 & \mathsf 0 & \mathsf 0 \\
         \mathsf 0 & \mathsf 0 & \mathsf I_3 & \mathsf 0 & \mathsf 0
     \end{pmatrix},
    \qquad
    \mathsf R_2 =
     \begin{pmatrix}
         \mathsf 0 & \mathsf 0 & \mathsf I_3 & \mathsf 0 & \mathsf 0 \\
         \mathsf 0 & \mathsf 0 & \mathsf 0 & \mathsf I_{\Gamma_1} & \mathsf 0 \\
         \mathsf 0 & \mathsf 0 & \mathsf 0 & \mathsf 0 & \mathsf I_2
     \end{pmatrix},
\end{equation*}
\begin{equation*}
    \widetilde{\mathsf P}_1 =
     \begin{pmatrix}
         \mathsf I_1 & \mathsf 0 & \mathsf 0 \\
         \mathsf 0 & \mathsf I_{\Gamma_2} & \mathsf 0 \\
         \mathsf 0 & \mathsf 0 & \alpha\,\mathsf I_3 \\
         \mathsf 0 & \mathsf 0 & \mathsf 0 \\
         \mathsf 0 & \mathsf 0 & \mathsf 0  
     \end{pmatrix},
    \qquad
    \widetilde{\mathsf P}_2 =
     \begin{pmatrix}
         \mathsf 0 & \mathsf 0 & \mathsf 0 \\
         \mathsf 0 & \mathsf 0 & \mathsf 0 \\
         (1-\alpha)\,\mathsf I_3 & \mathsf 0 & \mathsf 0 \\
         \mathsf 0 & \mathsf I_{\Gamma_1} & \mathsf 0 \\
         \mathsf 0 & \mathsf 0 & \mathsf I_2
     \end{pmatrix},
\end{equation*}
where $\alpha \in (0,1)$, while the restriction and prolongation operators on the substructure $S$ are
\begin{equation*}
    \overline{\mathsf R} =
     \begin{pmatrix}
     \mathsf 0 & \mathsf 0 & \mathsf 0 & \mathsf I_{\Gamma_1} & \mathsf 0 \\
     \mathsf 0 & \mathsf I_{\Gamma_2} & \mathsf 0 & \mathsf 0 & \mathsf 0    
     \end{pmatrix},
     \qquad
     \overline{\mathsf P} =\overline{\mathsf R}^T.
\end{equation*}

On the other hand, for the considered geometrical configuration, system \eqref{eq:os2_alg_matrix_conf} can be explicitly written as
\begin{equation}\label{eq:os2_alg_simplified}
\left(
\begin{array}{ccc|ccc|cc}
    \widetilde{\mathsf A}_{11} & \widetilde{\mathsf A}_{1\Gamma_2} & \mathsf 0 & \mathsf 0 & \mathsf 0 & \mathsf 0 & \mathsf 0 & \mathsf 0 \\
    \widetilde{\mathsf A}_{\Gamma_2 1} & \widetilde{\mathsf A}_{\Gamma_2 \Gamma_2} & \widetilde{\mathsf A}_{\Gamma_2 3} & \mathsf 0 & \mathsf 0 & \mathsf 0 & \mathsf 0 & \mathsf 0 \\
    \mathsf 0 & \widetilde{\mathsf A}_{3 \Gamma_2} & \widetilde{\mathsf A}_{3 3} & \mathsf 0 & \mathsf 0 & \mathsf 0  & \widetilde{\mathsf A}_{3 \Gamma_1}  & \mathsf 0\\
    \hline
    \mathsf 0 & \mathsf 0 & \mathsf 0 & \widetilde{\mathsf A}_{3 3}& \widetilde{\mathsf A}_{3 \Gamma_1} & \mathsf 0 & \mathsf 0  & \widetilde{\mathsf A}_{3 \Gamma_2}\\
    \mathsf 0 & \mathsf 0 & \mathsf 0 & \widetilde{\mathsf A}_{\Gamma_1 3} & \widetilde{\mathsf A}_{\Gamma_1 \Gamma_1} & \widetilde{\mathsf A}_{\Gamma_1 2} & \mathsf 0 & \mathsf 0 \\
    \mathsf 0 & \mathsf 0 & \mathsf 0 & \mathsf 0 & \widetilde{\mathsf A}_{2 \Gamma_1} & \widetilde{\mathsf A}_{2 2} & \mathsf 0 & \mathsf 0 \\
    \hline
    \mathsf 0 & \mathsf 0 & \mathsf 0 & \mathsf 0 & -\mathsf I_{\Gamma_1} & \mathsf 0 & \mathsf I_{\Gamma_1} & \mathsf 0\\
    \mathsf 0 & -\mathsf I_{\Gamma_2} & \mathsf 0 & \mathsf 0 & \mathsf 0 & \mathsf 0  & \mathsf 0 &  \mathsf I_{\Gamma_2}
\end{array}
\right)
\begin{pmatrix}
    \widetilde{\uv}_{1,1} \\
    \widetilde{\uv}_{1,\Gamma_2}\\
    \widetilde{\uv}_{1,3} \\
    \hline
    \widetilde{\uv}_{2,3} \\
    \widetilde{\uv}_{2,\Gamma_1}\\
    \widetilde{\uv}_{2,2} \\
    \hline
    \blambda_1 \\
    \blambda_2
\end{pmatrix}
=
\begin{pmatrix}
    \widetilde{\fv}_1 \\
    \widetilde{\fv}_{\Gamma_2} \\
    \widetilde{\fv}_{1,3} \\
    \hline
    \widetilde{\fv}_{2,3} \\
    \widetilde{\fv}_{\Gamma_1} \\
    \widetilde{\fv}_2 \\
    \hline
    \mathsf 0\\
    \mathsf 0
\end{pmatrix}
\end{equation}
where the first sub-index of $\widetilde{\uv}_{i,j}$ is used to identify quantities in the subdomains $\Omega_1$ and $\Omega_2$, respectively. Notice that
\begin{equation*}
    \mathsf A_{11} =
    \begin{pmatrix}
        \widetilde{\mathsf A}_{11} & \widetilde{\mathsf A}_{1\Gamma_2} & \mathsf 0\\
        \widetilde{\mathsf A}_{\Gamma_2 1} & \widetilde{\mathsf A}_{\Gamma_2 \Gamma_2} & \widetilde{\mathsf A}_{\Gamma_2 3} \\
        \mathsf 0 & \widetilde{\mathsf A}_{3 \Gamma_2} & \widetilde{\mathsf A}_{3 3}
    \end{pmatrix},
    \quad
    \mathsf A_{22} = 
    \begin{pmatrix}
        \widetilde{\mathsf A}_{3 3} & \widetilde{\mathsf A}_{3 \Gamma_1} & \mathsf 0 \\
        \widetilde{\mathsf A}_{\Gamma_1 3} & \widetilde{\mathsf A}_{\Gamma_1 \Gamma_1} & \widetilde{\mathsf A}_{\Gamma_1 2} \\
        \mathsf 0 & \widetilde{\mathsf A}_{2 \Gamma_1} & \widetilde{\mathsf A}_{2 2} \\
    \end{pmatrix},
    \quad
    \mathsf{A}_{1\Gamma_1} =
    \begin{pmatrix}
    \mathsf{0} \\ \mathsf{0} \\ \widetilde{\mathsf{A}}_{3\Gamma_1} 
    \end{pmatrix},
    \quad
    \mathsf{A}_{2\Gamma_2} =
    \begin{pmatrix}
    \widetilde{\mathsf{A}}_{3\Gamma_2} \\ \mathsf{0} \\ \mathsf{0} 
    \end{pmatrix},
\end{equation*}
where $\mathsf{A}_{11}$, $\mathsf{A}_{22}$, $\mathsf{A}_{1\Gamma1}$ and $\mathsf{A}_{2\Gamma_2}$ are the local matrices in \eqref{eq:os2_alg_matrix_conf}, while $\mathsf R_{\Gamma_1 2}=\begin{array}{ccc}[\mathsf 0 & \mathsf I_{\Gamma_1} & \mathsf 0]\end{array}$, and
$\mathsf R_{\Gamma_2 1}=\begin{array}{ccc}[\mathsf 0 & \mathsf I_{\Gamma_2} & \mathsf 0]\end{array}$. Quantities indicated with the same notation in \eqref{eq:srasGlobal} and \eqref{eq:os2_alg_simplified} coincide.

Using the notation introduced above, the matrix in (\ref{eq:preco_sras}) becomes
\begin{eqnarray*}
    \overline{\mathsf R}\,\mathsf P_{RAS}^{-1}\,\mathsf A\,\overline{\mathsf P}
    &=&
    \overline{\mathsf R}\,(\widetilde{\mathsf P}_1 \mathsf A_{11}^{-1} \mathsf R_1 + \widetilde{\mathsf P}_2 \mathsf A_{22}^{-1} \mathsf R_2)\,\mathsf A\,\overline{\mathsf P} \\
    &=& \overline{\mathsf R}\,\widetilde{\mathsf P}_1\,\mathsf A_{11}^{-1} \mathsf R_1\,\mathsf A\,\overline{\mathsf P} + \overline{\mathsf R}\,\widetilde{\mathsf P}_2 \mathsf A_{22}^{-1} \mathsf R_2\,\mathsf A\,\overline{\mathsf P} \\
    &=& \phantom{+}
    \overline{\mathsf R}\,\widetilde{\mathsf P}_1\,\mathsf A_{11}^{-1} \, 
    \left( 
    \begin{pmatrix}
     \widetilde{\mathsf A}_{11} & \widetilde{\mathsf A}_{1 \Gamma_2} & \mathsf 0 & \mathsf 0 & \mathsf 0 \\
     \widetilde{\mathsf A}_{\Gamma_2 1} & \widetilde{\mathsf A}_{\Gamma_2 \Gamma_2} & \widetilde{\mathsf A}_{\Gamma_2 3} & \mathsf 0 & \mathsf 0 \\
     \mathsf 0 & \widetilde{\mathsf A}_{3 \Gamma_2} & \widetilde{\mathsf A}_{3 3} & \mathsf 0 & \mathsf 0
    \end{pmatrix}
    +
    \begin{pmatrix}
        \mathsf 0 & \mathsf 0 & \mathsf 0 & \mathsf 0 & \mathsf 0 \\
        \mathsf 0 & \mathsf 0 & \mathsf 0 & \mathsf 0 & \mathsf 0 \\
        \mathsf 0 & \mathsf 0 & \mathsf 0 & \widetilde{\mathsf A}_{3 \Gamma_1} & \mathsf 0 \\
    \end{pmatrix}
    \right)\,\overline{\mathsf P} \\
    && +
    \overline{\mathsf R}\,\widetilde{\mathsf P}_2\,\widetilde{\mathsf A}_{22}^{-1} \, 
    \left( 
    \begin{pmatrix}
     \mathsf 0 & \mathsf 0 & \widetilde{\mathsf A}_{3 3} & \widetilde{\mathsf A}_{3 \Gamma_1} & \mathsf 0 \\
     \mathsf 0 & \mathsf 0 & \widetilde{\mathsf A}_{\Gamma_1 3} & \widetilde{\mathsf A}_{\Gamma_1 \Gamma_1} & \widetilde{\mathsf A}_{\Gamma_1 2} \\
     0 & \mathsf 0 & \mathsf 0 & \widetilde{\mathsf A}_{2 \Gamma_1} & \widetilde{\mathsf A}_{2 2}
    \end{pmatrix}
    +
    \begin{pmatrix}
        \mathsf 0 & \widetilde{\mathsf A}_{3 \Gamma_2} & \mathsf 0 & \mathsf 0 & \mathsf 0 \\
        \mathsf 0 & \mathsf 0 & \mathsf 0 & \mathsf 0 & \mathsf 0 \\
        \mathsf 0 & \mathsf 0 & \mathsf 0 & \mathsf 0 & \mathsf 0 \\
    \end{pmatrix}
    \right)\,\overline{\mathsf  P} \\
    &=&
    \begin{pmatrix}
        \mathsf I_{\Gamma_1} & \mathsf 0 \\
        \mathsf 0 & \mathsf I_{\Gamma_2}
    \end{pmatrix}
    +
    \begin{pmatrix}
     \mathsf 0 & \mathsf 0 & \mathsf 0 \\
     \mathsf 0 & \mathsf I_{\Gamma_2} & \mathsf 0
    \end{pmatrix}
    \mathsf A_{11}^{-1}
    \begin{pmatrix}
        \mathsf 0 & \mathsf 0 \\
        \mathsf 0 & \mathsf 0 \\
        \widetilde{\mathsf A}_{3 \Gamma_1} & \mathsf 0
    \end{pmatrix}
    +
    \begin{pmatrix}
     \mathsf 0 & \mathsf I_{\Gamma_1} & \mathsf 0 \\
     \mathsf 0 & \mathsf 0 & \mathsf 0
    \end{pmatrix}
    \widetilde{\mathsf A}_{22}^{-1}
    \begin{pmatrix}
        \mathsf 0 & \widetilde{\mathsf A}_{3 \Gamma_2} \\
        \mathsf 0 & \mathsf 0 \\
        \mathsf 0 & \mathsf 0
    \end{pmatrix},
\end{eqnarray*}
where we have used the definition of the matrices $\overline{\mathsf{R}}$, $\overline{\mathsf{P}}$, $\widetilde{\mathsf{P}}_1$ and $\widetilde{\mathsf{P}}_2$. The resulting matrix coincides with the Schur complement system matrix $\mathsf{\Sigma}$ defined in \eqref{eq:schur_2_conforming}.

Following an analogous procedure, one can verify that the right-hand side of \eqref{eq:preco_sras} coincides with the right-hand side $\mathsf{\boldsymbol{\chi}}$ of \eqref{eq:schur_2_conforming}.

Finally, we observe that $\uv_S$ coincides with $\blambda$.
Therefore, we can conclude that, in the case of two conforming overlapping subdomains, the ICDD method coincides with SRAS.


\subsection{Proof of Theorem \ref{thm:ls}}\label{appendix:proof}

Let us first recall that
for any bounded domain $\domain\subset {\mathbb R}^d$ ($d\geq 2$) of diameter $H$ with Lipschitz boundary $\partial\domain$, and for any $u\in H^1(\domain)$, there exists a positive constant $c_\domain$ depending only on the shape of $\domain$, but independent of $H$, such that 
\begin{equation}\label{trace_L2bordo}
\|u\|^2_{L^2(\partial\domain)}\leq c_\domain\left(\frac{1}{H}\|u\|^2_{L^2(\domain)}+
H\|\nabla u\|^2_{L^2(\domain)}\right).
\end{equation}
The latter estimate can be proved starting from the trace inequality $\|u\|_{H^{1/2}(\partial\domain)}\leq c_\domain^{1/2} \|u\|_{H^1(\domain)}$ and by applying scaling arguments \cite[Sect. 1.3 and Appendix A.2]{toselli-widlund}.

\bigskip

We now proceed to prove Theorem \ref{thm:ls}.

\smallskip

The assumptions made in Theorem \ref{thm:ls} on both $\domain$ and $\gamma_t$ ensure that the hypotheses of Lemma 3.6 of \cite{loisel_szyld_nummath} are satisfied, so that (\ref{eq:ls}) is an immediate consequence of the mentioned Lemma.
To prove (\ref{eq:CLS_behaviour}), we follow the ideas of the proof of Lemma 3.6 of \cite{loisel_szyld_nummath}.

\smallskip

Let $u^\lambda=u^\lambda(x,y)$. For any $y\in(0,H_y)$ and $t\in(0,H_x)$, it holds
\begin{equation*}
(u^\lambda)^2(H_x,y)-(u^\lambda)^2(t,y)
= \int_t^{H_x} \frac{\partial}{\partial x} (u^\lambda(x,y))^2 \, dx
=\int_t^{H_x} 2u^\lambda(x,y)
\frac{\partial u^\lambda}{\partial x}(x,y) \, dx
\end{equation*}
and
\begin{equation}\label{eq:diff_tracce}
\|\lambda\|_{L^2(\Gamma)}^2-\|u^\lambda\|_{L^2(\gamma_t)}^2
=\int_0^{H_y}\int_t^{H_x} 2u^\lambda(x,y)
\frac{\partial u^\lambda}{\partial x}(x,y) \,dx \,dy.
\end{equation}
Recalling that $u^\lambda$ is the harmonic extension of $\lambda$ so that $u^\lambda=0$ on $\partial\domain\setminus\Gamma$ and $-\nabla \cdot (\nu \nabla u^\lambda) + \gamma \, u^\lambda=0$ a.e. in $\domain$, by using Green's formula on $\domain_s=(0,s)\times(0,H_y)$, for any $s\in(0,H_x)$ and $\gamma_s=\{s\}\times(0,H_y)$, we find

\begin{eqnarray*}
\begin{array}{ll}
\displaystyle\int_{\gamma_s} \nu \, u^\lambda \, \frac{\partial u^\lambda}{\partial x} \, d\gamma_s  &
\displaystyle =
\int_{\partial\domain_s} \nu \, u^\lambda \, \frac{\partial u^\lambda}{\partial n}  \, d(\partial\domain_s) \\
&\displaystyle =
\int_{\domain_s} ( \nu |\nabla u^\lambda |^2 + \gamma \, (u^\lambda)^2  ) \, d\domain_s -
\int_{\domain_s} \underbrace{(-\nabla\cdot( \nu \nabla u^\lambda)  + \gamma u^\lambda)}_{=0\  a.e.} \, u^\lambda \, d\domain_s \\
&\displaystyle =\int_0^s g(x)dx,
\end{array}
\end{eqnarray*}
where 
$g(x)=\displaystyle \int_0^{H_y}\nu(x,y)|\nabla u^\lambda(x,y)|^2+ \gamma(x,y)(u^\lambda(x,y))^2 dy.$ 

Then, by applying the Fubini theorem and exchanging the order of derivation, we have
\begin{eqnarray}
    \int_t^{H_x}\int_{\gamma_s} \nu u^\lambda \frac{\partial u^\lambda}{\partial x} dy\,ds &=&
    \int_t^{H_x} \int_0^s g(x)dx\, ds \nonumber \\
    &=& \int_0^t\int_t^{H_x}g(x)ds\, dx+ \int_t^{H_x}\int_x^{H_x}g(x) ds\, dx \nonumber\\
    &=& \int_0^t g(x)(H_x-t)dx+\int_t^{H_x}g(x)(H_x-x)dx=\int_0^{H_x}g(x)\mu_t(x)dx \nonumber \\
    &=&\int_0^{H_y}\int_0^{H_x}(\nu|\nabla u^\lambda|^2 +\gamma (u^\lambda)^2)\mu_t(x)dx\,dy=: \|u^\lambda\|_{\mu_t}^2, \label{eq:norma_mu}
\end{eqnarray}
where 
\begin{eqnarray*}
\mu_t(x)=\left\{\begin{array}{ll}
H_x-t & 0<x<t\\
H_x-x & t\leq x \leq H_x
\end{array}\right.
\end{eqnarray*}
is a positive weighting function and $\|\cdot\|_{\mu_t}$ defines a norm on $H^1(\domain)$ .

This latter norm is not equivalent to the norm $\|\cdot\|_{H^1(\domain)}$, but to the canonical norm $\|\cdot\|_{H^1(\domain,\mu_t)}$ of the weighted Sobolev space $H^1(\domain, \mu_t)= \{u:\domain\to{\mathbb R}:\ \|u\|_{H^1(\domain, \mu_t)}=\left(\int_\domain (|\nabla u|^2+u^2)\mu_t({\bf x})d{\bf x}\right)^{1/2} <\infty \}$ (see, e.g. \cite[Ch. 3]{kufner}).
This means that there exist two positive constants $c_1\leq c_2$ depending on $\nu$, $\gamma$, 
$H_x$, and possibly on $H_x-t$, such that
\begin{equation}\label{equiv.norme.pesi}
c_1 \|u\|^2_{H^1(\domain,\mu_t)}\leq \|u\|^2_{\mu_t}\leq
c_2 \|u\|^2_{H^1(\domain,\mu_t)}
\qquad \forall u\in H^1(\domain).
\end{equation}

Set $\underline\gamma=\inf_{{\bf x}\in\domain} \gamma({\bf x})$ and $\underline\nu=\inf_{{\bf x}\in\domain} \nu({\bf x})$. Then, $c_2=\max\{\|\nu\|_\infty,\|\gamma\|_\infty\}$ while

\begin{equation*}
c_1 =
\left\{
\begin{array}{ll}
4\underline\nu/(4+2t H_x+(H_x-t)^2) & \text{if } \underline\gamma=0 ,\\
\min\{\underline\nu,\underline\gamma\} & \text{if } \underline\gamma \not= 0 .
\end{array}
\right.
\end{equation*}

The expression of $c_1$ in the case $\gamma=0$ is obtained by exploiting the fact that $u=0$ when $x=0$, the fact that $\mu_t$ is independent of $y$, and following standard procedures to compute the Poincar\'e constant.

From (\ref{eq:diff_tracce})--(\ref{eq:norma_mu}), it follows that 
\begin{eqnarray}\label{eq:stima_sup}
\begin{array}{ll}
\|\lambda\|^2_{L^2(\Gamma)}-\|u^\lambda\|^2_{L^2(\gamma_t)}&\displaystyle =
\int_0^{H_y}\int_t^{H_x}2u^\lambda(x,y)\frac{\partial u^\lambda}{\partial x}(x,y)dx dy\\[3mm]
&\displaystyle \leq \frac{2}{\underline{\nu}}\int_t^{H_s}\left(\int_{\gamma_s}\nu u^\lambda\frac{\partial u^\lambda}{\partial x} d\gamma_s\right)ds = \frac{2}{\underline\nu}\|u^\lambda\|_{\mu_t}^2.
\end{array}
\end{eqnarray}
and, similarly,
\begin{equation}\label{eq:stima_inf}
\|\lambda\|^2_{L^2(\Gamma)}-\|u^\lambda\|^2_{L^2(\gamma_t)}
\geq \frac{2}{\|\nu\|_\infty}\|u^\lambda\|_{\mu_t}^2.
\end{equation}

Now we set $\domain_t=(0,t)\times (0,H_y)$, thanks to the definition of ${\mu_t}({\bf x})$ and (\ref{equiv.norme.pesi}) it holds
\begin{equation}\label{eq:equiv_norme_pesate} 
(H_x-t)\|u^\lambda\|^2_{H^1(\domain_t)}\leq 
\|u^\lambda\|^2_{H^1(\domain,{\mu_t})}\leq
\frac{1}{c_1}\|u^\lambda\|^2_{\mu_t}.
\end{equation}

By the $L^2-$trace inequality (\ref{trace_L2bordo}), and denoting $H_t = \text{diam}(\domain_t) = \max\{t,H_y\}$, we have
\begin{equation}\label{eq:stima_traccia_gammat}
 \|u^\lambda\|^2_{L^2(\gamma_t)}
\leq c \, \max \{ H_t, H_t^{-1} \} \, \| u^\lambda\|_{H^1(\domain_t)}^2
\leq c\,\frac{\max\{ H_t, H_t^{-1} \}}{c_1(H_x-t)}\|u^\lambda\|^2_{\mu_t}
\end{equation}
where $c$ is a positive constant only depending on the shape of $\domain_t$.

Starting from (\ref{eq:stima_sup}), and by applying (\ref{eq:stima_traccia_gammat}), we obtain
\begin{equation}\label{eq:stima_traccia_Gamma}
    \|\lambda\|^2_{L^2(\Gamma)}\leq \left(\frac{2}{\underline\nu}+\frac{c\,\max\{ H_t, H_t^{-1} \}}{c_1(H_x-t)}\right)\|u^\lambda\|^2_{\mu_t}
\end{equation}
and, finally, by (\ref{eq:stima_inf}) and (\ref{eq:stima_traccia_Gamma}), we conclude that

\begin{equation*} 
\|u^\lambda\|^2_{L^2(\gamma_t)}\leq \left(1-\frac{\frac{2}{\|\nu\|_\infty}}{\frac{2}{\underline{\nu}}+\frac{c\max\{H_t,H_t^{-1}\}}{c_1(H_x-t)}}\right)
\|\lambda\|^2_{L^2(\Gamma)}
=\left(1-\frac{1}{\frac{\|\nu\|_\infty}{\underline{\nu}}+\frac{c\,\|\nu\|_\infty\,\max\{H_t,H_t^{-1}\}}{2\,c_1\,(H_x-t)}}\right)
\|\lambda\|^2_{L^2(\Gamma)} .
\end{equation*}

Notice that the quantity in brackets is positive and smaller than $1$.

\subsection{Discrete norm}\label{appendix:mass}

Let $(\xi_q,\omega_q)$ for $q=0,\ldots,p$, be the nodes and weights of Legendre-Gauss-Lobatto quadrature formulas \cite{chqz06} on $[-1,1]$ such that, for any $u,\ v\in C^0([-1,1])$, 
\begin{equation}\label{eq:LGL-innerproduct}
(u,v)_{LGL}=\sum_{q=0}^p u(\xi_q) v(\xi_q)\omega_q\simeq \int_{-1}^1 u(x) v(x) dx.
\end{equation}
Then, we define the discrete norm $\|u\|_{LGL}=(u,u)^{1/2}_{LGL}$.

The following estimate establishes the equivalence between the $L^2-$norm and the discrete LGL-norm for any polynomial of degree less than or equal to $p$ \cite[Sect. 5.3]{chqz06}:
\begin{equation}\label{eq:normEquivalence}
\|v_p\|_{L^2(-1,1)}\leq \|v_p\|_{LGL}\leq \sqrt{3}\|v_p\|_{L^2(-1,1)}\qquad \forall v_p\in \mathbb P_p.
\end{equation}
Let us denote by $\{\mu_i\}\in\mathbb P_p$ the Lagrange basis on $(-1,1)$ with support points given by the LGL nodes $\xi_q$, wih $q=0,\ldots, p$. The one-dimensional mass matrix associated with the discrete inner product (\ref{eq:LGL-innerproduct}) reads
\begin{equation*}
(\mathsf M_{LGL})_{ij}=\sum_{q=0}^p \mu_i(\xi_q)\mu_j(\xi_q)\omega_q=\omega_i\delta_{ij}\qquad \mbox{ for }i,j=0,\ldots,p,
\end{equation*}
where $\delta_{ij}$ is the Kronecker delta. Thus, $\mathsf M_{LGL}$ is diagonal, and its minimum and maximum eigenvalues satisfy the following estimates \cite[Sect. 4]{bm-handbook}, \begin{equation}\label{eq:extrema_eig_MLGL}
\lambda_{min}(\mathsf M_{LGL})=\frac{2}{p(p+1)}, \hskip 1.cm 
\lambda_{max}(\mathsf M_{LGL})\leq \frac{c}{p},
\end{equation}
where $c>0$ is independent of $p$.

Let $\mathsf M$ be the exact mass matrix on $(-1,1)$ with $\mathsf M_{ij}=\int_{-1}^1 \mu_i(x)\mu_j(x)dx$.
An immediate consequence of (\ref{eq:normEquivalence}) is that, if $\boldsymbol v$ is the array of the degrees of freedom of the polynomial $v_p\in\mathbb P_p$, then 
$\boldsymbol{v}^T \mathsf M \boldsymbol{v}=\|v_p\|^2_{L^2(-1,1)}$, 
$\boldsymbol{v}^T \mathsf M_{LGL} \boldsymbol{v}=\|v_p\|^2_{LGL}$, and
\begin{equation}\label{eq:normEquivalenceMatrix}
\boldsymbol{v}^T \mathsf M \boldsymbol{v}\leq \boldsymbol v^T\mathsf M_{LGL} \boldsymbol v \leq 3 \boldsymbol v^T \mathsf M \boldsymbol v \qquad \forall \boldsymbol v\in \mathbb R^{p+1}.
\end{equation}

\subsection{Estimate of $\lambda_{min}(\widetilde{\mathsf{\Sigma}}_s)$ for non-conforming meshes}\label{app:lambdamin}

Following \cite[formulas (5.4.33), (5.4.34)]{chqz06} and \cite[Theorem 14.2]{bm-handbook}), we begin by providing an interpolation estimate for SEM (or $hp$-FEM). More precisely, let $I\subset \mathbb R$ be a bounded interval, $h>0$ be the maximum diameter of the elements of a family of partitions $\mathcal{T}_h$ of $I$, and consider polynomial degree $p\geq 1$. Let $\mathscr{I}_{h,p}$ be the composite Lagrange interpolation operator of local degree $p$, with the Legendre--Gauss--Lobatto nodes in each element as interpolation nodes. For any real $k\in[0,1]$ and integer $s\geq 1$, there exists a positive constant $c_I$ independent of both $h$ and $p$ such that
\begin{equation}\label{eq:interpolation_SEM}
\|u-\mathscr{I}_{h,p} u\|_{H^k(I)} \leq c_I \, h^{\mu-k}p^{k-s}|u|_{H^{s,p}(I)}\qquad \forall u\in H^s(I),
\end{equation}
where $\mu=\min(s,p+1)$ and $|u|_{H^{s,p}(I)}=\left(\sum_{m=\min(s,p+1)}^s \|D^{m} u\|^2_{L^2(I)}\right)^{1/2}$.
This estimate follows from formulas (5.4.33), (5.4.34) of \cite{chqz06}, scaling arguments (see, e.g., \cite[Lemma 4.2]{babuska_suri_m2an}), and interpolation between Sobolev spaces (see, e.g., \cite[Chapter 1]{bm-handbook}).

Moreover, for any globally continuous function $u$ on $I$ of degree $p$ on each element of $\mathcal T_h$, it holds
\begin{equation}\label{eq:inverse_schwab}
|u|_{H^1(I)}\leq 2\sqrt{3} \, \frac{p^2}{h} \| u\|_{L^2(I)} \, .
\end{equation}
For the proof see \cite[Theorem 3.91]{schwab_book}.

Then, the following result for the interpolation operator $\mathscr I_{\Gamma_k}$ defined in \eqref{eq:Lagrange_interpol} holds.

\begin{lemma}\label{lemma:interpolazione}
Consider two non-conforming discretizations on the overlap, let Assumptions \ref{ass_mesh} and \ref{ass_decomposition} be satisfied, $\mathscr I_{\Gamma_k}$ (for $k=1,2$) be the interpolation operator defined in (\ref{eq:Lagrange_interpol}), and $\ell=3-k$. Then, there exists a positive constant $c_S>0$ independent of the local polynomial degrees $p_k$, $ p_\ell$, and mesh sizes $h_k$, $h_\ell$ such that, denoting by $\mu_\ell$ the trace on $\Gamma_k$ of a suitable function $u_{\ell}\in V_{\ell,h_\ell}$, it holds
\begin{equation*}
    \| \mu_\ell-\mathscr I_{\Gamma_k}\mu_\ell 
     \|_{L^2(\Gamma_k)} \leq c_S \frac{h_k}{h_\ell} \frac{p_\ell^2}{p_k} \, \|\mu_\ell\|_{L^2(\Gamma_k)}.
\end{equation*}
Moreover, the following stability estimate holds
\begin{equation}\label{eq:interpolation_stability}
     \| \mathscr I_{\Gamma_k}\mu_\ell\|_{L^2(\Gamma_k)}  \leq  S_{k\ell} \, \|\mu_\ell\|_{L^2(\Gamma_k)},
     \quad \text{with} \; S_{k\ell} = 1+c_S\frac{h_k}{h_\ell}\frac{p_\ell^2}{p_k}\,.
\end{equation}
\end{lemma}

\begin{proof}
Since $u_\ell\in V_\ell$ and the interface $\Gamma_k$ is a vertical segment, the trace $\mu_\ell={u_\ell}_{|\Gamma_k}$ is globally continuous and a piecewise polynomial on $\Gamma_k$, thus $\mu_\ell\in H^1(\Gamma_k)$. Consider \eqref{eq:interpolation_SEM} and notice that the interpolation operator $\mathscr I_{h,p}$ is nothing else but $\mathscr I_{\Gamma_k}$ with $h=h_k$ and $p=p_k$. Then, by applying (\ref{eq:interpolation_SEM})  and (\ref{eq:inverse_schwab}), we have
\begin{equation*}
     \| \mu_\ell-\mathscr I_{\Gamma_k}\mu_\ell 
     \|_{L^2(\Gamma_k)} \leq c_I \, h_k \, p_k^{-1} |\mu_\ell|_{H^1(\Gamma_k)} 
     \leq c_S \, \frac{h_k}{h_\ell} \, \frac{p_\ell^2}{p_k} \, \|\mu_\ell\|_{L^2(\Gamma_k)},
\end{equation*}
where $c_S=2\sqrt{3}\,c_I$.
Estimate \eqref{eq:interpolation_stability} follows straightforwardly from the triangular inequality.
\end{proof}

\bigskip

Thanks to \eqref{eq:interpolation_stability} and \eqref{eq:stimal2_gg}, there holds
\begin{equation*}
    \| \mathscr{I}_{\Gamma_k}(u_{\ell,h_\ell\vert\Gamma_k}) \|_{L^2(\Gamma_k)}
    \leq
    S_{k\ell}  \| u_{\ell,h_\ell} \|_{L^2(\Gamma_k)}
    \leq
    S_{k\ell} \,
    C_{\delta,\ell} \,
    \| \zeta_{\ell,h_\ell} \|_{L^2(\Gamma_\ell)}.
\end{equation*}

To obtain a lower bound for $\lambda_{min}(\widetilde{\mathsf{\Sigma}}_s)$, we can follow an analogous procedure to Theorem \ref{thm:lowerbound_min_eig}. However, in this case, for $k,\ell=1,2$, $k \not= \ell$, we should consider
\begin{eqnarray*}
\int_{\Gamma_k} \mathscr I_{\Gamma_k} (u_{\ell,h_\ell}^{\zeta_\ell}{}_{|\Gamma_k}) \zeta_{k,h_k} 
&\leq&
\| \mathscr I_{\Gamma_k} (u_{\ell,h_\ell}^{\zeta_\ell}{}_{|\Gamma_k}) \|_{L^2(\Gamma_k)} \, \| \zeta_{k,h_k} \|_{L^2(\Gamma_k)} \\
&\leq&
S_{k\ell} \, 
\| u_{\ell,h_\ell}^{\zeta_\ell}{}_{|\Gamma_k} \|_{L^2(\Gamma_k)} \, \| \zeta_{k,h_k} \|_{L^2(\Gamma_k)} \\
&\leq& 
S_{k\ell} \, 
C_{\delta,\ell} \,
\| \zeta_{\ell,h_\ell} \|_{L^2(\Gamma_\ell)} \, \| \zeta_{k,h_k} \|_{L^2(\Gamma_k)} \\
&\leq& \frac{1}{2} \,
S_{k\ell} \, 
C_{\delta,\ell} \,
\left( \| \zeta_{\ell,h_\ell} \|_{L^2(\Gamma_\ell)}^2 +  \| \zeta_{k,h_k} \|_{L^2(\Gamma_k)}^2 \right) \, .
\end{eqnarray*}
Then,
\begin{equation*}
\boldsymbol{\zeta}^T \, \widetilde{\mathsf{\Sigma}} \, \boldsymbol{\zeta} 
\geq 
\frac{1}{2}
\left(
1 - S_{21} \, C_{\delta,1} + 
1 - S_{12} \, C_{\delta,2} 
\right)
\left(
\| \zeta_{\ell,h_\ell} \|_{L^2(\Gamma_\ell)}^2 +  \| \zeta_{k,h_k} \|_{L^2(\Gamma_k)}^2 \right) \, .
\end{equation*}
Since $S_{k\ell}>1$, it is not possible to guarantee that
\begin{equation*}
1 - S_{k\ell} \, C_{\delta,\ell} \, 
> 0\,,
\end{equation*}
so that, for non-conforming meshes, a positive lower bound for $\zetabv^T \, \widetilde{\mathsf\Sigma} \, \zetabv$ cannot be established. As a consequence, we cannot ensure the positivity of the minimum eigenvalue of the matrix $\widetilde{\mathsf{\Sigma}}_s$. This is confirmed by the numerical results presented in Sect. \ref{sec:numericalAnalysisNonConforming}, Test \#4.

\subsection{Convergence theory for GMRES}\label{appendix:gmres}

For the sake of completeness and clarity, we reformulate here Theorem 2.1 of \cite{beckermann2006} that is used for the analysis of convergence of GMRES iterations in Sect.~\ref{sec:convergenceGMRES}.

\begin{theorem}\label{thm:gmres-beckermann}
    Let $\mathsf A\in\mathbb R^{n\times n}$ be a positive real matrix and $\boldsymbol{r}^{(m)}$ the residual at the $m$th iteration of GMRES of the solution of a linear system with matrix $\mathsf{A}$. 
    Let $\beta\in(0,\frac{\pi}{2})$ be such that
    \begin{equation*}
    cos(\beta)=\frac{\lambda_{min}(\mathsf A_s)}{\|\mathsf A\|_2}\qquad \mbox{and let} \qquad \gamma_\beta:=2\sin\left(\frac{\beta}{4-2\beta/\pi}\right),
    \end{equation*}
    where $\mathsf{A}_s$ is the symmetric part of $\mathsf{A}$.
    Then, for $m \geq 1$, the relative residual at the $m$th iteration of GMRES can be bounded as
    \begin{equation}\label{eq:res_gmres}
    \frac{\|\boldsymbol{r}^{(m)}\|_2}{\|\boldsymbol{r}^{(0)}\|_2}\leq (2+2/\sqrt{3})(2+\gamma_\beta)\gamma_\beta^m.
    \end{equation}
    Moreover, for a given tolerance $\epsilon>0$, the stopping test $\frac{\|\boldsymbol{r}^{(m)}\|_2}{\|\boldsymbol{r}^{(0)}\|_2}\leq \epsilon$ is satisfied provided that
\begin{equation}\label{eq:iteration_estimate}
    m \geq \frac{\log\epsilon -\log((2+2/\sqrt{3})(2+\gamma_\beta))}
    {\log\gamma_\beta}.
\end{equation}
\end{theorem}

\subsection{Practical implementation of ICDD methods}\label{sec:algorithm}

In this appendix, we provide pseudo-codes to show how the ICDD method \eqref{eq:schur_2} and the dual ICDD method \eqref{eq:schur_dual} can be practically implemented in the framework of matrix-free GMRES iterations.

Algorithm \ref{alg:1} constructs all the relevant stiffness matrices and right-hand side vectors, as well as the intergrid matrices $\mathsf T^0_{\ell k}$ and $\mathsf T^\Gamma_{\ell k}$ ($\ell,k=1,2$, $\ell \not= k$). Remember that these intergrid matrices reduce to \eqref{eq:T0_conforming} and \eqref{eq:Tgamma_conforming} in the case of conforming meshes on the overlap.

Algorithm \ref{alg:2} assembles the right-hand side $\boldsymbol{\chi}$ of the Schur complement system \eqref{eq:schur_2}.

Algorithm \ref{alg:5} provides an overview of the implementation of the ICDD method, with Algorithm \ref{alg:4} detailing the steps of a Krylov method and Algorithm \ref{alg:3} explaining how th matrix-vector product ${\mathsf \Sigma}\zetabv$ should be performed for a generic given vector $\zetabv$.

Finally, Algorithms \ref{alg:3dual}--\ref{alg:5dual} are the counterparts of Algorithms \ref{alg:3}--\ref{alg:5} in the context of the dual ICDD method.

\begin{algorithm}[h!]
\caption{Local matrices assembling (in parallel by any available code)} \label{alg:1}
\begin{algorithmic}[1]
\Procedure{Assemble\,}{meshes, data}
\For{$k=1,2$}
\State{set $\ell=3-k$}
\State \emph{Assemble the arrays for the problem in $\Omega_k$:}
\State $\mathsf A_{kk},\, \mathsf A_{k\Gamma_k}$ stiffness matrices
\State $\mathsf f_k$ right-hand side
\State $\mathsf T^0_{\ell k},\ \mathsf T^\Gamma_{\ell k}$ intergrid matrices from $\overline\Omega_k$ to $\Gamma_\ell$
\EndFor
\State \Return{arrays $\mathsf A_{11},\, \mathsf A_{1\Gamma_1},\, \mathsf A_{22},\, \mathsf A_{2\Gamma_2},\, \mathsf f_1,\, \mathsf f_2,\, \mathsf T^0_{21},\ \mathsf T^0_{12},\ \mathsf T^\Gamma_{21},\ \mathsf T^\Gamma_{12}$
\EndProcedure}
\end{algorithmic}
\end{algorithm}	

\begin{algorithm}[h!]
\caption{Computation of the right-hand side $\boldsymbol\chi$ of system \eqref{eq:schur_2}} \label{alg:2}
\begin{algorithmic}[1]
\Procedure{SchurRHS\,}{$\mathsf A_{11},\, \mathsf A_{22},\, \mathsf f_1,\, \mathsf f_2,\, \mathsf T^0_{21},\, \mathsf T^0_{12}$}
\For{$k=1,2$}
\State{set $\ell=3-k$}
\State\emph{Solve the problem in $\Omega_k$ with  homogeneous Dirichlet datum on $\Gamma_k$:}
\State $\mathsf A_{kk} \widehat{\uv}_k =
        \fv_k $
\State\emph{Interpolate } $\widehat{\uv}_k$ \emph{at the nodes of $\Gamma_\ell$:}
\State $\chiv_\ell=\mathsf T^0_{\ell k} \widehat{\uv}_k$ 
\EndFor
\State \Return{
$\chibv=[\chiv_1,\chiv_2]^T$}
\EndProcedure
\end{algorithmic}
\end{algorithm}	

\begin{algorithm}[h!]
\caption{Given $\zetabv= [\zetav_1,\zetav_2]^T$, compute $\mathsf \psibv =\mathsf \Sigma \zetabv$} \label{alg:3}
\begin{algorithmic}[1]
\Procedure{SchurEval\,}{ $\zetabv,\, \mathsf A_{11},\, \mathsf A_{1\Gamma_1},\, \mathsf A_{22},\, \mathsf A_{2\Gamma_2},\, \mathsf T^0_{12},\ \mathsf T^0_{21},\ \mathsf T^\Gamma_{12},\ \mathsf T^\Gamma_{21}$} 
\For{k=1,2}
\State{set $\ell=3-k$}
\State\emph{Solve the problem in $\Omega_k$ with Dirichlet datum $\zetav_k$ on $\Gamma_k$ (all other problem data are zero):}
\State $\mathsf A_{kk} {\uv}_k^0 =
        -\mathsf A_{k\Gamma_k} \zetav_k $
\State $\widetilde{\uv}_k=[{\uv}_k^0,\zetav_k]^T$
\State\emph{Interpolate} $\widetilde{\uv}_k$ \emph{at the nodes of $\Gamma_\ell$:}
        \State $\etav_\ell =\mathsf T^\Gamma_{\ell k}\zetav_k+\mathsf T^0_{\ell k}\uv^0_k$ 
\EndFor
\State \Return{
$\psibv=[\zetav_1-\etav_1,\, \zetav_2-\etav_2]^T$}

\EndProcedure
\end{algorithmic}
\end{algorithm}	

\begin{algorithm}[h!]
\caption{Solve the Schur complement system (\ref{eq:schur_2})} \label{alg:4}
\begin{algorithmic}[1]
\Procedure{SchurSolve\,}{$\chibv,\, \mathsf A_{11},\, \mathsf A_{1\Gamma_1},\, \mathsf A_{22},\, \mathsf A_{2\Gamma_2}, \, \mathsf T^0_{12},\ \mathsf T^0_{21},\ \mathsf T^\Gamma_{12},\ \mathsf T^\Gamma_{21}$ }
\State given $\zetabv^{(0)}=[\zetav_1^{(0)},\, \zetav_2^{(0)}]^T$:
\For{$m=0,\ldots,$ until convergence}
\State \emph{m-th Krylov iteration}
\State $\ldots$
\State \emph{Evaluate} $\psibv^{(m)}=\mathsf \Sigma \zetabv^{(m)}$:
\State $\psibv^{(m)}$ = \Call{SchurEval\,}{$\zetabv^{(m)},\, \mathsf A_{11},\, \mathsf A_{1\Gamma_1},\, \mathsf A_{22},\, \mathsf A_{2\Gamma_2},\, \mathsf T^0_{12},\ \mathsf T^0_{21},\ \mathsf T^\Gamma_{12},\ \mathsf T^\Gamma_{21}$}
\State $\ldots$
\State \emph{End Krylov iteration}
\EndFor
\State \Return{$\psibv^{(m)}$ at convergence}
\EndProcedure
\end{algorithmic}
\end{algorithm}	

\begin{algorithm}[h!]
\caption{ICDD solver} \label{alg:5}
\begin{algorithmic}[1]
\Procedure{ICDDSolver}{meshes, FE, data}
\State\emph{Assemble local matrices using any available code}
\State 
$ [\mathsf A_{11},\, \mathsf A_{1\Gamma_1},\, \mathsf A_{22},\, \mathsf A_{2\Gamma_2},\, \mathsf f_1,\, \mathsf f_2,\, \mathsf T^0_{21},\ \mathsf T^0_{12},\ \mathsf T^\Gamma_{21},\ \mathsf T^\Gamma_{12}] \gets\Call{Assemble}{\text{meshes, FE, data}}$
\State \emph{Compute the right hand side} $\chibv$ of the Schur--complement system
\State $\chibv\gets\Call{SchurRHS}{\mathsf A_{11},\, \mathsf A_{22},\, \mathsf f_1,\, \mathsf f_2,\, \mathsf T^0_{21},\, \mathsf T^0_{12}}$
\State \emph{Solve} $\mathsf \Sigma \lbv = \chibv$ \emph{by a Krylov method}
\State $\lbv\gets\Call{SchurSolve}{\chibv,\, \mathsf A_{11},\, \mathsf A_{1\Gamma_1},\, \mathsf A_{22},\, \mathsf A_{2\Gamma_2}, \, \mathsf T^0_{12},\ \mathsf T^0_{21},\ \mathsf T^\Gamma_{12},\ \mathsf T^\Gamma_{21}}$
\State\emph{Solve the complete local problems:}
\For{$k=1,2$}
\State $\mathsf A_{kk} \uv_k^0  = 
        \fv_k -\mathsf A_{k\Gamma_k} \lv_k$ 
\State $\uv_k=[\uv^0_k,\lv_k]^T$
\EndFor
\State \Return{$\uv_1,\ \uv_2$}
\EndProcedure
\end{algorithmic}
\end{algorithm}

\begin{algorithm}[h!]
\caption{Given $\zetabv= [\zetav_1,\zetav_2]^T$, compute $\mathsf \psibv =\mathsf \Sigma_D \zetabv$} \label{alg:3dual}
\begin{algorithmic}[1]
\Procedure{DualSchurEval\,}{ $\zetabv,\, \mathsf A_{11},\, \mathsf A_{1\Gamma_1},\, \mathsf A_{22},\, \mathsf A_{2\Gamma_2},\, \mathsf T^0_{12},\ \mathsf T^0_{21},\ \mathsf T^\Gamma_{12},\ \mathsf T^\Gamma_{21}$} 
\For{k=1,2}
\State{set $\ell=3-k$}
\State\emph{Solve the problem in $\Omega_k$ with Dirichlet datum $\zetav_k$ on $\Gamma_k$ (all other problem data are zero):}
\State $\mathsf A_k {\uv}_k^0 =
        -\mathsf A_{k\Gamma_k} \zetav_k $
\State $\widetilde{\uv}_k=[{\uv}_k^0,\zetav_k]^T$
\State\emph{Interpolate} $\widetilde{\uv}_k$ \emph{at the nodes of $\Gamma_\ell$:}
        \State $\etav_\ell =\mathsf T^\Gamma_{\ell k}\zetav_k+\mathsf T^0_{\ell k}\uv^0_k$ 
\EndFor
\State \Return{
$\psibv=[\zetav_1+\etav_1,\, \zetav_2+\etav_2]^T$}

\EndProcedure
\end{algorithmic}
\end{algorithm}	

\begin{algorithm}[h!]
\caption{Solve the Schur complement system $\mathsf \Sigma_D\lbv=\boldsymbol\chi_D$} \label{alg:4dual}
\begin{algorithmic}[1]
\Procedure{DualSchurSolve\,}{$\chibv_D,\, \mathsf A_{11},\, \mathsf A_{1\Gamma_1},\, \mathsf A_{22},\, \mathsf A_{2\Gamma_2}, \, \mathsf T^0_{12},\ \mathsf T^0_{21},\ \mathsf T^\Gamma_{12},\ \mathsf T^\Gamma_{21}$ }
\State given $\zetabv^{(0)}=[\zetav_1^{(0)},\, \zetav_2^{(0)}]^T$:
\For{$m=0,\ldots,$ until convergence}
\State \emph{m-th Krylov iteration}
\State $\ldots$
\State \emph{Evaluate} $\psibv^{(m)}=\mathsf \Sigma_D \zetabv^{(m)}$:
\State $\phibv^{(m)}$ = \Call{SchurEval\,}{$\zetabv^{(m)},\, \mathsf A_{11},\, \mathsf A_{1\Gamma_1},\, \mathsf A_{22},\, \mathsf A_{2\Gamma_2},\, \mathsf T^0_{12},\ \mathsf T^0_{21},\ \mathsf T^\Gamma_{12},\ \mathsf T^\Gamma_{21}$}
\State $\psibv^{(m)}$ = \Call{DualSchurEval\,}{$\phibv^{(m)},\, \mathsf A_{11},\, \mathsf A_{1\Gamma_1},\, \mathsf A_{22},\, \mathsf A_{2\Gamma_2},\, \mathsf T^0_{12},\ \mathsf T^0_{21},\ \mathsf T^\Gamma_{12},\ \mathsf T^\Gamma_{21}$}
\State $\ldots$
\State \emph{End Krylov iteration}
\EndFor
\State \Return{$\psibv^{(m)}$ at convergence}
\EndProcedure
\end{algorithmic}
\end{algorithm}	

\begin{algorithm}[h!]
\caption{Dual ICDD solver} \label{alg:5dual}
\begin{algorithmic}[1]
\Procedure{DualICDDSolver}{meshes, FE, data}
\State\emph{Assemble local matrices using any available code}
\State 
$ [\mathsf A_{11},\, \mathsf A_{1\Gamma_1},\, \mathsf A_{22},\, \mathsf A_{2\Gamma_2},\, \mathsf f_1,\, \mathsf f_2,\, \mathsf T^0_{21},\ \mathsf T^0_{12},\ \mathsf T^\Gamma_{21},\ \mathsf T^\Gamma_{12}] \gets\Call{Assemble}{\text{meshes, FE, data}}$
\State \emph{Compute the right hand side} $\chibv_D$ of the Schur--complement system $\mathsf \Sigma_D\lbv=\boldsymbol\chi_D$
\State $\chibv\gets\Call{SchurRHS}{\mathsf A_{11},\, \mathsf A_{22},\, \mathsf f_1,\, \mathsf f_2,\, \mathsf T^0_{21},\, \mathsf T^0_{12}}$
\State $\chibv_D\gets\Call{DualSchurEval}{\chibv,\mathsf A_{11},\, \mathsf A_{1\Gamma_1},\, \mathsf A_{22},\, \mathsf A_{2\Gamma_2},\, \mathsf T^0_{12},\ \mathsf T^0_{21},\ \mathsf T^\Gamma_{12},\ \mathsf T^\Gamma_{21}}$
\State \emph{Solve} $\mathsf\Sigma_D\lbv = \chibv_D$ \emph{by a Krylov method}
\State $\lbv\gets\Call{DualSchurSolve}{\chibv_D,\, \mathsf A_{11},\, \mathsf A_{1\Gamma_1},\, \mathsf A_{22},\, \mathsf A_{2\Gamma_2}, \, \mathsf T^0_{12},\ \mathsf T^0_{21},\ \mathsf T^\Gamma_{12},\ \mathsf T^\Gamma_{21}}$
\State\emph{Solve the complete local problems:}
\For{$k=1,2$}
\State $\mathsf A_{kk} \uv_k^0  = 
        \fv_k -\mathsf A_{k\Gamma_k} \lv_k$ 
\State $\uv_k=[\uv^0_k,\lv_k]^T$
\EndFor
\State \Return{$\uv_1,\ \uv_2$}
\EndProcedure
\end{algorithmic}
\end{algorithm}

\bibliographystyle{siam}
\bibliography{bibl}

\end{document}